\documentclass[12pt]{amsart}%
\usepackage{amsfonts}
\usepackage{amsmath}
\usepackage{amssymb}
\usepackage{graphicx}%
\makeatletter
\@addtoreset{equation}{section}
\makeatother

\newtheorem{theorem}{Theorem}[section]

\newtheorem{corollary}[theorem]{Corollary}

\newtheorem{definition}[theorem]{Definition}

\newtheorem{lemma}[theorem]{Lemma}

\newtheorem{proposition}[theorem]{Proposition}
\newtheorem{remark}[theorem]{Remark}

\makeatletter
\def\@makefnmark{}
\makeatother

\begin{document}
\title[Existence and Nonexistence of Extremals for critical Adams inequalities in $%
%TCIMACRO{\U{211d} }%
%BeginExpansion
\mathbb{R}
%EndExpansion
^{4}$ ]{Existence and Nonexistence of Extremals for critical Adams inequalities in $%
%TCIMACRO{\U{211d} }%
%BeginExpansion
\mathbb{R}
%EndExpansion
^{4}$ and Trudinger-Moser inequalities in $\mathbb{R}
%EndExpansion
^{2}$}

\author{Lu Chen, Guozhen Lu and Maochun Zhu}
\address{School of Mathematics and Statistics, Beijing Institute of Technology, Beijing 100081, P. R. China}
\email{chenlu5818804@163.com}

\address{Department of Mathematics\\
University of Connecticut\\
Storrs, CT 06269, USA}
\email{guozhen.lu@uconn.edu}

\address{Faculty of Science\\
Jiangsu University\\
Zhenjiang, 212013, P. R. China\\}
\email{zhumaochun2006@126.com}

\thanks{The second author was partly supported by a grant from the Simons Foundation. The third author was partly supported by grants from the NNSF of China (No. 11601190) and Natural Science of Jiangsu Province Foundation No. BK20160483.}

\begin{abstract}
Though much work has been done with respect to the existence of extremals of the critical first order Trudinger-Moser inequalities in $W^{1,n}(\mathbb{R}^n)$ and higher order Adams inequalities on finite domain $\Omega\subset \mathbb{R}^n$,
whether there exists an extremal function for the critical higher order Adams inequalities on the entire space $\mathbb{R}^n$ still remains open. The current paper represents the first attempt in this direction. The classical blow-up procedure cannot apply to solving
the existence of critical Adams type inequality because of the absence of the P\'{o}lya-Szeg\"{o}\ type inequality. In this paper, we develop some new ideas and approaches based on a sharp Fourier rearrangement principle (see \cite{Lenzmann}), sharp constants of the higher-order Gagliardo-Nirenberg inequalities and optimal poly-harmonic truncations to study the existence and nonexistence of the maximizers for the Adams inequalities in $\mathbb{R}^4$ of the form
$$   S(\alpha)=\sup_{\|u\|_{H^2}=1}\int_{\mathbb{R}^4}\big(\exp(32\pi^2|u|^2)-1-\alpha|u|^2\big)dx,$$
where $\alpha \in (-\infty, 32\pi^2)$. We establish the existence of the threshold
$\alpha^{\ast}$, where $\alpha^{\ast}\geq \frac{(32\pi^{2})^2B_{2}}{2}$ and $B_2\geq \frac{1}{24\pi^2}$, such that $S\left( \alpha\right) $ is attained if
$32\pi^{2}-\alpha<\alpha^{\ast}$, and is not attained if $32\pi^{2}-\alpha>\alpha^{\ast}$.  This phenomena has not been  observed before  even in
the case of first order Trudinger-Moser inequality. Therefore, we also establish  the existence and non-existence of an extremal function for the Trudinger-Moser  inequality on $\mathbb{R}^2$.
Furthermore, the symmetry of the extremal functions can also be deduced through the Fourier rearrangement principle.

\end{abstract}

\maketitle {\small {\bf Keywords:} Trudiner-Moser inequality, Adams inequality, blow up analysis, extremal function, Sharp Fourier rearrangement principle, sharp constants, threshold. \\

}
\section{Introduction}
Let $\Omega\subseteq%
%TCIMACRO{\U{211d} }%
%BeginExpansion
\mathbb{R}
%EndExpansion
^{n}$ and $W_{0}^{m,p}\left(  \Omega\right)  $ denote the usual Sobolev space
consisting of functions vanishing on boundary $\partial\Omega$ together with
their derivatives of order less than $m-1$, that is, the completion of
$C_{0}^{\infty}\left(  \Omega\right)  $ under the norm%
\[
\left\Vert u\right\Vert _{W^{m,p}\left(  \Omega\right)  }=\left(  \int
_{\Omega}\left(  \left\vert u\right\vert ^{p}+\left\vert \Delta^{m/2}%
u\right\vert ^{p}\right)  dx\right)  ^{\frac{1}{p}}.
\]
If $1<p<n/m$, the classical Sobolev embedding asserts that $W_{0}^{m,p}\left(
\Omega\right)  \hookrightarrow L^{p^{\ast}}\left(  \Omega\right)  $ for
$p^{\ast}=\frac{np}{n-mp}$. However, when $p=n/m$, it is known that
$W_{0}^{m,p}\left(  \Omega\right)  \hookrightarrow L^{\infty}\left(
\Omega\right)  $ does not hold. The borderline case of the optimal Sobolev embedding is
the well-known Trudinger-Moser inequality $\left(  m=1\right)$ (\cite{Mo},
\cite{Tru}) and Adams inequality $\left(  m>1\right)$ (\cite{A}). \\[0.11in]\textbf{Trudinger-Moser inequality.} The Trudinger inequality was established
independently by Yudovi\v{c} \cite{Yu}, Poho\v{z}aev \cite{Po} and Trudinger
\cite{Tru}. In 1971, Trudinger's inequality was sharpened in \cite{Mo} by proving%

\begin{equation}
\underset{\left\Vert \nabla u\right\Vert _{n\left(  \Omega\right)  }\leq
1}{\underset{u\in W_{0}^{1,n}\left(  \Omega\right)  }{\sup}}\int_{\Omega
}e^{\alpha\left\vert u\right\vert ^{\frac{n}{n-1}}}dx<\infty\text{ iff }%
\alpha\leq\alpha_{n}=n\omega_{n-1}^{\frac{1}{n-1}}, \label{moser-tru}%
\end{equation}
for any bounded domain $\Omega\subset%
%TCIMACRO{\U{211d} }%
%BeginExpansion
\mathbb{R}
%EndExpansion
^{n}$, where $\omega_{n-1}$ denotes the $n-1$ dimensional surface measure of
the unit ball in $%
%TCIMACRO{\U{211d} }%
%BeginExpansion
\mathbb{R}
%EndExpansion
^{n}$.

When the volume of $\Omega$ is infinite, there are several extensions of the
Trudinger-Moser inequality, see Cao \cite{cao} in the case $n=2$ and for
any dimension ($n\geq2$) by do \'{O} \cite{J.M. do1}.   Adachi-Tanaka
\cite{Adachi-Tanaka} obtained a sharp Trudinger-Moser on $%
%TCIMACRO{\U{211d} }%
%BeginExpansion
\mathbb{R}
%EndExpansion
^{n}$. Unlike in the inequality (\ref{moser-tru}), the result of \cite{Adachi-Tanaka}%
\ has a subcritical form, that is $\alpha<\alpha_{n}$. In \cite{ruf} and
\cite{liruf}, Li and Ruf showed that the exponent $\alpha_{n}$ becomes
admissible if the Dirichlet norm $\int_{\mathbb{R}^{n}}\left\vert \nabla
u\right\vert ^{2}dx$ is replaced by Sobolev norm $\int_{\mathbb{R}^{n}}\left(
\left\vert u\right\vert ^{2}+\left\vert \nabla u\right\vert ^{2}\right)  dx$,
more precisely, they proved that%

\[
\underset{\int_{\mathbb{R}^{n}}\left(  \left\vert u\right\vert ^{n}+\left\vert
\nabla u\right\vert ^{n}\right)  dx\leq1}{\underset{u\in W^{1,n}\left(
%TCIMACRO{\U{211d} }%
%BeginExpansion
\mathbb{R}
%EndExpansion
^{n}\right)  }{\sup}}\int_{%
%TCIMACRO{\U{211d} }%
%BeginExpansion
\mathbb{R}
%EndExpansion
^{2}}\Phi_{n}\left(  \alpha\left\vert u\right\vert ^{\frac{n}{n-1}}\right)
dx<+\infty,\text{ iff }\alpha\leq\alpha_{n},
\]
where $\Phi_{n}\left(  t\right)  =e^{t}-\underset{j=0}{\overset{N-2}{\sum}%
}\frac{t^{j}}{j!}$.

We should note that all the earlier proofs of both critical and subcritical Trudinger-Moser inequalities rely on
the P\'{o}lya-Szeg\"{o} symmetrization argument which is not available in many other non-Euclidean settings.
Lam and Lu in \cite{LamLu1} developed a symmetrization-free argument using the level sets of the functions under consideration and derive critical Trudinger-Moser inequalities on the Heisenberg groups from local inequalities obtained in \cite{cohn-lu} to global ones (see also \cite{LaLu4}, \cite{li-lu-zhu}). For such an argument in the subcritical case,  see   \cite{LamLuTang}. These also give an alternative proof of both critical and subcritical Trudinger-Moser in the Euclidean space $\mathbb{R}^n$.

\textbf{Existence of extremals for Trudinger-Moser
inequality.\ }A classical problem related to Trudinger-Moser inequalities
is  to investigate the existence of extremal functions.\ The first
proof of the existence of extremals for Trudinger-Moser inequality
(\ref{moser-tru}) was given by Carleson and Chang in their celebrated work
\cite{c-c} when the finite domain is a ball in $%
%TCIMACRO{\U{211d} }%
%BeginExpansion
\mathbb{R}
%EndExpansion
^{n}$. After that, the existence of extremals was proved for any bounded
domains\ in \cite{Flucher} and \cite{lin} in $\mathbb{R}^n$. More related results can be found
in several works, see e.g. Y. X. Li (\cite{Li2}, \cite{Li 1}, \cite{Li3}) for existence of
extremals on compact Riemannian manifold, and  Li and Ruf (\cite{ruf},\cite{liruf}) for existence of extremals on
unbounded domain. Malchiodi and Martinazzi \cite{Malchiodi1} further investigated the
blow-up of a sequence of critical points of the Trudinger-Moser functionals on the planar
disk.\\[0.11in]\textbf{Adams inequality on bounded domains}. In
1988, Adams~\cite{A} extended the Trudinger-Moser inequality (\ref{moser-tru})
to the higher order space $W_{0}^{m,\frac{n}{m}}(\Omega)$ and
obtained
\begin{equation}
\underset{\Vert\Delta^{\frac{m}{2}}u\Vert_{\frac{n}{m}}\leq{1}}{\underset
{u\in{W_{0}^{m,\frac{n}{m}}(\Omega)}}{\sup}}\int_{\Omega}\exp(\beta
|u(x)|^{\frac{n}{n-m}})dx\left\{
\begin{array}
[c]{l}%
\leq c\left\vert \Omega\right\vert \text{ if }\beta\leq{\beta(n,m),}\\
=+\infty\text{ \ if }\beta>{\beta(n,m),}%
\end{array}
\right.  \label{adams}%
\end{equation}
\textit{where}%
\[
\Delta^{\frac{m}{2}}u=%
%TCIMACRO{\QDATOPD{\{}{.}{\Delta^{l}u\text{ }\ \ \ \ \ \ \text{is }%
%m=2l,l\in\U{2115} }{\nabla\Delta^{l}u\text{, is }m=2l+1,l\in\U{2115} }}%
%BeginExpansion
\genfrac{\{}{.}{0pt}{0}{\Delta^{l}u\text{ }\ \ \ \ \ \ \text{is }%
m=2l,l\in\mathbb{N} }{\nabla\Delta^{l}u\text{, is }m=2l+1,l\in\mathbb{N} }%
%EndExpansion
\]
and
\[
\beta(n,m)=%
\begin{cases}
\frac{n}{\omega_{n-1}}[\frac{\pi^{n/2}2^{m}\Gamma(\frac{m+1}{2})}{\Gamma
(\frac{n-m+1}{2})}]^{\frac{n}{n-m}},\text{\textit{when}}~m~\text{\textit{is}
\textit{odd.}}\\
\frac{n}{\omega_{n-1}}[\frac{\pi^{n/2}2^{m}\Gamma(\frac{m}{2})}{\Gamma
(\frac{n-m}{2})}]^{\frac{n}{n-m}},\text{\textit{when}}~m~\text{\textit{is}
\textit{even.}}%
\end{cases}
\]
\textit{ }

Later, Tarsi \cite{Tar} proved that the Adams
inequality (\ref{adams}) also holds for a larger class of Sobolev functions,
i.e. the functions with homogeneous Navier boundary condition:%

\[
W_{N}^{m,\frac{n}{m}}\left(  \Omega\right)  =\left\{  u\in W^{m,\frac{n}{m}%
}\left(  \Omega\right)  ,s.t.\text{ }\Delta^{j}u=0\text{ on }\partial
\Omega\text{ for }0\leq j\leq\left[  \left(  m-1\right)  /2\right]  \right\}
.
\]
\textbf{Adams inequality on the entire Euclidean space }$\mathbb{R}^{n}$.\ In
1995, Ozawa~\cite{O} obtained  the Adams inequality in Sobolev
space $W^{m,\frac{n}{m}}(\mathbb{R}^{n})$ on the entire Euclidean space
$\mathbb{R}^{n}$ by using the restriction $\Vert\Delta^{\frac{m}{2}}%
u\Vert_{\frac{n}{m}}\leq{1}$. However, with the argument in \cite{O,KSW}, one
cannot obtain the best possible exponent $\beta$ for this type of
inequality.  Sharp Adams inequalities on even dimensional space $\mathbb{R}^n$ was proved by Ruf and Sani ~\cite{RS}
under
  the
stronger constraint
$$\{
u\in W^{{m},\frac{{n}}{m}}| \Vert(I-\Delta)^{\frac{m}{2}}\Vert_{\frac{n}{m}%
}\leq1
\},$$
when the order of derivatives  $m$ is an even integer. While the order of derivatives $m$ is odd, the inequality was established by
Lam and Lu ~\cite{LaLu6}. Moreover,
the following Adams inequality on
  Sobolev spaces $W^{\gamma, \frac{n}{\gamma}}(\mathbb{R}^{n})$
of arbitrary positive fractional order $\gamma<n$ was established by Lam and Lu using a rearrangement-free argument ~\cite{LaLu4}.

\begin{theorem}
\label{frac2}Let $0<\gamma<n$ be an arbitrary real positive number,
$p=\frac{n}{\gamma}$ and $\tau>0$. There holds%
\[
\underset{u\in W^{\gamma,p}\left(
%TCIMACRO{\U{211d} }%
%BeginExpansion
\mathbb{R}
%EndExpansion
^{n}\right)  ,\left\Vert \left(  \tau I-\Delta\right)  ^{\frac{\gamma}{2}%
}u\right\Vert _{p}\leq1}{\sup}\int_{%
%TCIMACRO{\U{211d} }%
%BeginExpansion
\mathbb{R}
%EndExpansion
^{n}}\phi\left(  \beta_{0}\left(  n,\gamma\right)  \left\vert u\right\vert
^{p^{\prime}}\right)  dx<\infty
\]
where
\begin{align*}
\phi(t)  &  =e^{t}-%
%TCIMACRO{\dsum \limits_{j=0}^{j_{p}-2}}%
%BeginExpansion
{\displaystyle\sum\limits_{j=0}^{j_{p}-2}}
%EndExpansion
\frac{t^{j}}{j!},\\
j_{p}  &  =\min\left\{  j\in%
%TCIMACRO{\U{2115} }%
%BeginExpansion
\mathbb{N}
%EndExpansion
:j\geq p\right\}  \geq p,
\end{align*}
and
\begin{align*}
p^{\prime}  &  =\frac{p}{p-1},\\
\beta_{0}\left(  n,\gamma\right)   &  =\frac{n}{\omega_{n-1}}\left[  \frac
{\pi^{n/2}2^{\gamma}\Gamma\left(  \gamma/2\right)  }{\Gamma\left(
\frac{n-\gamma}{2}\right)  }\right]  ^{p^{\prime}}.
\end{align*}
Furthermore this inequality is sharp, i.e., if $\beta_{0}\left(
n,\gamma\right)  $ is replaced by any $\beta>\beta_{0}\left(  n,\gamma\right)
$, then the supremum is infinite.
\end{theorem}

 The following Adams inequality was established  in \cite{LaLu4}  for $m=2$ and
subsequently in \cite{Fontana} for  $m>2$:
\begin{equation}
\underset{\Vert\Delta^{\frac{m}{2}}u\Vert_{\frac{n}{m}}^{\frac{n}{m}}+\Vert
u\Vert_{\frac{n}{m}}^{\frac{n}{m}}\leq{1}}{\underset{u\in{W^{m,\frac{n}{m}%
}(\mathbb{R}^{n})}}{\sup}}\int_{{\mathbb{R}^{n}}}\Phi_{n,m}(\beta|u(x)|^{\frac{n}{n-m}})dx\left\{
\begin{array}
[c]{l}%
\leq C_{m,n}\text{ if }\beta\leq{\beta(n,m),}\\
=+\infty\text{ \ if }\beta>{\beta(n,m),}%
\end{array}
\right.  \label{Adams entire space}%
\end{equation}
where $$\Phi_{n,m}(t)=e^t-\sum_{j=0}^{j_{\frac{n}{m}}-2}\frac{t^j}{j!}, \ j_{\frac{n}{m}}=\min \{j \in \mathbb{N}:j\geq \frac{n}{m}\}.$$

We mention that there are sharpened Trudinger-Moser and Adams inequalities with exact growth in $\mathbb{R}^n$.
In 2011, Ibrahim et al \cite{IMN} discovered a new kind of
Trudinger-Moser inequality on $%
%TCIMACRO{\U{211d} }%
%BeginExpansion
\mathbb{R}
%EndExpansion
^{2}$--the Trudinger-Moser inequality\ with the exact growth condition:
\begin{equation}
\underset{\int_{\mathbb{R}^{4}}|\nabla u|^{2}dx\leq1}{\underset{u\in
H^{1}\left(
%TCIMACRO{\U{211d} }%
%BeginExpansion
\mathbb{R}
%EndExpansion
^{2}\right)  }{\sup}}\int_{\mathbb{R}^{2}}\frac{\exp(4\pi|u|^{2}%
)-1}{(1+|u|)^{p}}dx\leq C_{p}\int_{\mathbb{R}^{2}}|u|^{2}dx\ \text{iff }%
p\geq2. \label{exact}%
\end{equation}
Later, (\ref{exact}) was extended to the general case $n\geq3$\ by Masmoudi
and Sani \cite{MS2} (see \cite{lamluzhang}) for more general form) and to the
framework of hyperbolic space by Lu and Tang in \cite{LuTa2}.
\vskip0.1cm

The second order Adams' inequality\ with the exact growth condition was obtained by
Masmoudi and Sani \cite{MS} in dimension $4$: \ \
\begin{equation}
\underset{\int_{\mathbb{R}^{4}}|\Delta v|^{2}dx\leq1}{\underset{u\in
H^{2}\left(
%TCIMACRO{\U{211d} }%
%BeginExpansion
\mathbb{R}
%EndExpansion
^{4}\right)  }{\sup}}\int_{\mathbb{R}^{4}}\frac{\exp(32\pi^{2}|v|^{2}%
)-1}{(1+|v|)^{p}}dx\leq C_{p}\int_{\mathbb{R}^{4}}|v|^{2}dx\ \text{iff }%
p\geq2, \label{int1}%
\end{equation}
and then established in any dimension $n\geq3$ by Lu et al in \cite{LuTZ} (see
\cite{MS1} for higher order case).

\textbf{Existence
of extremals for Adams inequality.} The first result of the existence of
extremals of Adams' inequality (\ref{adams}) on bounded domain $\Omega\subset \mathbb{R}^n$ was obtained by Lu and Yang in
\cite{lu-yang 1}. We note that the work of Carleson and Chang was based on the
rearrangement argument to reduce the problem to the one-dimensional case.
However, the symmetrization technique cannot be used for the Adams inequality,
since there is no  corresponding P\'{o}lya-Szeg\"{o}\ type inequality in
the higher order case. In \cite{lu-yang 1}, the authors applied the
capacity-type estimates and the Pohozaev identity to obtain the existence of extremals for bounded
domains in the case $n=4$ and $m=2$. Recently, DelaTorre and Mancini
\cite{Dela} extended the results of \cite{lu-yang 1} to arbitrary even dimension.

\subsection{The main results and Outline of the paper}

An interesting and intriguing  question is whether the Adams inequality on any unbounded domain
has an extremal. As far as we know, nothing is known at the present. In
this work, we are devoted\ to studying this kind of problem for the special case
$n=4$ and $m=2$.

Setting%
\[
S\left(  \alpha\right)  =\underset{u\in H}{\sup}\int_{%
%TCIMACRO{\U{211d} }%
%BeginExpansion
\mathbb{R}
%EndExpansion
^{4}}\left(  \exp\left(  32\pi^{2}\left\vert u\right\vert ^{2}\right)
-1-\alpha\left\vert u\right\vert ^{2}\right)  dx,
\]
and%

\begin{equation}
B_{2}=\sup_{u\in H^{2}(\mathbb{R}^{4})}\frac{\Vert v\Vert_{4}^{4}}{\Vert\Delta
v\Vert_{2}^{2}\Vert v\Vert_{2}^{2}}, \label{GNBN}%
\end{equation}
where $\alpha\in%
%TCIMACRO{\U{211d} }%
%BeginExpansion
(-\infty, 32\pi^2)
%EndExpansion
$, and%
\[
H:=\left\{  \left.  u\in H^{2}\left(
%TCIMACRO{\U{211d} }%
%BeginExpansion
\mathbb{R}
%EndExpansion
^{4}\right)  \right\vert \left\Vert u\right\Vert _{H^{2}(\mathbb{R}^{4}%
)}=\big(\int_{\mathbb{R}^{4}}|u|^{2}+|\Delta u|^{2}dx\big)^{\frac{1}{2}%
}=1\right\}  .
\]

First, we can prove the following result.

\begin{theorem}\label{maintheorem}
\label{attain} If $32\pi^{2}-\alpha<\frac{(32\pi^{2})^2B_{2}}{2}$, then
$S(\alpha)$ has a radially symmetric extremal function.
\end{theorem}

Naturally, one may ask whether extremal functions of critical Adams
inequalities must be radially symmetric. Recently, Lenzmann and Sork \cite{Lenzmann} introduced the Fourier-rearrangement inequalities. Though they did not prove the existence of the Adams inequality on the entire space,  they observed that every  possible maximizer of  $S(\alpha)$, if exists,  must be radially symmetric and real valued (up to translation and constant phase). In fact, assume that $u$ is a maximizer
for $S(\alpha)$, and define $u^{\sharp}$ by $u^{\sharp}=\mathcal{F}^{-1}\{(\mathcal{F}(u))^{\ast
}\}$, where $\mathcal{F}$ denotes the Fourier transform on $\mathbb{R}^{4}$ (with its
inverse $\mathcal{F}^{-1}$) and $u^{\ast}$ stands for the Schwarz symmetrization of $u$.
We easily see that $u^{\sharp}$ is also a maximizer for $S(\alpha)$ with
$\Vert\Delta u^{\sharp}\Vert_{2}=\Vert\Delta u\Vert_{2}$. Using the property
of the Fourier rearrangement from \cite{Lenzmann}, we conclude that
\[
u(x)=e^{i\alpha}u^{\sharp}(x-x_{0})\ \ \mbox{for any}\ x\in\mathbb{R}^{4}%
\]
with some constants $\alpha\in\mathbb{R}$ and $x_{0}\in\mathbb{R}^{4}$. That
is to say that $u$ is radially symmetric and real valued up to translation and
constant phase. Therefore, combining our Theorem \ref{attain} with Lenzmann and Sork's result, we obtain that
\begin{corollary}\label{cor}
 If $32\pi^{2}-\alpha<\frac{(32\pi^{2})^2B_{2}}{2}$, the extremals of the Adams inequalities must be  radially symmetric and real valued (up to translation and constant phase).
\end{corollary}

The method developed in this paper on the existence of the extremals of the critical Adams inequality also gives new insight on the existence of extremal functions for the first order critical Trudinger-Moser inequality on the entire space.
By adapting the same method
 as used in the proof of Theorem \ref{attain}, we can also obtain a similar existence result. More precisely,
  if we define $$\tilde{S}(\alpha)=\sup_{\|u\|_{H^1}=1}\int_{\mathbb{R}^2}\big(\exp(4\pi|u|^2)-1-\alpha|u|^2\big)dx,$$
then there exists an extremal function for  $\tilde{S}(\alpha)$ when $4\pi-\alpha<\frac{(4\pi)^2B_{1}}{2}$,
where $$B_1=\sup_{u\in H^{1}(\mathbb{R}^{2})}\frac{\Vert v\Vert_{4}^{4}}{\Vert\nabla
v\Vert_{2}^{2}\Vert v\Vert_{2}^{2}}.$$
According to the result in \cite{Wein}, we know $B_1>\frac{1}{2\pi}$ and $\frac{(4\pi)^2B_{1}}{2}>4\pi$. Therefore, even if we slightly enlarge the coefficient of the first term of the critical Trudinger-Moser functional $\int_{\mathbb{R}^2}\big(\exp(4\pi|u|^2)-1\big)dx$, the Trudinger-Moser inequality still has an extremal function. Then we can deduce the following stronger existence result than currently known in the literature.

\begin{theorem}\label{addthm2}
There exists $\alpha_0>0$ such that for any $\beta \in (0, \alpha_0)$, the critical Trudinger-Moser inequality $\sup_{\|u\|_{H^1}=1}\int_{\mathbb{R}^2}\big(\exp(4\pi|u|^2)-1+\beta|u|^2\big)dx$ has an extremal function. Furthermore, all extremals of the critical Trudinger-Moser inequalities must be non-negative, radially symmetric and real valued up to translation and
constant phase.

\end{theorem}

Since the method to prove Theorem \ref{addthm2} is inspired by  and similar to that of proving Theorem \ref{maintheorem} and Corollary \ref{cor}
for the critical  Adams inequalities on the entire space $\mathbb{R}^4$, we will be sketchy and only give the outline
of proofs in Section \ref{section6}.

Though the general strategy we use here is the  blow up analysis, it is considerably more difficult to deal with than the situation on  bounded domains. The failure of the P\'olya-Szeg\"{o} inequality for the higher order derivatives will not allow us to take care of the maximizing sequence as in the first order Trudinger-Moser inequality on finite balls. To overcome this difficulty,
we first apply the method based on the Fourier rearrangement (see
\cite{Lenzmann}) to obtain the existence of radially symmetric extremals for
the subcritical Adams functional on the entire space. \ Then, we take a sequence $\beta
_{k}\rightarrow32\pi^{2}$ and find a radially symmetric maximizing sequence
$u_{k}\in H^{2}\left(
%TCIMACRO{\U{211d} }%
%BeginExpansion
\mathbb{R}
%EndExpansion
^{4}\right)  $ for $S\left(  \alpha\right)  $. If $u_{k}$ is bounded in
$L^{\infty}\left(
%TCIMACRO{\U{211d} }%
%BeginExpansion
\mathbb{R}
%EndExpansion
^{4}\right)  $, i.e. $c_{k}:=\underset{x\in%
%TCIMACRO{\U{211d} }%
%BeginExpansion
\mathbb{R}
%EndExpansion
^{4}}{\max}\left\vert u_{k}\right\vert <\infty$, we can easily show that
$u_{k}$ converges to a function $u\,$in\ $H^{2}\left(
%TCIMACRO{\U{211d} }%
%BeginExpansion
\mathbb{R}
%EndExpansion
^{4}\right)  $ by\ the standard elliptic estimates. If $c_{k}\rightarrow
+\infty$, i.e. the blow up arises, we apply the blow up analysis method to
analyze the asymptotic behavior of $u_{k}$ near and far away from the origin,
which is the blowing up point, and we are able to derive an upper bound for the Adams
functional:
\[
S\left(  \alpha\right)  \leq\frac{\pi^{2}}{6}\exp\left(  \frac{5}{3}+32\pi
^{2}A\right)  ,
\]
where $A$ is the value at $0$ of\ the trace of the regular part of the Green
function $G$ for the operator $\Delta^{2}+1$. Finally, we construct a function
sequence to show that the upper bound can actually be surpassed, this
implies that the concentration phenomenon will not happen.\ \vskip0.2cm

At first sight, this type of approach may seem to be a straightforward
generalization of the existing theories. However, this is not the case. Several substantial difficulties exist and some serious
subtlety arises. We are going to describe some of them below.

\vskip0.1cm

First of all, unlike the case on a bounded domain, in order to establish the
existence of a maximizer of the subcritical Adams functional in the entire space $%
%TCIMACRO{\U{211d} }%
%BeginExpansion
\mathbb{R}
%EndExpansion
^{4}$, we need to avoid the lack of compactness. In this case, concentration
phenomenon does not occur and vanishing phenomenon is the issue due to the
unboundedness of the domain.\ \ For this reason, we will impose an extra
assumption on $\alpha$ such as $32\pi^{2}-\alpha<\frac{(32\pi^{2})^2B_{2}}{2}$ and adapt
the argument used in \cite{Ishiwata} to rule out the vanishing phenomenon. \vskip0.1cm

Secondly, when we try to analyze the asymptotic behavior of $u_{k}$, a crucial
step is to prove a local estimate for $\Delta u_{k}$:
\begin{equation}
c_{k}\int_{B_{Rr_{k}}}|\Delta u_{k}|dx\leq C(Rr_{k})^{2}. \label{es for gra}%
\end{equation}
When $\Omega$ is a bounded domain, (\ref{es for gra}) can be proved by the
following estimate
\[
\int_{\Omega}|\Delta\left(  u_{k}^{2}\right)  |dx<c
\]
and a representation formula (see \cite{Mar}). However,
because of the
unboundedness of the domain, the argument in \cite{Mar} cannot be directly
applied in our setting. In order to obtain the estimate (\ref{es for gra}), we
will try to truncate $u_{k}$ and adapt the approach in \cite{Mar} to our
situation. But as we know, in the second order Sobolev space $H^{2}\left(
\mathbb{R}^{4}\right)$, one cannot truncate $u_{k}$ linearly. To overcome
this difficulty, we will apply the bi-harmonic type of truncation as
used in \cite{Dela}. We remark that this kind of truncations have many nice
properties: on the one hand, they preserve the high order regularity, on the
other hand, their behaviors are very similar to the constant in a ball. Nevertheless,
we must point out that it is necessary for us to find an optimal bi-harmonic truncation
in our case.
\vskip0.1cm

When we try to obtain the upper bound of the concentration compactness of the Adams inequality on the entire
Euclidean space\textbf{ }$\mathbb{R}^{4}$, firstly, one needs to know the specific value
of the upper bound for any blow up function sequences in $H_{0}^{2}\left(
B_{R}\right)  $, but this value cannot be directly derived from the result of
Lu and Yang \cite{lu-yang 1} in the case of the finite domain. We will show that the value of that upper bound
is $\frac{1}{3}\left\vert B_{R}\right\vert \exp\left(  -\frac{1}{3}\right)  $
by solving the corresponding ODE's. Secondly, in our situation we cannot truncate $u_{k}$ directly by the linear truncation as  in \cite{liruf}. Therefore,  we have to construct some polynomial truncation functions to preserve some regularity on the boundary of the balls. In the calculation, we find that the polynomial truncation functions will bring some extra energy, which will enlarge the estimate of the upper bound for the normalized concentration sequence of the Adams inequality, such that the upper bound may not be surpassed by  any test function sequence. In order to address this problem, we will construct the ��optimal�� polynomial truncation functions which generate the smallest energy, such that the ��exact�� upper bound of Adams functional for normalized concentration sequence obtained can be surpassed by the some test function. It should be noted that elliptic estimate of the optimal polynomial truncation is far from the exact upper bound of the concentration compactness sequence, we need the precise expression of polynomial truncation without any error estimate. Therefore, we also need the quantitative estimate with  respect to the upper bound of the concentration compactness sequence.

\vskip0.1cm

Finally, although we can show that $u_{k}$ is radially symmetric by the
Fourier rearrangement argument, this function sequence is not necessarily positive,
which makes\ the\ proof of the existence of a maximizer more complicated.

Our second result is as follows.

\begin{theorem}
\label{nonatain}There exists some constant $\alpha^{\ast\ast}>0$ such that,
when $32\pi^{2}-\alpha>\alpha^{\ast\ast}$, $S$ is not attained.
\end{theorem}

The proof of this theorem is based on the precise estimates for the upper
bounds of the best constants of higher order Gagliardo-Nirenberg inequalities.
To calculate the best constants, we will exploit the method of Beckner in
\cite{Beckner}. We stress that the upper bound derived here is  sharp and
can be of independent interest.\ Indeed, we take $B_{2}$ for example. On the
one hand, by calculating the number associate to the trial function $\left(
1+\left\vert x\right\vert \right)  ^{-\gamma}$ and letting $\gamma
\rightarrow+\infty$, one can found that $\frac{1}{24\pi^{2}}$ is a lower bound
of $B_{2}$. On the other hand, by the upper bound formula we have $B_{2}%
<\frac{32}{729\pi^{2}}$ (see Appendix). Then, we get
\[
\frac{1}{24\pi^{2}}\left(  \sim4.2217\times10^{-3}\right)  \leq B_{2}%
<\frac{32}{729\pi^{2}}\left(  \sim4.4476\times10^{-3}\right)  ,
\]
which indicates that our estimates are quite precise.

Define
\[
\alpha^{\ast}=\sup\left\{  \left.  \left(  32\pi^{2}-\alpha\right)
\right\vert S\left(  \alpha\right)  \text{ is attained}\right\}  .
\]
Based on Theorem \ref{attain} and Theorem \ref{nonatain}, we have $\frac{(32\pi^{2})^2B_{2}}{2}\leq \alpha^{*}\leq \alpha^{\ast\ast}<+\infty$.
Furthermore, we can obtain the following surprising result.

\begin{theorem}
\label{fina}When $32\pi^{2}-\alpha<\alpha^{\ast}$, then $S\left(  \alpha\right)< \alpha^{\ast}$ and $S(\alpha)$ can be attained, while when $32\pi^{2}-\alpha>\alpha^{\ast}$, $S\left(  \alpha\right)=32\pi^{2}-\alpha$, and $S(\alpha)$ is not attained.
\end{theorem}

Similarly, we can also obtain the following existence and nonexistence of maximizers for the Trudinger-Moser inequality in $\mathbb{R}^2$.
\begin{theorem}\label{addthm3}
When $4\pi-\alpha<\beta^{\ast}$, then $\tilde{S}\left(  \alpha\right)< \beta^{\ast}$ and $\tilde{S}(\alpha)$ can be attained; When $4\pi-\alpha>\beta^{\ast}$, then $\tilde{S}\left(  \alpha\right)=4\pi-\alpha$ and $\tilde{S}(\alpha)$ is not attained, where $\beta^{*}$
is defined as \[
\beta^{\ast}=\sup\left\{  \left.  \left(  4\pi-\alpha\right)
\right\vert S\left(  \alpha\right)  \text{ is attained}\right\} .
\] and $\beta^{*}\geq \frac{(4\pi)^2B_{1}}{2}>4\pi$.
\end{theorem}

\begin{remark}
The proof of Theorems \ref{addthm2} and \ref{addthm3} concerning the existence and nonexistence of maximizers for  the Trudinger-Moser inequality in $\mathbb{R}^2$ is similar in spirit  to that of the existence and nonexistence of maximizers for the Adams inequality in $\mathbb{R}^4$. Therefore,  we have chosen to give the sketch of the proof in Section 6.
\end{remark}

It is important to point out that Theorem \ref{fina} provides a further
insight on the existence or nonexistence of extremals for Adams inequality on
the whole space. From the proof of Theorem \ref{attain}, we know that the supremum of Adams functional is larger than the upper bound of concentration-compactness
sequences. Hence \textit{whether }$S\left(  \alpha\right)  $\textit{ is
attained\ highly depends on the vanishing phenomena, whose energy level is
determined only by the coefficient of the first term of }$S\left(
\alpha\right)  $\textit{. Thus, changing the coefficients of finite terms
(especially the first term) of }$S\left(  \alpha\right)  $\textit{, will
not\ affect on the validity of the Adams inequality and the upper bound of
concentration sequences, but will change the existence or nonexistence of
extremals. }It seems that this phenomenon  has not been noticed before, even in
the case of Trudinger-Moser inequality.
 \vskip0.1cm

Once the existence and nonexistence of extremals for Adams inequality in the
special case $m=2,n=4$ are established, a natural, but nontrivial extension is
to establish similar results for any arbitrary $m\geq2$. The proof for this
extension has some extra difficulties to overcome and we have decided to
address this problem in a forthcoming paper. \vskip 0.1cm

The following remarks are in order. The problem considered here was initially suggested by the second author to the first and third authors several years ago. We have worked together on the problem since then. This is a revised version of the manuscript posted as arXiv:1812.00413v1 by the first
and third authors.
A mistake in that version was found recently. In particular, the argument of obtaining the optimal upper bound
of the concentration-compactness sequence was incorrect in that version. As we pointed out earlier in the introduction,
in the derivation of the optimal upper bound
of the concentration-compactness sequence,  the polynomial truncation
functions will add some extra energy which will enlarge the estimate of the upper bound
for the normalized concentration sequence of the Adams inequality. Thus, the upper
bound may be too large to be surpassed by the any test function sequence. In the old version,
we used the elliptic estimates   to get the
upper bound of the concentration compactness sequence which  were  far from being the exact bound.
In this new version, we apply the optimal and precise  expression
of polynomial truncation without any error estimate to address  this
issue and thus derive the optimal upper bound. In this new version, we have also obtained
the existence of extremal functions for the critical Trudinger-Moser supremum $\sup_{\|u\|_{H^1}=1}\int_{\mathbb{R}^2}\big(\exp(4\pi|u|^2)-1+\beta|u|^2\big)dx$ for a
perturbation term of $\beta |u|^2$
which gives more information than those known in the literature.

\medskip
We finally remark that there is some recent development on the existence and nonexistence of extremal functions
for subcritical Trudinger-Moser inequalities established by Lam, Lu and Zhang \cite{LLZ1}  using the equivalence and identities between the supremums for the critical and subcritical Trudinger-Moser inequalities in $\mathbb{R}^n$ established by the same authors in \cite{LLZ3}. For subcritical Adams inequalities on the entire space, the existence of extremal functions has been proved by Chen, Lu and Zhang \cite{ChenLuZhang}.

This paper is organized as follows. Section 2 is devoted to proving existence
of radially symmetric maximizing sequence for the critical Adams functional;
in Section 3, we will analyze the asymptotic behavior of the maximizing
sequence, and derive an upper bound for the Adams' inequality when the blowing
up arises; In Section 4, we prove the existence of extremals (Theorem
\ref{attain}) by constructing a proper test function sequence. In Section 5,
we give the proof for Theorem \ref{nonatain} and Theorem \ref{fina} by
estimating the best constant of higher order Gagliardo-Nirenberg inequalities.
In Section 6, we establish the existence and nonexistence of extremal functions
for the Trudinger-Moser inequality in $\mathbb{R}^2$.
For the convenience of the reader, the work of estimating the best constants
of Gagliardo-Nirenberg inequalities and some known results concerning elliptic
estimates for operator $\Delta^{2}$ are arranged in the Appendix. \vskip0.1cm

Throughout this paper, the letter $c$ always denotes some positive constant
which may vary from line to line.
\section{\bigskip The maximizing sequence for critical Adams functional}

\subsection{Existence of extremals for the subcritical Adams functionals}

In this section, we will establish the existence of extremal functions for
subcritical Adams functional. Set
\[
I_{\beta}^{\alpha}\left(  u\right)  =\int_{B_{R}}\left(  \exp(\beta
|u|^{2})-1-\alpha|u|^{2}\right)  dx.
\]

\begin{lemma}
\label{lem2} For any $0<\beta<32\pi^{2}$, there exists a radially symmetric
extremal function $u\in H$ such that%
\[
I_{\beta}^{\alpha}\left(  u\right)  =\underset{u\in H}{\sup}I_{\beta}^{\alpha
}\left(  u\right)  ,
\]
provided $\beta-\frac{\beta^{2}B_{2}}{2}<\alpha<\beta$.
\end{lemma}

\begin{remark}
It follows from that Lemma \ref{sharp1} in the Appendix that $B_{2}<\frac
{1}{16\pi^{2}}$, which leads to that $\beta-\frac{\beta^{2}B_{2}}{2}>0$.
\end{remark}

\begin{proof}
Define $u^{\sharp}$ by $u^{\sharp}=\mathcal{F}^{-1}\{(\mathcal{F}(u))^{\ast}\}$, where $\mathcal{F\mathcal{}}$ denotes
the Fourier transform on $\mathbb{R}^{4}$ (with its inverse $\mathcal{F}^{-1}$) and
$f^{\ast}$ stands for the Schwarz symmetrization of $f$. Using the property of
the Fourier rearrangement from \cite{Lenzmann}, one can derive that
\[
\Vert\Delta u^{\sharp}\Vert_{2}\leq\Vert\Delta u\Vert_{2},\Vert u^{\sharp
}\Vert_{2}=\Vert u\Vert_{2}, \Vert u^{\sharp
}\Vert_{q}\geq \Vert u\Vert_{2}\ (q>2).
\]
Hence
\[
\underset{u\in H}{\sup}I_{\beta}^{\alpha}(u)=\underset{u\in H_{r}}{\sup
}I_{\beta}^{\alpha}(u),
\]
where $H_{r}$ denotes all radial functions in $H$. Therefore, we may assume
that $\{u_{k}\}_{k}\in H$ is a radially maximizing sequence for $\underset
{u\in H}{\sup}I_{\beta}^{\alpha}(u)$, that is
\[
\Vert u_{k}\Vert_{H^{2}(\mathbb{R}^{4})}=1,\lim_{k\rightarrow\infty}I_{\beta
}^{\alpha}(u_{k})=\underset{u\in H}{\sup}I_{\beta}^{\alpha}(u).
\]
By the Sobolev compact embedding, there exists a subsequence $\{u_{k}\}_{k}$
such that
\[%
\begin{split}
&  u_{k}(x)\rightarrow u(x),\quad\mbox{strongly}~\mbox{in}~L^{q}%
(B_{R}(0))\ \ \mathrm{for\ any}\ R>0,q>1\\
&  u_{k}(x)\rightarrow u(x),\quad\text{for a.e.}~x\in\mathbb{R}^{4}.
\end{split}
\]

Since $\exp(\beta|u|^{2})-1-\alpha|u|^{2}\in L^{p}(B_{R})$ for some $p>1$, we
have
\begin{equation}
\lim_{k\rightarrow\infty}\int_{B_{R}}\left(  \exp(\beta|u_{k}|^{2}%
)-1-\alpha|u_{k}|^{2}\right)  dx=\int_{B_{R}}\left(  \exp(\beta|u|^{2}%
)-1-\alpha|u|^{2}\right)  dx. \label{add1}%
\end{equation}
On the other hand, it follows from the radial lemma that
\begin{equation}%
\begin{split}
&  \lim_{k\rightarrow\infty}\int_{\mathbb{R}^{4}\setminus B_{R}}\left(
\exp(\beta|u_{k}|^{2})-1-\beta u_{k}^{2}\right)  dx\\
&  \ \ \leq c\lim_{k\rightarrow\infty}\int_{\mathbb{R}^{4}\setminus B_{R}%
}|u_{k}|^{4}dx\\
&  \ \ \leq c\sup_{k}\Vert u_{k}\Vert_{H^{2}(\mathbb{R}^{4})}^{2}R^{-2}.
\end{split}
\label{add2}%
\end{equation}

From (\ref{add1}) and (\ref{add2}), we derive that\ \ \
\[%
\begin{split}
&  \lim_{k\rightarrow\infty}\int_{\mathbb{R}^{4}}\left(  \exp(\beta|u_{k}%
|^{2})-1-\beta u_{k}^{2}\right)  dx\\
&  =\underset{R\rightarrow\infty}{\lim}\lim_{k\rightarrow\infty}\left(
\int_{B_{R}}+\int_{%
%TCIMACRO{\U{211d} }%
%BeginExpansion
\mathbb{R}
%EndExpansion
^{4}\backslash B_{R}}\right)  \left(  \exp(\beta|u_{k}|^{2})-1-\beta u_{k}%
^{2}\right)  dx\\
&  =\int_{\mathbb{R}^{4}}\left(  \exp(\beta|u|^{2})-1-\beta\left\vert
u\right\vert ^{2}\right)  dx.
\end{split}
\]
Hence, we have%
\begin{align}
&  \lim_{k\rightarrow\infty}\int_{\mathbb{R}^{4}}\left(  \exp(\beta|u_{k}%
|^{2})-1-\alpha|u_{k}|^{2}\right)  dx\nonumber\\
&  =\lim_{k\rightarrow\infty}\int_{\mathbb{R}^{4}}\left(  \exp(\beta
|u_{k}|^{2})-1-\beta|u_{k}|^{2}\right)  dx+\left(  \beta-\alpha\right)
\lim_{k\rightarrow\infty}\int_{\mathbb{R}^{4}}u_{k}^{2}dx\nonumber\\
&  =\int_{\mathbb{R}^{4}}\left(  \exp(\beta|u|^{2})-1-\beta|u|^{2}\right)
dx+\left(  \beta-\alpha\right)  \lim_{k\rightarrow\infty}\int_{\mathbb{R}^{4}%
}u_{k}^{2}dx\nonumber\\
&  =\int_{\mathbb{R}^{4}}\left(  \exp(\beta|u|^{2})-1-\alpha|u|^{2}\right)
dx+\lim_{k\rightarrow\infty}\left(  \beta-\alpha\right)  \int_{\mathbb{R}^{4}%
}\left(  u_{k}^{2}-|u|^{2}\right)  dx. \label{1.1}%
\end{align}
When $u\neq0$, we set
\[
\tau^{4}=\lim_{k\rightarrow\infty}\frac{\int_{\mathbb{R}^{4}}u_{k}^{2}dx}%
{\int_{\mathbb{R}^{4}}u^{2}dx},
\]
by Fatou's lemma, we have $\tau\geq1$. Let $\tilde{u}(x)=u(\frac{x}{\tau})$,
we can easily verify the following fact:%
\[
\int_{\mathbb{R}^{4}}|\Delta\tilde{u}|^{2}dx=\int_{\mathbb{R}^{4}}|\Delta
u|^{2}dx\leq\lim_{k\rightarrow\infty}\int_{\mathbb{R}^{4}}|\Delta u_{k}%
|^{2}dx,
\]%
\[
\int_{\mathbb{R}^{4}}\tilde{u}^{2}dx=\tau^{4}\int_{\mathbb{R}^{4}}u^{2}%
dx=\lim_{k\rightarrow\infty}\int_{\mathbb{R}^{4}}u_{k}^{2}dx
\]
and
\[
\int_{\mathbb{R}^{4}}|\left(  \Delta\tilde{u}|^{2}+\tilde{u}^{2}\right)
dx\leq\lim_{k\rightarrow\infty}\int_{\mathbb{R}^{4}}\left(  |\Delta u_{k}%
|^{2}+u_{k}^{2}\right)  dx=1.
\]

Hence, by (\ref{1.1}) we get%
\begin{equation}%
\begin{split}
&  \underset{u\in H}{\sup}I_{\beta}^{\alpha}(u)\geq\int_{\mathbb{R}^{4}%
}\left(  \exp(\beta\tilde{u}^{2})-1-\alpha\tilde{u}^{2}\right)  dx=\tau
^{4}\int_{\mathbb{R}^{4}}\left(  \exp(\beta u^{2})-1-\alpha u^{2}\right)  dx\\
&  =\int_{\mathbb{R}^{4}}\left(  \exp(\beta u^{2})-1-\alpha u^{2}\right)
dx+(\tau^{4}-1)\left(  \beta-\alpha\right)  \int_{\mathbb{R}^{4}}u^{2}dx+\\
&  +(\tau^{4}-1)\int_{\mathbb{R}^{4}}\left(  \exp(\beta u^{2})-1-\beta
u^{2}\right)  dx\\
&  \geq\lim_{k\rightarrow\infty}\int_{\mathbb{R}^{4}}\left(  \exp(\beta
u_{k}^{2})-1-\alpha u_{k}^{2}\right)  dx+(\tau^{4}-1)\int_{\mathbb{R}^{4}%
}\left(  \exp(\beta u^{2})-1-\beta u^{2}\right)  dx\\
&  =\underset{u\in H}{\sup}I_{\beta}^{\alpha}\left(  u\right)  +(\tau
^{4}-1)\int_{\mathbb{R}^{4}}\left(  \exp(\beta u^{2})-1-\beta u^{2}\right)
dx.
\end{split}
\label{pro}%
\end{equation}
Since $\exp(\beta u^{2})-1-\beta u^{2}>0$, we have $\tau=1$, and then
\[
\underset{u\in H}{\sup}I_{\beta}^{\alpha}(u)=\int_{\mathbb{R}^{4}}\left(
\exp(\beta u^{2})-1-\alpha u^{2}\right)  dx.
\]
Therefore, $u$ is an extremal function for $\underset{u\in H}{\sup}I_{\beta
}^{\alpha}(u)$.

Next, it suffices to show that $u=0$ is impossible to happen. Assume by
contradiction that $u=0$, we derive from radial lemma that
\begin{equation}%
\begin{split}
\underset{u\in H}{\sup}I_{\beta}^{\alpha}(u)  &  =\lim_{k\rightarrow\infty
}\int_{\mathbb{R}^{4}}\left(  \exp(\beta|u_{k}|^{2})-1-\alpha u_{k}%
^{2}\right)  dx\\
&  =\lim_{R\rightarrow\infty}\lim_{k\rightarrow\infty}\left(  \int_{B_{R}%
}+\int_{\mathbb{R}^{4}\setminus B_{R}}\right)  \left(  \exp(\beta|u_{k}%
|^{2})-1-\alpha u_{k}^{2}\right)  dx\\
&  =\lim_{R\rightarrow\infty}\lim_{k\rightarrow\infty}\left(  \beta
-\alpha\right)  \int_{\mathbb{R}^{4}\setminus B_{R}}u_{k}^{2}dx\leq
\beta-\alpha.
\end{split}
\label{s1}%
\end{equation}
On the other hand, for any $v\in H^{2}(\mathbb{R}^{4})$ and $t>0$, we
introduce a family of functions $v_{t}$ by
\[
v_{t}(x)=t^{\frac{1}{2}}v(t^{\frac{1}{4}}x),
\]
and we easily verify that
\[
\Vert\Delta v_{t}\Vert_{2}^{2}=t\Vert\Delta v\Vert_{2}^{2},\Vert v_{t}%
\Vert_{p}^{p}=t^{\frac{p-2}{2}}\Vert v\Vert_{p}^{p}.
\]
Hence, it follows that
\begin{equation}%
\begin{split}
&  \int_{\mathbb{R}^{4}}\left(  \exp\Big(\beta\big(\frac{v_{t}}{\Vert
v_{t}\Vert_{H^{2}(\mathbb{R}^{4})}}\big)^{2}\Big)-1-\alpha\left(  \frac{v_{t}%
}{\Vert v_{t}\Vert_{H^{2}(\mathbb{R}^{4})}}\right)  ^{2}\right)  dx\\
&  \ \ \geq\left(  \beta-\alpha\right)  \frac{\Vert v_{t}\Vert_{2}^{2}}%
{\Vert\Delta v_{t}\Vert_{2}^{2}+\Vert v_{t}\Vert_{2}^{2}}+\frac{\beta^{2}}%
{2}\frac{\Vert v_{t}\Vert_{4}^{4}}{(\Vert\Delta v_{t}\Vert_{2}^{2}+\Vert
v_{t}\Vert_{2}^{2})^{2}}\\
&  \ \ =\left(  \beta-\alpha\right)  \Big(\frac{\Vert v\Vert_{2}^{2}}%
{t\Vert\Delta v\Vert_{2}^{2}+\Vert v\Vert_{2}^{2}}\Big)+\frac{\beta^{2}}%
{2}\Big(\frac{t\Vert v\Vert_{4}^{4}}{(t\Vert\Delta v\Vert_{2}^{2}+\Vert
v\Vert_{2}^{2})^{2}}\Big)\\
&  \ \ =\left(  \beta-\alpha\right)  \Big(\frac{\Vert v\Vert_{2}^{2}}%
{t\Vert\Delta v\Vert_{2}^{2}+\Vert v\Vert_{2}^{2}}+\frac{\beta^{2}}{2\left(
\beta-\alpha\right)  }\frac{t\Vert v\Vert_{4}^{4}}{(t\Vert\Delta v\Vert
_{2}^{2}+\Vert v\Vert_{2}^{2})^{2}}\Big)\\
&  \ \ =\left(  \beta-\alpha\right)  g_{v}(t).
\end{split}
\label{s2}%
\end{equation}
Note that $g_{v}(0)=1$, once we show that $g_{v}^{\prime}(t)>0$ for small
$t>0$, then we have $g_{v}(t)>g_{v}(0)$ for small $t>0$, which leads to
$\underset{u\in H}{\sup}I_{\beta}^{\alpha}(u)>\beta-\alpha$. Combining
\eqref{s1} and \eqref{s2}, we obtain a contradiction. This accomplishes the
proof of Lemma \ref{lem2}. \vskip0.1cm

In the following, we show there exists some $v\in H^{2}\left(  \mathbb{R}%
^{4}\right)  $ such that $g_{v}^{\prime}(0)>0$. Indeed, after a direct
calculation, we have
\begin{equation}%
\begin{split}
g_{v}^{\prime}(0)  &  =-\frac{\Vert\Delta v\Vert_{2}^{2}}{\Vert v\Vert_{2}%
^{2}}+\frac{\beta^{2}}{2\left(  \beta-\alpha\right)  }\frac{\Vert v\Vert
_{4}^{4}}{\Vert v\Vert_{2}^{4}}\\
&  =\frac{\Vert\Delta v\Vert_{2}^{2}}{\Vert v\Vert_{2}^{2}}(-1+\frac{\beta
^{2}}{2\left(  \beta-\alpha\right)  }\frac{\Vert v\Vert_{4}^{4}}{\Vert\Delta
v\Vert_{2}^{2}\Vert v\Vert_{2}^{2}}).
\end{split}
\label{c1}%
\end{equation}
It can be shown this supremun \[
\underset{v\in H^{2}\left(
%TCIMACRO{\U{211d} }%
%BeginExpansion
\mathbb{R}
%EndExpansion
^{4}\right)  \backslash\{0\}}{\sup}\frac{\Vert v\Vert_{4}^{4}}{\Vert\Delta
v\Vert_{2}^{2}\Vert v\Vert_{2}^{2}}%
\] could be attained by some $Q\in H^{2}%
(\mathbb{R}^{4})$, which must satisfy (after a rescaling $Q\rightarrow\mu
Q\left(  \lambda\cdot\right)  $, see \cite{Bellazzini}) the nonlinear
equation
\begin{equation}
\Delta^{2}Q+Q-\left\vert Q\right\vert ^{2}Q=0\text{ in }%
%TCIMACRO{\U{211d} }%
%BeginExpansion
\mathbb{R}
%EndExpansion
^{4}. \label{G-N Eq}%
\end{equation}
Set $v=Q$, we get $g_{v}^{\prime}(0)=\frac{\Vert\Delta Q\Vert_{2}^{2}}{\Vert
Q\Vert_{2}^{2}}\left(  -1+\frac{\beta^{2}}{2\left(  \beta-\alpha\right)
}B_{2}\right)  $. Thus, if $\frac{\beta^{2}B_{2}}{2}>\beta-\alpha$, we have
$g_{v}^{\prime}(0)>0$.
\end{proof}

\subsection{The radially symmetric maximizing sequence for critical
functional}

Let $\left\{  \beta_{k}\right\}  $ be an increasing sequence which converges
to $32\pi^{2}$. According to Lemma \ref{lem2}, we see that there exists a
radial function sequence $\{u_{k}\}_{k}$ satisfying $\Vert u_{k}\Vert
_{H^{2}(\mathbb{R}^{4})}=1$ such that
\[
\int_{\mathbb{R}^{4}}\left(  \exp(\beta_{k}|u_{k}|^{2})-1-\alpha|u_{k}%
|^{2}\right)  dx=\underset{u\in H}{\sup}\int_{\mathbb{R}^{4}}\left(
\exp(\beta_{k}|u|^{2})-1-\alpha|u|^{2}\right)  dx.
\]
It is not difficult to see that
\[
\underset{k\rightarrow\infty}{\lim}\int_{%
%TCIMACRO{\U{211d} }%
%BeginExpansion
\mathbb{R}
%EndExpansion
^{4}}\left(  \exp\left(  \beta_{k}\left\vert u_{k}\right\vert ^{2}\right)
-1-\alpha\left\vert u_{k}\right\vert ^{2}\right)  dx=S\left(  \alpha\right)
.
\]
In fact, for any given $\varphi\in H^{2}\left(
%TCIMACRO{\U{211d} }%
%BeginExpansion
\mathbb{R}
%EndExpansion
^{4}\right)  $ with $\int_{%
%TCIMACRO{\U{211d} }%
%BeginExpansion
\mathbb{R}
%EndExpansion
^{4}}\left(  \left\vert \varphi\right\vert ^{2}+\left\vert \Delta
\varphi\right\vert ^{2}\right)  dx=1$, we have
\[
\int_{%
%TCIMACRO{\U{211d} }%
%BeginExpansion
\mathbb{R}
%EndExpansion
^{4}}\left(  \exp\left(  \beta_{k}\left\vert \varphi\right\vert ^{2}\right)
-1-\alpha\left\vert \varphi\right\vert ^{2}\right)  dx\leq\int_{%
%TCIMACRO{\U{211d} }%
%BeginExpansion
\mathbb{R}
%EndExpansion
^{4}}\left(  \exp\left(  \beta_{k}\left\vert u_{k}\right\vert ^{2}\right)
-1-\alpha\left\vert u_{k}\right\vert ^{2}\right)  dx.
\]
It follows from Levi's lemma that
\[
\int_{%
%TCIMACRO{\U{211d} }%
%BeginExpansion
\mathbb{R}
%EndExpansion
^{4}}\left(  \exp\left(  32\pi^{2}\left\vert \varphi\right\vert ^{2}\right)
-1-\alpha\left\vert \varphi\right\vert ^{2}\right)  dx\leq\underset
{k\rightarrow\infty}{\lim}\int_{%
%TCIMACRO{\U{211d} }%
%BeginExpansion
\mathbb{R}
%EndExpansion
^{4}}\left(  \exp\left(  \beta_{k}\left\vert u_{k}\right\vert ^{2}\right)
-1-\alpha\left\vert u_{k}\right\vert ^{2}\right)  dx,
\]
which implies that
\[
\underset{k\rightarrow\infty}{\lim}\int_{%
%TCIMACRO{\U{211d} }%
%BeginExpansion
\mathbb{R}
%EndExpansion
^{4}}\left(  \exp\left(  \beta_{k}\left\vert u_{k}\right\vert ^{2}\right)
-1-\alpha\left\vert u_{k}\right\vert ^{2}\right)  dx=S\left(  \alpha\right)
.
\]

An easy computation shows that the Euler--Lagrange equation of $u_{k}$ is
given by the following bi-harmonic equation in $\mathbb{R}^{4}$:%

\begin{equation}
\triangle^{2}u_{k}+u_{k}=\lambda_{k}^{-1}u_{k}\left(  \exp(\beta_{k}u_{k}%
^{2})-\frac{\alpha}{\beta_{k}}\right)  , \label{euler}%
\end{equation}
where $\left\Vert u_{k}\right\Vert _{H^{2}\left(  \mathbb{R}^{4}\right)  }=1$
and $\lambda_{k}=\int_{%
%TCIMACRO{\U{211d} }%
%BeginExpansion
\mathbb{R}
%EndExpansion
^{4}}u_{k}^{2}\left(  \exp\left\{  \beta_{k}u_{k}^{2}\right\}  -\frac{\alpha
}{\beta_{k}}\right)  dx$. Since%
\[
\lambda_{k}^{-1}u_{k}\left(  \exp\left\{  \beta_{k}u_{k}^{2}\right\}
-\frac{\alpha}{\beta_{k}}\right)  \in L^{p}_{loc}(\mathbb{R}^{4})
\]
for any $1\leq p<\infty$, by Lemma \ref{local estimate}, we know $u_{k}\in
C^{\infty}(\mathbb{R}^{4})$.

Now, we give the following important observation.

\begin{lemma}
\label{lamna}$\underset{k}{\inf}\lambda_{k}>0.$
\end{lemma}

\begin{proof}
We assume by contradiction that $\lambda_{k}\rightarrow0$. Since $\exp t-1\leq
t\exp t$, we derive that
\begin{equation}
\int_{%
%TCIMACRO{\U{211d} }%
%BeginExpansion
\mathbb{R}
%EndExpansion
^{4}}\left(  \beta_{k}u_{k}^{2}\exp\left(  \beta_{k}u_{k}^{2}\right)  \right)
dx\geq\int_{%
%TCIMACRO{\U{211d} }%
%BeginExpansion
\mathbb{R}
%EndExpansion
^{4}}\left(  \exp(\beta_{k}u_{k}^{2})-1\right)  dx. \label{1}%
\end{equation}
Hence
\begin{align*}
&  \underset{u\in H}{\sup}\frac{1}{32\pi^{2}}\int_{%
%TCIMACRO{\U{211d} }%
%BeginExpansion
\mathbb{R}
%EndExpansion
^{4}}\left(  \exp\left(  32\pi^{2}\left\vert u\right\vert ^{2}\right)
-1-\alpha\left\vert u\right\vert ^{2}\right)  dx\\
&  =\underset{k\rightarrow\infty}{\lim}\frac{1}{\beta_{k}}\int_{\mathbb{R}%
^{4}}\left(  \exp(\beta_{k}u_{k}^{2})-1-\alpha\left\vert u_{k}\right\vert
^{2}\right)  dx\\
&  \leq\underset{k\rightarrow\infty}{\lim}\frac{1}{\beta_{k}}\int_{%
%TCIMACRO{\U{211d} }%
%BeginExpansion
\mathbb{R}
%EndExpansion
^{4}}\left(  \beta_{k}u_{k}^{2}\exp\left(  \beta_{k}u_{k}^{2}\right)
-\alpha\left\vert u_{k}\right\vert ^{2}\right)  dx\\
&  =\underset{k\rightarrow\infty}{\lim}\int_{%
%TCIMACRO{\U{211d} }%
%BeginExpansion
\mathbb{R}
%EndExpansion
^{4}}u_{k}^{2}\left(  \exp\left\{  \beta_{k}u_{k}^{2}\right\}  -\frac{\alpha
}{\beta_{k}}\right)  dx\\
&  =\underset{k\rightarrow\infty}{\lim}\lambda_{k}\rightarrow0,
\end{align*}
which is a contradiction.
\end{proof}

\bigskip Now, we introduce the following

\begin{definition}
We said that $\{u_{k}\}_{k}$ is a normalized vanishing sequence, if
$\{u_{k}\}_{k}$ satisfies $\Vert u_{k}\Vert_{H^{2}(\mathbb{R}^{4})}=1$,
$u_{k}\rightharpoonup0$ in $H^{2}(\mathbb{R}^{4})$ and
\[
\lim\limits_{R\rightarrow\infty}\lim\limits_{k\rightarrow\infty}\int_{B_{R}%
}\left(  \exp(\beta_{k}|u_{k}|^{2})-1-\alpha\left\vert u_{k}\right\vert
^{2}\right)  dx=0.
\]
\bigskip
\end{definition}

Extracting a subsequence and changing the sign of $u_{k}$, we can always take
a point $x_{k}\in%
%TCIMACRO{\U{211d} }%
%BeginExpansion
\mathbb{R}
%EndExpansion
^{4}$ such that
\[
c_{k}=\max\left\vert u_{k}\right\vert =u_{k}\left(  x_{k}\right)  .
\]
If $c_{k}$ is bounded from above, we have the following

\begin{lemma}
\label{lemx2} If $\sup_{k}c_{k}<+\infty$, then one of the following holds.

(i) $u\neq0$ and $S\left(  \alpha\right)  $ could be achieved by a radial
function $u\in H^{2}(\mathbb{R}^{4})$,

(ii) $u=0$ and $\{u_{k}\}$ is a normalized vanishing sequence, furthermore,
$S\left(  \alpha\right)  \leq32\pi^{2}-\alpha$.
\end{lemma}

\begin{proof}
If $\sup_{k}c_{k}<+\infty$, it follows from the standard elliptic estimates
(see Lemma \ref{local estimate}) that $u_{k}\rightarrow u$ in $C_{loc}%
^{3}(\mathbb{R}^{4})$. Then for any $R>0$, we have
\begin{equation}
\lim_{k\rightarrow\infty}\int_{B_{R}}\left(  \exp(\beta_{k}|u_{k}%
|^{2})-1-\alpha\left\vert u_{k}\right\vert ^{2}\right)  dx=\int_{B_{R}}\left(
\exp(32\pi^{2}|u|^{2})-1-\alpha\left\vert u\right\vert ^{2}\right)  dx.
\label{9}%
\end{equation}
On the other hand, according to the radial lemma, we derive that
\[%
\begin{split}
&  \lim_{k\rightarrow\infty}\int_{\mathbb{R}^{4}\setminus B_{R}}\left(
\exp(\beta_{k}|u_{k}|^{2})-1-\beta_{k}u_{k}^{2}\right)  dx\\
&  \ \ \leq c\lim_{k\rightarrow\infty}\int_{\mathbb{R}^{4}\setminus B_{R}%
}|u_{k}|^{4}dx\\
&  \ \ \leq c\sup_{k}\Vert u_{k}\Vert_{H^{2}(\mathbb{R}^{4})}^{2}R^{-2}.
\end{split}
\]
Similar as (\ref{1.1}), we have
\begin{align}
\lim_{k\rightarrow\infty}\int_{\mathbb{R}^{4}}\left(  \exp(\beta_{k}%
|u_{k}|^{2})-1-\alpha\left\vert u_{k}\right\vert ^{2}\right)  dx  &
=\int_{\mathbb{R}^{4}}\left(  \exp(32\pi^{2}|u|^{2})-1-\alpha\left\vert
u\right\vert ^{2}\right)  dx\nonumber\\
&  +\left(  32\pi^{2}-\alpha\right)  \lim_{k\rightarrow\infty}\int
_{\mathbb{R}^{4}}(u_{k}^{2}-u^{2})dx. \label{1.1a}%
\end{align}

When $u\neq0$, we set
\[
\tau^{4}=\lim_{k\rightarrow\infty}\frac{\int_{\mathbb{R}^{4}}u_{k}^{2}dx}%
{\int_{\mathbb{R}^{4}}u^{2}dx},
\]
\bigskip and let $\tilde{u}(x)=u(\frac{x}{\tau})$, as we did in (\ref{pro}),
we have
\[%
\begin{split}
S\left(  \alpha\right)   &  \geq\int_{\mathbb{R}^{4}}\left(  \exp(32\pi
^{2}\tilde{u}^{2})-1-\alpha\tilde{u}^{2}\right)  dx\\
&  =\tau^{4}\int_{\mathbb{R}^{4}}\left(  \exp(32\pi^{2}u^{2})-1-\alpha
u^{2}\right)  dx\\
&  \geq S\left(  \alpha\right)  +(\tau^{4}-1)\int_{\mathbb{R}^{4}}\left(
\exp(32\pi^{2}u^{2})-1-32\pi^{2}u^{2}\right)  dx.
\end{split}
\]
Since $\exp(32\pi^{2}u^{2})-1-32\pi^{2}u^{2}>0$, we have $\tau=1$, and then
\[
S\left(  \alpha\right)  =\int_{\mathbb{R}^{4}}\left(  \exp(32\pi^{2}%
u^{2})-1-\alpha u^{2}\right)  dx.
\]
So, $u$ is an extremal function.

When $u=0$, by (\ref{9}) we know $\{u_{k}\}$ is a normalized vanishing
sequence. Furthermore, by (\ref{1.1a}), we get
\[%
\begin{split}
S\left(  \alpha\right)   &  =\lim_{k\rightarrow\infty}\int_{\mathbb{R}^{4}%
}\left(  \exp(\beta_{k}|u_{k}|^{2})-1-\alpha u_{k}^{2}\right)  dx\\
&  =\lim_{R\rightarrow\infty}\lim_{k\rightarrow\infty}\left(  \beta_{k}%
-\alpha\right)  \int_{\mathbb{R}^{4}}u_{k}^{2}dx\leq32\pi^{2}-\alpha.
\end{split}
\]

\end{proof}

\bigskip In the following, we show that the second case of Lemma \ref{lemx2}
will not happen.

Setting
\[
d_{nv}=\sup_{u_{k}:(NVS)}\lim\limits_{k\rightarrow\infty}\int_{\mathbb{R}^{4}%
}\left(  \exp(\beta_{k}|u_{k}|^{2})-1-\alpha\left\vert u_{k}\right\vert
^{2}\right)  dx,
\]
we have the following

\begin{proposition}%
\[
d_{nv}=32\pi^{2}-\alpha.
\]

\end{proposition}

\begin{proof}
Recalling in the proof of Lemma \ref{lemx2}, we have verified that if
$\{u_{k}\}_{k}$ is a radially symmetric normalized vanishing sequence, then
\[
\lim_{k\rightarrow\infty}\int_{\mathbb{R}^{4}}\left(  \exp(\beta_{k}%
|u_{k}|^{2})-1-\alpha|u_{k}|^{2}\right)  dx\leq32\pi^{2}-\alpha,
\]
that is, $d_{nv}\leq32\pi^{2}-\alpha$. Next, we show that there exists a
radially symmetric normalized vanishing sequence $\{v_{k}\}$ such that
\[
\lim_{k\rightarrow\infty}\int_{\mathbb{R}^{4}}\left(  \exp(\beta_{k}%
|v_{k}|^{2})-1-\alpha|v_{k}|^{2}\right)  dx=32\pi^{2}-\alpha.
\]
Picking a smooth radially symmetric function $\eta$ satisfying $\Vert
\Delta\eta\Vert_{2}=\Vert\eta\Vert_{2}=1$ with a compact support. Let
$\omega_{k}$ be a function defined by $\omega_{k}(x)=\rho_{k}^{2}\eta(\rho
_{k}x)$ for $\rho_{k}>0$, it is easy to check that
\[
\Vert\Delta\omega_{k}\Vert_{2}=\rho_{k}^{2}\text{ and }\Vert\omega_{k}%
\Vert_{2}=1.
\]
Setting $\bar{\omega}_{k}=\frac{\omega_{k}}{(1+\rho_{k}^{4})^{\frac{1}{2}}}$
and letting $\lim\limits_{k\rightarrow\infty}\rho_{k}=0$, we can verify
that $\Vert\bar{\omega}_{k}\Vert_{H^{2}\left(  \mathbb{R}^{4}\right)  }=1$,
$\bar{\omega}_{k}\rightarrow0$ in $L_{loc}^{2}(\mathbb{R}^{4})$ and
\[
\lim\limits_{k\rightarrow\infty}\Vert\Delta\bar{\omega}_{k}\Vert_{2}%
=0,\lim\limits_{k\rightarrow\infty}\Vert\bar{\omega}_{k}\Vert_{2}=1.
\]
Hence, $\{\bar{\omega}_{k}\}_{k}$ is a radially symmetric normalized vanishing
sequence. Through the radial lemma and the definition of the normalized
vanishing sequence, we have
\[%
\begin{split}
&  \lim_{k\rightarrow\infty}\int_{\mathbb{R}^{4}}\left(  \exp(\beta_{k}%
|\bar{\omega}_{k}|^{2})-1-\alpha|\bar{\omega}_{k}|^{2}\right)  dx\\
&  =\lim_{k\rightarrow\infty}\int_{\mathbb{R}^{4}\setminus B_{R}}\left(
\exp(\beta_{k}|\bar{\omega}_{k}|^{2})-1-\alpha|\bar{\omega}_{k}|^{2}\right)
dx\\
&  =\lim_{k\rightarrow\infty}\int_{\mathbb{R}^{4}\setminus B_{R}}\left(
\beta_{k}-\alpha\right)  \bar{\omega}_{k}^{2}dx\\
&  =\lim_{k\rightarrow\infty}\int_{\mathbb{R}^{4}}\left(  \beta_{k}%
-\alpha\right)  \bar{\omega}_{k}^{2}dx=32\pi^{2}-\alpha,
\end{split}
\]
which completes the proof.
\end{proof}

\begin{proposition}
It holds that $S>d_{nv}$.
\end{proposition}

\begin{proof}
For any $v\in H^{2}(\mathbb{R}^{4})$ and $t>0$, we introduce a family of
functions $v_{t}$ by
\[
v_{t}(x)=t^{\frac{1}{2}}v(t^{\frac{1}{4}}x).
\]
then
\[
\Vert\Delta v_{t}\Vert_{2}^{2}=t\Vert\Delta v\Vert_{2}^{2},\Vert v_{t}%
\Vert_{p}^{p}=t^{\frac{p-2}{2}}\Vert v\Vert_{p}^{p}%
\]
for any $p\geq2$. Similar as that in Lemma \ref{lem2}, we can get%
\[%
\begin{split}
&  \int_{\mathbb{R}^{4}}\left(  \exp\left(  32\pi^{2}\left(  \frac{v_{t}%
}{\Vert v_{t}\Vert_{H^{2}(\mathbb{R}^{4})}}\right)  ^{2}\right)
-1-\alpha\left(  \frac{v_{t}}{\Vert v_{t}\Vert_{H^{2}(\mathbb{R}^{4})}%
}\right)  ^{2}\right)  dx\\
&  \geq\left(  32\pi^{2}-\alpha\right)  g_{v}(t),
\end{split}
\]
where
\[
g_{v}(t)=\left(  \frac{\Vert v\Vert_{2}^{2}}{t\Vert\Delta v\Vert_{2}^{2}+\Vert
v\Vert_{2}^{2}}+\frac{(32\pi^{2})^2}{2\left(  32\pi^{2}-\alpha\right)  }%
\frac{t\Vert v\Vert_{4}^{4}}{(t\Vert\Delta v\Vert_{2}^{2}+\Vert v\Vert_{2}%
^{2})^{2}}\right)  .
\]
Furthermore, one can show that $g_{Q}^{\prime}(t)>0$\ ($Q$ is the ground
state\ solution of (\ref{G-N Eq})) for small $t>0$, provided $32\pi^{2}%
-\alpha<\frac{\left(  32\pi^{2}\right)  ^{2}B_{2}}{2}$, which implies that
$S>d_{nv}$.
\end{proof}

\section{Blow up analysis}

In this section, we are interested the blow-up case, that is,
\begin{equation}
c_{k}\rightarrow+\infty, \label{blowup}%
\end{equation}
the method of blow-up analysis will be used to analyze the asymptotic behavior
of the radially maximizing sequence $\left\{  u_{k}\right\}  _{k} $. By the
radial lemma, we have $x_{k}\rightarrow0\in%
%TCIMACRO{\U{211d} }%
%BeginExpansion
\mathbb{R}
%EndExpansion
^{4}$. We call $0$ the blow-up point. Here and in the sequel, we do not
distinguish sequence and subsequence, the reader can understand it from the context.

Since $u_{k}$ is bounded in $H_{r}^{2}\left(
%TCIMACRO{\U{211d} }%
%BeginExpansion
\mathbb{R}
%EndExpansion
^{4}\right)  $, we have%

\begin{equation}
\left\{
\begin{array}
[c]{l}%
u_{k}\rightharpoonup u\text{ weakly in }H_{r}^{2}\left(
%TCIMACRO{\U{211d} }%
%BeginExpansion
\mathbb{R}
%EndExpansion
^{4}\right) \\
u_{k}\rightarrow u\text{ in }L^{s}\left(
%TCIMACRO{\U{211d} }%
%BeginExpansion
\mathbb{R}
%EndExpansion
^{4}\right)  \text{, }\forall s>2\\
\beta_{k}\rightarrow32\pi^{2}\text{. }%
\end{array}
\right.  \label{convergence}%
\end{equation}

\subsection{\bigskip Asymptotic behavior of $\left\{  u_{k}\right\}  _{k}$
near the $0$}

Let
\[
r_{k}^{4}=\frac{\lambda_{k}}{c_{k}^{2}e^{\beta_{k}c_{k}^{2}}}.
\]
We claim that $r_{k}^{4}$ converges to zero rapidly. Indeed we have for any
$\gamma<32\pi^{2}$,%

\begin{align}
r_{k}^{4}c_{k}^{2}e^{\gamma c_{k}^{2}}  &  =e^{\left(  \gamma-\beta
_{k}\right)  c_{k}^{2}}\int_{%
%TCIMACRO{\U{211d} }%
%BeginExpansion
\mathbb{R}
%EndExpansion
^{4}}u_{k}^{2}(\exp\left(  \beta_{k}u_{k}^{2}\right)-\frac{\alpha}{\beta_{k}})  dx\nonumber\\
&  \leq\int_{%
%TCIMACRO{\U{211d} }%
%BeginExpansion
\mathbb{R}
%EndExpansion
^{4}}u_{k}^{2}\exp\left(  \beta_{k}u_{k}^{2}\right)  \exp\left(  \left(
\gamma-\beta_{k}\right)  u_{k}^{2}\right)  dx\nonumber\\
&  =\int_{%
%TCIMACRO{\U{211d} }%
%BeginExpansion
\mathbb{R}
%EndExpansion
^{4}}u_{k}^{2}\exp\left(  \gamma u_{k}^{2}\right)  dx\nonumber\\
&  =\int_{%
%TCIMACRO{\U{211d} }%
%BeginExpansion
\mathbb{R}
%EndExpansion
^{4}}u_{k}^{2}\left(  \exp\left(  \gamma u_{k}^{2}\right)  -1\right)
dx+\int_{%
%TCIMACRO{\U{211d} }%
%BeginExpansion
\mathbb{R}
%EndExpansion
^{4}}u_{k}^{2}dx\nonumber\\
&  \leq\left(  \int_{%
%TCIMACRO{\U{211d} }%
%BeginExpansion
\mathbb{R}
%EndExpansion
^{4}}u_{k}^{s}dx\right)  ^{\frac{2}{s}}\left(  \int_{%
%TCIMACRO{\U{211d} }%
%BeginExpansion
\mathbb{R}
%EndExpansion
^{4}}\left(  \exp\left(  \gamma u_{k}^{2}\right)  -1\right)  ^{\frac{s}{s-2}%
}dx\right)  ^{\frac{s-2}{s}}+\int_{%
%TCIMACRO{\U{211d} }%
%BeginExpansion
\mathbb{R}
%EndExpansion
^{4}}u_{k}^{2}dx\nonumber\\
&  \leq c\left(  \int_{%
%TCIMACRO{\U{211d} }%
%BeginExpansion
\mathbb{R}
%EndExpansion
^{4}}u_{k}^{s}dx\right)  ^{\frac{2}{s}}\left(  \int_{%
%TCIMACRO{\U{211d} }%
%BeginExpansion
\mathbb{R}
%EndExpansion
^{4}}\left(  \exp\left(  \frac{\gamma s}{s-2}u_{k}^{2}\right)  -1\right)
dx\right)  ^{\frac{s-2}{s}}+\int_{%
%TCIMACRO{\U{211d} }%
%BeginExpansion
\mathbb{R}
%EndExpansion
^{4}}u_{k}^{2}dx\nonumber\\
&  \leq c \label{decay for rk}%
\end{align}
provided $s$ large enough, here we have used the Adams inequality in $%
%TCIMACRO{\U{211d} }%
%BeginExpansion
\mathbb{R}
%EndExpansion
^{4}$ and (\ref{convergence}).

To understand the asymptotic behavior of $u_{k}$ near the blow-up point, we
define three sequences of functions on $%
%TCIMACRO{\U{211d} }%
%BeginExpansion
\mathbb{R}
%EndExpansion
^{4}$, namely
\[
\left\{
\begin{array}
[c]{c}%
\phi_{k}\left(  x\right)  =\frac{u_{k}\left(  x_{k}+r_{k}x\right)  }{c_{k}},\\
v_{k}\left(  x\right)  =u_{k}\left(  x_{k}+r_{k}x\right)  -u_{k}\left(
x_{k}\right)  ,\\
\psi_{k}\left(  x\right)  =c_{k}\left(  u_{k}\left(  x_{k}+r_{k}x\right)
-c_{k}\right)  ,
\end{array}
\right.
\]
where $\phi_{k}$, $v_{k}$ and $\psi_{k}$ are defined on $\Omega_{k}:=\left\{
x\in%
%TCIMACRO{\U{211d} }%
%BeginExpansion
\mathbb{R}
%EndExpansion
^{4}:r_{k}x\in B_{1}\right\}  $.

\begin{lemma}
\label{convergence for fei}\bigskip$\phi_{k}\left(  x\right)  \rightarrow1$ in
$C_{loc}^{3}\left(
%TCIMACRO{\U{211d} }%
%BeginExpansion
\mathbb{R}
%EndExpansion
^{4}\right)  $.
\end{lemma}

\begin{proof}
From equation (\ref{euler}), the decay estimate of $r_{k}$ and the fact that
$\phi_{k}\leq1$, we know that for any $R>0$, and $x\in B_{R}\left(  0\right)
$, $\phi_{k}\left(  x\right)  $ satisfy%
\begin{align*}
\left\vert \Delta^{2}\phi_{k}\left(  x\right)  \right\vert  &  =\left\vert
\frac{r_{k}^{4}}{c_{k}}\left(  \Delta^{2}u_{k}\right)  \left(  x_{k}%
+r_{k}x\right)  \right\vert \\
&  =\left\vert r_{k}^{4}\left(  \lambda_{k}^{-1}\phi_{k}\exp\left\{  \beta
_{k}u_{k}^{2}\left(  x_{k}+r_{k}x\right)  \right\}  -\left(  1+\frac{\alpha
}{\lambda_{k}\beta_{k}}\right)  \phi_{k}\right)  \right\vert \\
&  \leq\left\vert r_{k}^{4}\left(  \lambda_{k}^{-1}\phi_{k}\exp\left\{
\beta_{k}c_{k}^{2}\right\}  -\left(  1+\frac{\alpha}{\beta_{k}}\lambda
_{k}^{-1}\right)  \phi_{k}\right)  \right\vert \\
&  \leq cc_{k}^{-2}\left(  1+o\left(  1\right)  \right)  \rightarrow0,
\end{align*}
and $\phi_{k}\left(  x\right) $ is bounded in $L_{loc}^{1}(\mathbb{R}^{4})$. The standard
regularity theory gives for any $R>0$ and some $0<\gamma<1$, $\Vert\phi_{k}\left(  x\right)
\Vert_{C^{3,\gamma}(B_{R}(0))}$ are uniformly bounded with the respect to $k$.
Through the Arzela-Ascoli theorem, there exists $\phi\in C^{3}(\mathbb{R}%
^{4})$ such that $\phi_{k}\left(  x\right)  \rightarrow\phi$ in $C_{loc}^{3}\left(
%TCIMACRO{\U{211d} }%
%BeginExpansion
\mathbb{R}
%EndExpansion
^{4}\right)  $ with $\triangle^2\phi=0$ in $%
%TCIMACRO{\U{211d} }%
%BeginExpansion
\mathbb{R}
%EndExpansion
^{4}$. Since $\phi_{k}\left(  0\right)  =1$, we have $\phi=1$ in $%
%TCIMACRO{\U{211d} }%
%BeginExpansion
\mathbb{R}
%EndExpansion
^{4}$ by the Liouville Theorem.
\end{proof}

\begin{lemma}
\label{guji1} We have
\[
v_{k}(x)=u_{k}(x_{k}+r_{k}x)-u_{k}(x_{k})\rightarrow0\ \ \mathrm{in}%
\ C_{loc}^{3}(\mathbb{R}^{4})\
\]
as $k\rightarrow\infty$. Hence
\[
|\nabla^{i}u_{k}(x)|=o(r_{k}^{-i})\text{ in\ }B_{Rr_{k}}\text{,}%
i=1,2,3\text{,}%
\]
for any $R>0.$
\end{lemma}

\begin{proof}
It is obvious that $v_{k}$ solves the equation
\[
(-\Delta)^{2}v_{k}+r_{k}^{4}u_{k}(x_{k}+r_{k}x_{k})=\frac{u_{k}(x_{k}+r_{k}%
x)}{c_{k}^{2}}\exp(\beta_{k}u_{k}^{2}-\beta_{k}c_{k}^{2})-\frac{1}{\lambda
_{k}}\frac{\alpha}{\beta_{k}}r_{k}^{4}u_{k}(x_{k}+r_{k}x_{k}).
\]
Set $\Delta v_{k}=g_{k}$, and then $\Delta g_{k}=f_{k}$, where
\[
f_{k}=\frac{u_{k}(x_{k}+r_{k}x)}{c_{k}^{2}}\exp(\beta_{k}u_{k}^{2}-\beta
_{k}c_{k}^{2})-\left(  \frac{1}{\lambda_{k}}\frac{\alpha}{\beta_{k}}+1\right)
r_{k}^{4}u_{k}(x_{k}+r_{k}x_{k}).
\]
since $u_{k}$ is bounded in $H^{2}(\mathbb{R}^{4})$, directly computations
yields that
\[
\int_{\mathbb{R}^{4}}|g_{k}|^{2}dx=\int_{\mathbb{R}^{4}}|\Delta v_{k}%
|^{2}dx=\int_{\mathbb{R}^{4}}|\Delta u_{k}|^{2}dx<c.
\]
Also, since $f_{k}$\ is bounded in $L_{loc}^{p}(\mathbb{R}^{4})$ for any
$p\geq1$, by Lemma \ref{local estimate} we obtain that for some $0<\gamma<1$,
\begin{equation}
\Vert g_{k}\Vert_{C^{1,\gamma}(B_{R})}\leq c,\label{addddd}%
\end{equation}
for any $R>0$. On the other hand, by using Pizzetti's formula (see Lemma
\ref{Pizz}), we can derive
\[
\int_{B_{R}(0)}v_{k}(x)dx=c_{0}R^{4}v_{k}(0)+c_{1}R^{6}\Delta v_{k}%
(0)+c_{2}R^{8}\Delta^{2}v_{k}(\xi),
\]
for some $\xi\in B_{R}(0)$.

By (\ref{addddd}) and observe that $v_{k}\leq0$ and $v_{k}(0)=0$, then
$v_{k}(x)$ is bounded in $L_{loc}^{1}(\mathbb{R}^{4})$. Hence again by Lemma
\ref{local estimate}, we obtain that there exists some $v\in C^{3}%
(\mathbb{R}^{4})$ such that
\[
v_{k}(x)\rightarrow v\ \ \mathrm{in}\ C_{loc}^{3}(\mathbb{R}^{4})
\]
with $v$ satisfying
\[
(-\Delta)^{2}v=0.
\]
By the Lemma \ref{Liouville} and $v\leq 0$ , we know that $v$ is a polynomial degree at most 6. Since $\int_{\mathbb{R}^{4}}|\Delta v
|^{2}dx\leq \lim\limits_{k\rightarrow \infty} \int_{\mathbb{R}^{4}}|\Delta v_{k}
|^{2}dx\leq C$, then $v$ must be a constant. This together with  $v(0)=0$ implies $v(x)=0$.
\end{proof}

\vskip0.1cm

The following lemma plays an important role in determining the limit behavior
of $\psi_{k}\left(  x\right)  $.

\begin{lemma}
[Gradient estimate on $B_{Rr_{k}}$]\label{gra}For any $R>0$, there holds
\[
c_{k}\int_{B_{Rr_{k}}}|\Delta u_{k}|dx\leq c(Rr_{k})^{2}.
\]
Furthermore, we have
\[
\int_{B_{R}}|\Delta\psi_{k}|dx=c_{k}r_{k}^{-2}\int_{B_{Rr_{k}}}|\Delta
u_{k}|dx\leq cR^{2}.
\]

\end{lemma}

\begin{proof}
For any $R_{0}>0$, we introduce a sequence of bi-harmonic functions
$u_{k}^{R_{0}}$ solving
\begin{equation}%
%TCIMACRO{\QDATOPD{\{}{.}{\Delta^{2}u_{k}^{R_{0}}=0\text{ in }\overline
%{B_{R_{0}}\left(  x_{k}\right)  }}{\partial_{v}^{i}u_{k}^{R_{0}}=\partial
%_{v}^{i}u_{k}\text{ on }\partial\overline{B_{\delta}\left(  x_{k}\right)
%}\text{, }i=1,2} }%
%BeginExpansion
\genfrac{\{}{.}{0pt}{0}{\Delta^{2}u_{k}^{R_{0}}=0\text{ in }\overline
{B_{R_{0}}\left(  x_{k}\right),  }}{\partial_{v}^{i}u_{k}^{R_{0}}=\partial
_{v}^{i}u_{k}\text{ on }\partial\overline{R_{0}\left(  x_{k}\right)
}\text{, }i=0,1.}
%EndExpansion
\label{truc}%
\end{equation}
By the elliptic estimates (see Lemma \ref{whole estimate}) and the radial
lemma, we derive that
\begin{equation}
\Vert u_{k}^{R_{0}}\Vert_{C^{4}(B_{R_{0}})}<\frac{c}{R_{0}^{\tau}}
\label{es tru}%
\end{equation}
for some $\tau>0$.

Observe that $u_{k}-u_{k}^{R_{0}}$ satisfies the following equation
\[
\left\{
\begin{array}
[c]{l}%
\Delta^{2}\left(  u_{k}-u_{k}^{R_{0}}\right)  =\lambda_{k}^{-1}u_{k}\left(
\exp\left\{  \beta_{k}u_{k}^{2}\right\}  -\frac{\alpha}{\beta_{k}}%
u_{k}\right)  -u_{k}\text{ in }\overline{B_{R_{0}}\left(  0\right)  }\\
\partial_{v}^{i}\left(  u_{k}-u_{k}^{R_{0}}\right)  =0\text{ \ }\ \text{on
\ }\partial\overline{B_{R_{0}}\left(  0\right)  }\text{, }i=0,1.
\end{array}
\right.
\]

Set $f_{k}=\lambda_{k}^{-1}u_{k}\left(  \exp\left\{  \beta_{k}u_{k}%
^{2}\right\}  -\frac{\alpha}{\beta_{k}}u_{k}\right)  -u_{k}$, and define
$L(\log L)^{\alpha}(B_{R_{0}})$ as the space
\[
L(\log L)^{\alpha}(B_{R_{0}}):=\{f\in L^{1}(B_{R_0}):\Vert f\Vert_{L(\log
L)^{\alpha}}:=\int_{B_{R_{0}}}|f|(\log^{\alpha}(2+|f|))dx<\infty\}.
\]
endowed with the norm $\Vert f\Vert_{L(\log L)^{\alpha}}=\int_{B_{R_{0}}%
}|f|(\log^{\alpha}(2+|f|))dx$.

It is not difficult to check that $f_{k}$ is bounded in $ L(\log L)^{\frac{1}{2}}(B_{R_{0}%
})$. This together with Lemma \ref{Pizz1} directly leads to
\begin{equation}
\Vert\nabla^{j}(u_{k}-u_{k}^{R_{0}})\Vert_{L^{(4/j,2)}}\leq C,\ \ j=1,2,3.
\label{lorenz}%
\end{equation}
where $\Vert\cdot\Vert_{L^{(4/j,2)}}$ is the Lorentz norm. For the definition
of Lorentz spaces and their basic properties, we refer the interested reader to
\cite{ONeil}.

After some computation, we obtain
\begin{align*}
\left\vert \Delta^{2}((u_{k}-u_{k}^{R_{0}})^{2})\right\vert  & \leq\left\vert
2(u_{k}-u_{k}^{R_{0}})\Delta^{2}(u_{k}-u_{k}^{R_{0}})\right\vert \\
& +C\sum_{j=1}^{3}\left\vert \nabla^{j}(u_{k}-u_{k}^{R_{0}})\right\vert
\left\vert \nabla^{4-j}(u_{k}-u_{k}^{R_{0}})\right\vert |.
\end{align*}
Thanks to the Lorentz estimates of gradients (\ref{lorenz}) and some
H\"{o}lder type inequality of O'Neil \cite{ONeil}, we know the term
$\sum_{j=1}^{3}\left\vert \nabla^{j}(u_{k}-u_{k}^{R_{0}})\right\vert
\left\vert \nabla^{4-j}(u_{k}-u_{k}^{R_{0}})\right\vert $ is bounded in
$L^{1}(B_{R_{0}})$.

We now show that $|2(u_{k}-u_{k}^{R_{0}}%
)\Delta^{2}(u_{k}-u_{k}^{R_{0}})|$ is also bounded in $L^{1}(B_{R_{0}})$. In
fact, we observe that
\[%
\begin{split}
\int_{B_{R_{0}}}|2(u_{k}-u_{k}^{R_{0}})\Delta^{2}(u_{k}-u_{k}^{R_{0}})|dx  &
\leq2\int_{B_{R_{0}}}|u_{k}\Delta^{2}(u_{k})|dx+2\int_{B_{R_{0}}}|u_{k}%
^{R_{0}}\Delta^{2}(u_{k})|dx\\
&  =I_{1}+I_{2}.
\end{split}
\]
For $I_{1}$, by equation (\ref{euler}), we obtain
\[%
\begin{split}
\int_{B_{R_{0}}}|u_{k}\Delta^{2}(u_{k})|dx  &  \leq\int_{\mathbb{R}^{4}%
}\lambda_{k}^{-1}u_{k}^{2}\left(  \exp(\beta_{k}u_{k}^{2})-\frac{\alpha}%
{\beta_{k}}\right)  dx+\int_{\mathbb{R}^{4}}|u_{k}|^{2}dx\\
&  =\int_{\mathbb{R}^{4}}|\Delta u_{k}|^{2}dx+2\int_{\mathbb{R}^{4}}%
|u_{k}|^{2}dx\leq c.
\end{split}
\]
For $I_{2}$, we have
\begin{align*}
\int_{B_{R_{0}}}|u_{k}^{R_{0}}\Delta^{2}(u_{k})|dx  &  \leq c\int_{B_{R_{0}}%
}|u_{k}\Delta^{2}(u_{k})|dx+c\int_{B_{R_{0}}\cap\{|u_{k}|\leq1\}}|\Delta
^{2}(u_{k})|dx\\
&  \leq c(R_{0}).
\end{align*}
Using the above estimates, we conclude that $\int_{B_{R_{0}}}|\Delta
^{2}((u_{k}-u_{k}^{R_{0}})^{2})|dx\leq c$. Carrying out the same procedure as in the proof of
Lemma 6 of \cite{Mar}, we have for any $R>0$,
\begin{equation}
\int_{B_{Rr_{k}}}\Delta((u_{k}-u_{k}^{R_{0}})^{2})dx\leq C(Rr_{k})^{2}.
\label{addd}%
\end{equation}
Combining (\ref{addd}) and (\ref{es tru}), we derive by Lemma \ref{guji1}
that
\begin{equation}%
\begin{split}
\int_{B_{Rr_{k}}}|\Delta(u_{k}^{2})|dx &  \leq c\left\{  \int_{B_{Rr_{k}}%
}\Delta((u_{k}-u_{k}^{R_{0}})^{2})dx+\int_{B_{Rr_{k}}}|\Delta\left(  \left(
u_{k}^{R_{0}}\right)  ^{2}\right)  |dx\right.  \\
&  \ \ +\int_{B_{Rr_{k}}}|u_{k}^{R_{0}}\Delta u_{k}|dx+\int_{B_{Rr_{k}}}%
|u_{k}\Delta u_{k}^{R_{0}}|dx+\\
&  \left.  +\int_{B_{Rr_{k}}}|\nabla u_{k}\nabla u_{k}^{R_{0}}|dx\right\}  \\
&  \leq c\int_{B_{Rr_{k}}}\Delta((u_{k}-u_{k}^{R_{0}})^{2})dx+o(r_{k}^{2}).
\end{split}
\label{guji2}%
\end{equation}

On the other hand, we also have
\begin{equation}
c_{k}|\Delta u_{k}|\leq cu_{k}|\Delta u_{k}|\leq c\left(  \Delta(u_{k}%
^{2})+|\nabla u_{k}|^{2}\right)  \leq c\Delta(u_{k}^{2})+o(r_{k}^{-2}).
\label{guji3}%
\end{equation}
From (\ref{guji2}) and (\ref{guji3}), we conclude that
\[
c_{k}\int_{B_{Rr_{k}}}|\Delta u_{k}|dx\leq c(Rr_{k})^{2}.
\]
Hence it follows that for any $R>0$,
\[
\int_{B_{R}}|\Delta\psi_{k}|dx=c_{k}r_{k}^{-2}\int_{B_{Rr_{k}}}|\Delta
u_{k}|dx\leq cR^{2}.
\]
This accomplishes the proof.
\end{proof}

Now, we are in position to analyze the limit behavior of $\psi_{k}\left(
x\right)  $.

\begin{lemma}
\label{class}We have
\[
\psi_{k}\left(  x\right)  \rightarrow\psi\ \mathrm{in}\ C_{loc}^{3}%
(\mathbb{R}^{4}),
\]
where $\psi$ satisfies the equation
\[
(-\Delta)^{2}\psi=\exp(64\pi^{2}\psi).
\]
Furthermore, we have
\[
\psi(x)=\frac{1}{16\pi^{2}}\log\frac{1}{1+\frac{\pi}{\sqrt{6}}|x|^{2}},
\]
and $\int_{%
%TCIMACRO{\U{211d} }%
%BeginExpansion
\mathbb{R}
%EndExpansion
^{4}}\exp\left(  64\pi^{2}\psi\left(  x\right)  \right)  dx=1
$.
\end{lemma}

\begin{proof}
By equation (\ref{euler}), we see that $\psi_{k}$ satisfies the equation
\[
(-\Delta)^{2}\psi_{k}+c_{k}r_{k}^{4}u_{k}(x_{k}+r_{k}%
x)=\frac{u_{k}(x_{k}+r_{k}%
x)}{c_{k}}\exp(\beta_{k}u_{k}^{2}-\beta_{k}c_{k}^{2})-\frac{1}{\lambda_{k}%
}\frac{\alpha}{\beta_{k}}c_{k}r_{k}^{4}u_{k}(x_{k}+r_{k}%
x).
\]
According to Lemma \ref{gra}, we know that
\[
\int_{B_{R}}|\Delta\psi_{k}|dx=c_{k}r_{k}^{-2}\int_{B_{Rr_{k}}}|\Delta
u_{k}|dx\leq cR^{2}.
\]
This together with the elliptic estimates (see Lemma \ref{local estimate})
yields that $\Vert\Delta\psi_{k}\Vert_{C_{loc}^{1,\alpha}}\leq c$. As  in
Lemma \ref{guji1}, we know there exists some $\psi\in C^{3}(\mathbb{R}^{4})$
such that
\[
\psi_{k}\left(  x\right)  \rightarrow\psi\ \mathrm{in}\ C_{loc}^{3}%
(\mathbb{R}^{4}),
\]
with $\psi$ satisfying the equation
\[
(-\Delta)^{2}\psi=\exp(64\pi^{2}\psi).
\]
By Fatou's lemma, we have%
\[
\int_{\mathbb{R}^{4}}\exp(64\pi^{2}\psi)dx\leq\frac{1}{\lambda_{k}}%
\int_{\mathbb{R}^{4}}u_{k}^{2}\left(  \exp\left\{  \beta_{k}u_{k}^{2}\right\}
-\frac{\alpha}{\beta_{k}}\right)  dx\leq1.
\]
We now claim that $\psi$ must take the form as
\begin{equation}
\psi(x)=\frac{1}{16\pi^{2}}\log\frac{1}{1+\frac{\pi}{\sqrt{6}}|x|^{2}}.
\label{guji5}%
\end{equation}
We argue this by contradiction. If $\psi$ don't have the form as
(\ref{guji5}), according to \cite{lin1} (see also \cite{Mar1}), there exists
some $a<0$ such that
\[
\lim_{|x|\rightarrow+\infty}(-\Delta)\psi(x)=a.
\]
This would imply
\[
\lim_{k\rightarrow+\infty}\int_{B_{R}}|\Delta\psi_{k}(x)|dx=|a|vol(B_{1}%
(0))R^{4}+o(R^{4})\ \ \mathrm{as}\ R\rightarrow+\infty,
\]
but this contradicts $\int_{B_{R}}|\Delta\psi_{k}|dx<cR^{2}$. Hence we have
\[
\psi(x)=\frac{1}{16\pi^{2}}\log\frac{1}{1+\frac{\pi}{\sqrt{6}}|x|^{2}}.
\]
Furthermore, careful computations lead to
\begin{align*}
\int_{%
%TCIMACRO{\U{211d} }%
%BeginExpansion
\mathbb{R}
%EndExpansion
^{4}}\exp\left(  64\pi^{2}\psi\left(  x\right)  \right)  dx  &  =\int_{%
%TCIMACRO{\U{211d} }%
%BeginExpansion
\mathbb{R}
%EndExpansion
^{4}}\left(  \frac{1}{1+\frac{\pi}{\sqrt{6}}\left\vert x\right\vert ^{2}%
}\right)  ^{4}dx\text{ }(\text{setting }t=\frac{\pi}{\sqrt{6}}\left\vert
x\right\vert ^{2})\\
&  =\frac{\omega_{3}}{4}\int_{0}^{\infty}\left(  1+t\right)  ^{-4}d\left\vert
x\right\vert ^{4}\\
&  =\frac{3\omega_{3}}{\pi^{2}}\int_{0}^{\infty}\left(  1+t\right)  ^{-4}tdt\\
&  =\frac{3\omega_{3}}{\pi^{2}}\cdot\frac{1}{6}=1.
\end{align*}

\end{proof}

\subsection{\bigskip Bi-harmonic truncations}

In the following, we will need some bi-harmonic truncations $u_{k}^{M}$ which
was studied in \cite{Dela}. Roughly speaking,\ the value of truncations
$u_{k}^{M}$ is close to $\frac{c_{k}}{M}$ in a small ball centered at $x_{k}$,
and coincides with $u_{k}$ outside the same ball.

\begin{lemma}
\label{truncation}\cite[Lemma 4.20]{Dela} For any $M>1$ and $k\in%
%TCIMACRO{\U{2115} }%
%BeginExpansion
\mathbb{N}
%EndExpansion
$, there exists a radius $\rho_{k}^{M}>0$ and a constant $c=c\left(  M\right)
$ such that

1. $u_{k}\geq\frac{c_{k}}{M}$ in $B_{\rho_{k}^{M}}\left(  x_{k}\right)  ;$

2.$\left\vert u_{k}-\frac{c_{k}}{M}\right\vert \leq\frac{c}{c_{k}}$ on
$\partial B_{\rho_{k}^{M}}\left(  x_{k}\right)  ;$

3.$\left\vert \nabla^{l}u_{k}\right\vert \leq\frac{c}{c_{k}\left(  \rho
_{k}^{M}\right)  ^{l}}$ on $\partial B_{\rho_{k}^{M}}\left(  x_{k}\right)  $,
for any $1\leq l\leq3;$

4. $\rho_{k}^{M}\rightarrow0$ and $\frac{\rho_{k}^{M}}{r_{k}}\rightarrow
+\infty$, as $k\rightarrow\infty.$
\end{lemma}

Let $u_{k}^{\rho_{k}^{M}}\in C^{4}\left(  \overline{B_{\rho_{k}^{M}}\left(
x_{k}\right)  }\right)  $ be the unique solution of%

\[%
%TCIMACRO{\QDATOPD{\{}{.}{\Delta^{2}u_{k}^{\rho_{k}^{M}}%
%=0\ \,\ \ \ \ \ \ \ \ \ \ \ \ \text{\ in }B_{\rho_{k}^{M}}\left(
%x_{k}\right)  ,}{\partial_{v}^{i}u_{k}^{\rho_{k}^{M}}=\partial_{v}^{i}%
%u_{k}\ \text{on }\partial B_{\rho_{k}^{M}}\left(  x_{k}\right)  ,i=0,1.}}%
%BeginExpansion
\genfrac{\{}{.}{0pt}{0}{\Delta^{2}u_{k}^{\rho_{k}^{M}}%
=0\ \,\ \ \ \ \ \ \ \ \ \ \ \ \text{\ in }B_{\rho_{k}^{M}}\left(
x_{k}\right)  ,}{\partial_{v}^{i}u_{k}^{\rho_{k}^{M}}=\partial_{v}^{i}%
u_{k}\ \text{on }\partial B_{\rho_{k}^{M}}\left(  x_{k}\right)  ,i=0,1.}%
%EndExpansion
\]
We consider the function
\[
u_{k}^{M}=%
%TCIMACRO{\QDATOPD{\{}{.}{u_{k}^{\rho_{k}^{M}}\ \ \ \text{\ in }B_{\rho_{k}%
%^{M}}\left(  x_{k}\right)  ,}{u_{k}\ \ \text{in }\U{211d} ^{4}\backslash
%B_{\rho_{k}^{M}}\left(  x_{k}\right)  .}}%
%BeginExpansion
\genfrac{\{}{.}{0pt}{0}{u_{k}^{\rho_{k}^{M}}\ \ \ \text{\ in }B_{\rho_{k}^{M}%
}\left(  x_{k}\right)  ,}{u_{k}\ \ \text{in }\mathbb{R} ^{4}\backslash
B_{\rho_{k}^{M}}\left(  x_{k}\right)  .}%
%EndExpansion
\]

\begin{lemma}
\cite[Lemma 4.21]{Dela} For any $M>1$, we have
\[
u_{k}^{M}=\frac{c_{k}}{M}+O\left(  c_{k}^{-1}\right)  ,
\]
uniformly on $\overline{B_{\rho_{k}^{M}}\left(  x_{k}\right)  }$.
\end{lemma}

\begin{lemma}
\label{cufoff mod}For any $M>1$, there holds
\[
\underset{k\rightarrow\infty}{\lim\sup}\int_{%
%TCIMACRO{\U{211d} }%
%BeginExpansion
\mathbb{R}
%EndExpansion
^{4}}\left(  \left\vert \Delta u_{k}^{M}\right\vert ^{2}+\left\vert u_{k}%
^{M}\right\vert ^{2}\right)  dx\leq\frac{1}{M}.
\]

\end{lemma}

\begin{proof}
Since $u_{k}$ converges in $L^{p}\left(  B_{1}\right)  $ for any $p>1$, by
Lemma \ref{truncation}, we have
\[
\int_{B_{\rho_{k}^{M}}(x_{k})}\left\vert u_{k}^{M}\right\vert ^{p}dx\leq
c\int_{B_{\rho_{k}^{M}}(x_{k})}\left\vert u_{k}\right\vert ^{p}dx\rightarrow0
\]
and
\begin{equation}
\int_{B_{\rho_{k}^{M}}(x_{k})}\left\vert u_{k}^{p}u_{k}^{M}\right\vert
dx\leq\int_{B_{\rho_{k}^{M}}(x_{k})}\left\vert u_{k}^{p}\left(  u_{k}+O\left(
c_{k}^{-1}\right)  \right)  \right\vert dx\rightarrow0, \label{test1}%
\end{equation}
as $k\rightarrow\infty$.

Testing (\ref{euler}) with $\left(  u_{k}-u_{k}^{M}\right)  $, by Lemma
\ref{truncation}, for any $R>0,$ we have
\begin{align*}
&  \int_{B_{\rho_{k}^{M}}(x_{k})}\left(  \Delta u_{k}\Delta\left(  u_{k}%
-u_{k}^{M}\right)  +u_{k}\left(  u_{k}-u_{k}^{M}\right)  \right)  dx\\
&  =\int_{B_{\rho_{k}^{M}}(x_{k})}\lambda_{k}^{-1}u_{k}\left(  \exp\left\{
\beta_{k}u_{k}^{2}\right\}  -\frac{\alpha}{\beta_{k}}\right)  \left(
u_{k}-u_{k}^{M}\right)  dx\\
&  \geq\int_{B_{Rr_{k}}(x_{k})}\lambda_{k}^{-1}u_{k}\exp\left\{  \beta
_{k}u_{k}^{2}\right\}  \left(  u_{k}-u_{k}^{M}\right)  dx\\
&  =\int_{B_{Rr_{k}}(x_{k})}\lambda_{k}^{-1}c_{k}\exp\left\{  \beta_{k}%
u_{k}^{2}\right\}  \left(  c_{k}-\frac{c_{k}}{M}\right)  dx+o_{k}\left(
1\right) \\
&  =\int_{B_{R}}\left(  1-\frac{1}{M}\right)  \exp\left\{  2\beta_{k}\psi_{k}\left(  x\right)  \right\}  dx+o_{k}\left(  1\right)
\\
&  \geq\left(  1-\frac{1}{M}\right)  \int_{B_{R}}\exp\left\{  64\pi^{2}%
\psi\left(  x\right)  \right\}  dx+o_{k}\left(  1\right)  .
\end{align*}
Letting $R\rightarrow\infty$, we get
\begin{equation}
\int_{B_{\rho_{k}^{M}}(x_{k})}\left(  \Delta u_{k}\Delta\left(  u_{k}%
-u_{k}^{M}\right)  +u_{k}\left(  u_{k}-u_{k}^{M}\right)  \right)
dx\geq1-\frac{1}{M}+o_{k}\left(  1\right)  . \label{inside}%
\end{equation}
Observe that%
\begin{align*}
&  \int_{%
%TCIMACRO{\U{211d} }%
%BeginExpansion
\mathbb{R}
%EndExpansion
^{4}}\left(  \left\vert \Delta u_{k}^{M}\right\vert ^{2}+\left\vert u_{k}%
^{M}\right\vert ^{2}\right)  dx\\
&  =\int_{B_{\rho_{k}^{M}}(x_{k})}\left\vert \Delta u_{k}^{M}\right\vert
^{2}dx+\int_{%
%TCIMACRO{\U{211d} }%
%BeginExpansion
\mathbb{R}
%EndExpansion
^{4}\backslash B_{\rho_{k}^{M}}(x_{k})}\left\vert \Delta u_{k}\right\vert
^{2}dx+\int_{B_{\rho_{k}^{M}}(x_{k})}\left\vert u_{k}^{M}\right\vert
^{2}dx+\int_{%
%TCIMACRO{\U{211d} }%
%BeginExpansion
\mathbb{R}
%EndExpansion
^{4}\backslash B_{\rho_{k}^{M}}(x_{k})}\left\vert u_{k}\right\vert ^{2}dx\\
&  =\int_{B_{\rho_{k}^{M}}(x_{k})}\left\vert \Delta u_{k}^{M}\right\vert
^{2}dx+\int_{B_{\rho_{k}^{M}}(x_{k})}\left\vert u_{k}^{M}\right\vert
^{2}dx+1-\int_{B_{\rho_{k}^{M}}(x_{k})}\left\vert \Delta u_{k}\right\vert
^{2}dx-\int_{B_{\rho_{k}^{M}}(x_{k})}\left\vert u_{k}\right\vert ^{2}dx\\
&  =\int_{B_{\rho_{k}^{M}}(x_{k})}\left\vert \Delta u_{k}^{M}\right\vert
^{2}dx+\int_{B_{\rho_{k}^{M}}(x_{k})}\left\vert u_{k}^{M}\right\vert
^{2}dx+1-\int_{B_{\rho_{k}^{M}}(x_{k})}\Delta u_{k}\Delta\left(  u_{k}%
-u_{k}^{M}\right)  dx\\
&  -\int_{B_{\rho_{k}^{M}}(x_{k})}u_{k}\left(  u_{k}-u_{k}^{M}\right)
dx-\int_{B_{\rho_{k}^{M}}(x_{k})}\Delta u_{k}\Delta u_{k}^{M}dx-\int
_{B_{\rho_{k}^{M}}(x_{k})}u_{k}u_{k}^{M}dx,
\end{align*}
by (\ref{inside}) and (\ref{test1}), we have
\begin{align*}
&  \int_{%
%TCIMACRO{\U{211d} }%
%BeginExpansion
\mathbb{R}
%EndExpansion
^{4}}\left(  \left\vert \Delta u_{k}^{M}\right\vert ^{2}+\left\vert u_{k}%
^{M}\right\vert ^{2}\right)  dx\\
&  \leq\frac{1}{M}+\int_{B_{\rho_{k}^{M}}(x_{k})}\left\vert \Delta u_{k}%
^{M}\right\vert ^{2}dx+\int_{B_{\rho_{k}^{M}}(x_{k})}\left\vert u_{k}%
^{M}\right\vert ^{2}dx-\int_{B_{\rho_{k}^{M}}(x_{k})}\Delta u_{k}\Delta
u_{k}^{M}dx-\int_{B_{\rho_{k}^{M}}(x_{k})}u_{k}u_{k}^{M}dx\\
&  \leq\frac{1}{M}+\int_{B_{\rho_{k}^{M}}(x_{k})}\Delta u_{k}^{M}\Delta\left(
u_{k}^{M}-u_{k}\right)  dx+\int_{B_{\rho_{k}^{M}}(x_{k})}\left\vert u_{k}%
^{M}\right\vert ^{2}dx-\int_{B_{\rho_{k}^{M}}(x_{k})}u_{k}u_{k}^{M}dx\\
&  =\frac{1}{M}+\int_{B_{\rho_{k}^{M}}(x_{k})}\left\vert u_{k}^{M}\right\vert
^{2}dx-\int_{B_{\rho_{k}^{M}}(x_{k})}u_{k}u_{k}^{M}dx\\
&  =\frac{1}{M}+o_{k}\left(  1\right)  .
\end{align*}
Hence the lemma is proved.
\end{proof}

With the help of bi-harmonic truncations $u_{k}^{M}$, we can show the
following result.

\begin{corollary}
\bigskip We have,
\[
\underset{k\rightarrow\infty}{\lim\sup}\int_{%
%TCIMACRO{\U{211d} }%
%BeginExpansion
\mathbb{R}
%EndExpansion
^{4}\backslash B_{\delta}}\left(  \left\vert u_{k}\right\vert ^{2}+\left\vert
\Delta u_{k}\right\vert ^{2}\right)  dx=0,
\]
for any $\delta>0$, and then
\begin{equation}
\left\vert \Delta u_{k}\right\vert ^{2}dx\rightharpoonup\delta_{0} \text{ in
the sense of measure,} \label{dirc m}%
\end{equation}
where $\delta_{0} $ is the Dirac measure supported at $0$.
\end{corollary}

\begin{lemma}
\label{concentre}We have
\begin{align*}
&  \underset{k\rightarrow\infty}{\lim}\int_{%
%TCIMACRO{\U{211d} }%
%BeginExpansion
\mathbb{R}
%EndExpansion
^{4}}\left(  \exp\left(  \beta_{k}u_{k}^{2}\right)  -1-\alpha u_{k}%
^{2}\right)  dx\\
&  =\underset{L\rightarrow\infty}{\lim}\underset{k\rightarrow\infty}{\lim}%
\int_{B_{Lr_{k}}(x_{k})}\left(  \exp\left(  \beta_{k}u_{k}^{2}\right)
-1-\alpha u_{k}^{2}\right)  dx\\
&  =\underset{k\rightarrow\infty}{\lim}\frac{\lambda_{k}}{c_{k}^{2}}%
\end{align*}
and consequently,
\[
\frac{\lambda_{k}}{c_{k}}\rightarrow\infty\text{ and }\underset{k}{\sup}%
\frac{c_{k}^{2}}{\lambda_{k}}<\infty.
\]

\end{lemma}

\begin{proof}
Direct computations yield that
\begin{align*}
&  \int_{%
%TCIMACRO{\U{211d} }%
%BeginExpansion
\mathbb{R}
%EndExpansion
^{4}}\left(  \exp\left(  \beta_{k}u_{k}^{2}\right)  -1-\alpha u_{k}%
^{2}\right)  dx\\
&  =\left(  \int_{B_{\rho_{k}^{M}}(x_{k})}+\int_{%
%TCIMACRO{\U{211d} }%
%BeginExpansion
\mathbb{R}
%EndExpansion
^{4}\backslash B_{\rho_{k}^{M}}(x_{k})}\right)  \left(  \exp\left(  \beta
_{k}u_{k}^{2}\right)  -1-\alpha u_{k}^{2}\right)  dx\\
&  \leq\int_{B_{\rho_{k}^{M}}(x_{k})}\left(  \exp\left(  \beta_{k}u_{k}%
^{2}\right)  -1-\alpha u_{k}^{2}\right)  dx+\int_{%
%TCIMACRO{\U{211d} }%
%BeginExpansion
\mathbb{R}
%EndExpansion
^{4}}\left(  \exp\left(  \beta_{k}\left(  u_{k}^{M}\right)  ^{2}\right)
-1-\alpha\left(  u_{k}^{M}\right)  ^{2}\right)  dx.
\end{align*}
Taking some $L$ such that $u_{k}\leq1$ on $%
%TCIMACRO{\U{211d} }%
%BeginExpansion
\mathbb{R}
%EndExpansion
^{4}\backslash B_{L}$, then we have
\[
\underset{k\rightarrow\infty}{\lim}\int_{%
%TCIMACRO{\U{211d} }%
%BeginExpansion
\mathbb{R}
%EndExpansion
^{4}\backslash B_{L}}\left(  \exp\left(  \beta_{k}u_{k}^{2}\right)  -1-\alpha
u_{k}^{2}\right)  dx\leq\underset{k\rightarrow\infty}{\lim}c\int_{%
%TCIMACRO{\U{211d} }%
%BeginExpansion
\mathbb{R}
%EndExpansion
^{4}}u_{k}^{2}dx=0.
\]

In view of Lemma \ref{cufoff mod} and the Adams' inequality with the Navier
boundary condition (see \cite{Tar}),\ we obtain
\[
\underset{k\rightarrow\infty}{\sup}\int_{B_{L}}\left(  \exp\left(  \beta
_{k}p^{\prime}\left(  u_{k}^{M}-u_{k}\left(  L\right)  \right)  ^{2}\right)
-1\right)  dx<\infty,
\]
for any $p^{\prime}<M$. Since \
\[
p\left(  u_{k}^{M}\right)  ^{2}\leq p^{\prime}\left(  u_{k}^{M}-u_{k}\left(
L\right)  \right)  ^{2}+c\left(  p,p^{\prime}\right)  ,\text{if }p<p^{\prime
},
\]
then we get%
\[
\underset{k\rightarrow\infty}{\sup}\int_{B_{L}}\left(  \exp\left(  \beta
_{k}p\left(  u_{k}^{M}\right)  ^{2}\right)  -1\right)  dx<\infty,
\]
for any $p<M$. The weak compactness of Banach space implies
\[
\underset{k\rightarrow\infty}{\lim}\int_{B_{L}}\left(  \exp\left(  \beta
_{k}\left(  u_{k}^{M}\right)  ^{2}\right)  -1\right)  dx=0.
\]
Hence, we get
\begin{equation}%
\begin{split}
&  \underset{k\rightarrow\infty}{\lim}\int_{%
%TCIMACRO{\U{211d} }%
%BeginExpansion
\mathbb{R}
%EndExpansion
^{4}}\left(  \exp\left(  \beta_{k}u_{k}^{2}\right)  -1-\alpha u_{k}%
^{2}\right)  dx\\
&  =\underset{k\rightarrow\infty}{\lim}\left(  \int_{B_{\rho_{k}^{M}}(x_{k}%
)}\left(  \exp\left(  \beta_{k}u_{k}^{2}\right)  -1-\alpha u_{k}^{2}\right)
dx+o_{k}\left(  1\right)  \right) \\
&  \leq\underset{k\rightarrow\infty}{\lim}M^{2}\frac{\lambda_{k}}{c_{k}^{2}%
}\int_{%
%TCIMACRO{\U{211d} }%
%BeginExpansion
\mathbb{R}
%EndExpansion
^{4}}\frac{u_{k}^{2}}{\lambda_{k}}\left(  \exp\left(  \beta_{k}u_{k}%
^{2}\right)  -\frac{\alpha}{\beta_{k}}\right)  dx\\
&  =M^{2}\underset{k\rightarrow\infty}{\lim}\frac{\lambda_{k}}{c_{k}^{2}}.
\end{split}
\label{gj1}%
\end{equation}
On the other hand, we get
\begin{equation}%
\begin{split}
&  \underset{L\rightarrow\infty}{\lim}\underset{k\rightarrow\infty}{\lim}%
\int_{B_{Lr_{k}}(x_{k})}\left(  \exp\left(  \beta_{k}u_{k}^{2}\right)
-1-\alpha u_{k}^{2}\right)  dx\\
&  \ \ =\underset{L\rightarrow\infty}{\lim}\underset{k\rightarrow\infty}{\lim
}\frac{\lambda_{k}}{c_{k}^{2}}\int_{B_{L}}\exp\left(  \beta_{k}u_{k}^{2}%
(r_{k}x+x_{k})-\beta_{k}c_{k}^{2}\right)  dx\\
&  \ \ =\underset{k\rightarrow\infty}{\lim}\frac{\lambda_{k}}{c_{k}^{2}%
}\left(  \int_{%
%TCIMACRO{\U{211d} }%
%BeginExpansion
\mathbb{R}
%EndExpansion
^{4}}\exp\left(  64\pi^{2}\psi\left(  x\right)  \right)  dx+o_{k}\left(
1\right)  \right)  \\
&  \ \ =\underset{k\rightarrow\infty}{\lim}\frac{\lambda_{k}}{c_{k}^{2}}.
\end{split}
\label{gj2}%
\end{equation}

Combining (\ref{gj1}) and (\ref{gj2}), and letting $M\rightarrow1$, we
conclude that
\begin{align*}
\underset{k\rightarrow\infty}{\lim}\int_{%
%TCIMACRO{\U{211d} }%
%BeginExpansion
\mathbb{R}
%EndExpansion
^{4}}\left(  \exp\left(  \beta_{k}u_{k}^{2}\right)  -1-\alpha u_{k}%
^{2}\right)  dx  &  =\underset{k\rightarrow\infty}{\lim}\frac{\lambda_{k}%
}{c_{k}^{2}}\\
&  =\underset{L\rightarrow\infty}{\lim}\underset{k\rightarrow\infty}{\lim}%
\int_{B_{Lr_{k}}(x_{k})}\left(  \exp\left(  \beta_{k}u_{k}^{2}\right)
-1-\alpha u_{k}^{2}\right)  dx.
\end{align*}

\end{proof}

Now, we introduce the following quantities:%

\begin{align*}
b_{k} &  =\underset{R\rightarrow\infty}{\lim}\underset{k\rightarrow\infty
}{\lim}\frac{\lambda_{k}}{\int_{B_{R}(x_{k})}\left\vert u_{k}\right\vert
\left(  \exp\left(  \beta_{k}u_{k}^{2}\right)  -\frac{\alpha}{\beta_{k}%
}\right)  dx}\text{, }\tau=\underset{k\rightarrow\infty}{\lim}\frac{b_{k}%
}{c_{k}}\text{ and}\\
\sigma &  =\underset{R\rightarrow\infty}{\lim}\underset{k\rightarrow\infty
}{\lim}\frac{\int_{B_{R}(x_{k})}u_{k}\left(  \exp\left(  \beta_{k}u_{k}%
^{2}\right)  -\frac{\alpha}{\beta_{k}}\right)  dx}{\int_{B_{R}(x_{k}%
)}\left\vert u_{k}\right\vert \left(  \exp\left(  \beta_{k}u_{k}^{2}\right)
-\frac{\alpha}{\beta_{k}}\right)  dx}\text{.}%
\end{align*}

\begin{lemma}
It holds $\sigma=1$.
\end{lemma}

\begin{proof}
For any $M>1$ and $R>0$, we have
\begin{align*}
\int_{B_{R}(x_{k})}u_{k}\left(  \exp\left(  \beta_{k}u_{k}^{2}\right)
-\frac{\alpha}{\beta_{k}}\right)  dx  &  =\int_{B_{\rho_{k}^{M}(x_{k})}}%
u_{k}\left(  \exp\left(  \beta_{k}u_{k}^{2}\right)  -\frac{\alpha}{\beta_{k}%
}\right)  dx\\
&  +\int_{B_{R}(x_{k})\backslash B_{\rho_{k}^{M}(x_{k})}}u_{k}^{M}\left(
\exp\left(  \beta_{k}\left(  u_{k}^{M}\right)  ^{2}\right)  -\frac{\alpha
}{\beta_{k}}\right)  dx,
\end{align*}
By Lemma \ref{cufoff mod}, we know that\ $\exp\left(  \beta_{k}u_{k}%
^{2}\right)  -\frac{\alpha}{\beta_{k}}$ is bounded in $L^{p}\left(
B_{R}(x_{k})\backslash B_{\rho_{k}^{M}(x_{k})}\right)  $ for some $p>1$,\ then
we have%
\[
\int_{B_{R}(x_{k})\backslash B_{\rho_{k}^{M}(x_{k})}}u_{k}^{M}\left(  \exp\left(
\beta_{k}\left(  u_{k}^{M}\right)  ^{2}\right)  -\frac{\alpha}{\beta_{k}%
}\right)  dx\rightarrow0,\text{ as }k\rightarrow\infty.
\]
This implies that
\begin{equation}
\int_{B_{R}(x_{k})}u_{k}\left(  \exp\left(  \beta_{k}u_{k}^{2}\right)
-\frac{\alpha}{\beta_{k}}\right)  dx=\int_{B_{\rho_{k}^{M}(x_{k})}}\left\vert
u_{k}\right\vert \left(  \exp\left(  \beta_{k}u_{k}^{2}\right)  -\frac{\alpha
}{\beta_{k}}\right)  dx+o_{k}\left(  1\right)  . \label{22}%
\end{equation}
Similarly, we also have
\begin{equation}
\int_{B_{R}(x_{k})}\left\vert u_{k}\right\vert \left(  \exp\left(  \beta
_{k}u_{k}^{2}\right)  -\frac{\alpha}{\beta_{k}}\right)  dx=\int_{B_{\rho
_{k}^{M}}(x_{k})}\left\vert u_{k}\right\vert \left(  \exp\left(  \beta
_{k}u_{k}^{2}\right)  -\frac{\alpha}{\beta_{k}}\right)  dx+o_{k}\left(
1\right)  . \label{23}%
\end{equation}
On the other hand, it is not hard to see that
\begin{align}
c_{k}\int_{B_{\rho_{k}^{M}(x_{k})}}\left(  \exp\left(  \beta_{k}u_{k}%
^{2}\right)  -\frac{\alpha}{\beta_{k}}\right)  dx  &  \geq\int_{B_{\rho
_{k}^{M}}(x_{k})}\left\vert u_{k}\right\vert \left(  \exp\left(  \beta
_{k}u_{k}^{2}\right)  -\frac{\alpha}{\beta_{k}}\right)  dx\nonumber\\
&  \geq\frac{c_{k}}{M}\int_{B_{\rho_{k}^{M}}(x_{k})}\left(  \exp\left(
\beta_{k}u_{k}^{2}\right)  -\frac{\alpha}{\beta_{k}}\right)  dx.
\label{esti 2}%
\end{align}
Furthermore, by (4) of Lemma \ref{truncation} and Lemma \ref{concentre}, we derive that
\begin{align}
\underset{k\rightarrow\infty}{\lim}\int_{B_{\rho_{k}^{M}(x_{k})}}\left(
\exp\left(  \beta_{k}u_{k}^{2}\right)  -\frac{\alpha}{\beta_{k}}\right)  dx
&  \geq\underset{L\rightarrow\infty}{\lim}\underset{k\rightarrow\infty}{\lim
}\int_{B_{Lr_{k}}(x_{k})}\left(  \exp\left(  \beta_{k}u_{k}^{2}\right)
-\frac{\alpha}{\beta_{k}}\right)  dx\nonumber\\
&  \geq\underset{L\rightarrow\infty}{\lim}\underset{k\rightarrow\infty}{\lim
}\int_{B_{Lr_{k}}(x_{k})}\left(  \exp\left(  \beta_{k}u_{k}^{2}\right)
-1-\alpha u_{k}^{2}\right)  dx\nonumber\\
&  \geq\int_{%
%TCIMACRO{\U{211d} }%
%BeginExpansion
\mathbb{R}
%EndExpansion
^{4}}\left(  \exp\left(  \beta_{k}u_{k}^{2}\right)  -1-\alpha u_{k}%
^{2}\right)  dx>0. \label{25}%
\end{align}

Therefore, combining (\ref{22})-(\ref{25}), we get%
\[
\frac{1}{M}+o_{k}(1)\leq\frac{\int_{B_{R}(x_{k})}u_{k}\left(  \exp\left(
\beta_{k}u_{k}^{2}\right)  -\frac{\alpha}{\beta_{k}}\right)  dx}{\int
_{B_{R}(x_{k})}\left\vert u_{k}\right\vert \left(  \exp\left(  \beta_{k}%
u_{k}^{2}\right)  -\frac{\alpha}{\beta_{k}}\right)  dx}\leq1.
\]
Letting $k\rightarrow\infty,$ $R\rightarrow\infty$ and $M\rightarrow1$, we
derive that $\sigma=1$.
\end{proof}

\begin{lemma}
\label{tau} It holds $\tau=1$.
\end{lemma}

\begin{proof}
For any $R>0$, similar to (\ref{esti 2}), we have
\begin{align*}
\int_{%
%TCIMACRO{\U{211d} }%
%BeginExpansion
\mathbb{R}
%EndExpansion
^{4}}\left\vert u_{k}\right\vert ^{2}\left(  \exp\left(  \beta_{k}u_{k}%
^{2}\right)  -\frac{\alpha}{\beta_{k}}\right)  dx  &  =\int_{B_{\rho_{k}%
^{M}(x_{k})}}\left\vert u_{k}\right\vert ^{2}\left(  \exp\left(  \beta
_{k}u_{k}^{2}\right)  -\frac{\alpha}{\beta_{k}}\right)  dx+o_{k}\left(
1\right) \\
&  \geq\left(  \frac{c_{k}}{M}\right)  ^{2}\int_{B_{\rho_{k}^{M}}(x_{k}%
)}\left(  \exp\left(  \beta_{k}u_{k}^{2}\right)  -\frac{\alpha}{\beta_{k}%
}\right)  dx+o_{k}(1)
\end{align*}
and%
\[
\int_{B_{R}(x_{k})}\left\vert u_{k}\right\vert \left(  \exp\left(  \beta
_{k}u_{k}^{2}\right)  -\frac{\alpha}{\beta_{k}}\right)  dx\leq c_{k}%
\int_{B_{\rho_{k}^{M}}(x_{k})}\left(  \exp\left(  \beta_{k}u_{k}^{2}\right)
-\frac{\alpha}{\beta_{k}}\right)  dx+o_{k}(1).
\]
Thus,
\begin{align*}
b_{k}  &  =\underset{R\rightarrow\infty}{\lim}\frac{\int_{%
%TCIMACRO{\U{211d} }%
%BeginExpansion
\mathbb{R}
%EndExpansion
^{4}}\left\vert u_{k}\right\vert ^{2}\left(  \exp\left(  \beta_{k}u_{k}%
^{2}\right)  -\frac{\alpha}{\beta_{k}}\right)  dx}{\int_{B_{R}(x_{k}%
)}\left\vert u_{k}\right\vert \left(  \exp\left(  \beta_{k}u_{k}^{2}\right)
-\frac{\alpha}{\beta_{k}}\right)  dx}\\
&  \geq\frac{\left(  \frac{c_{k}}{M}\right)  ^{2}}{c_{k}}=\frac{c_{k}}{M^{2}},
\end{align*}
letting $M\rightarrow1$, we conclude that $\tau=1$.
\medskip

\end{proof}
\subsection{\bigskip Asymptotic behavior of $u_{k}$ away from the blow-up
point $0$}

In the following, we consider the asymptotic behavior of $u_{k}$ away from the
blow-up point $0$.

We recall that the crucial tool in studying the regularity of higher order
equations is the fundamental solution\emph{ of the operator }$\Delta^{2}%
+1$). The fundamental solution $\Gamma\left(  x,y\right)  $ for
$\Delta^{2}+1$ in $%
%TCIMACRO{\U{211d} }%
%BeginExpansion
\mathbb{R}
%EndExpansion
^{4}$ is the solution of
\[
\left(  \Delta^{2}+1\right)  \Gamma\left(  x,y\right)  =\delta_{x}\left(
y\right)  \text{ in }%
%TCIMACRO{\U{211d} }%
%BeginExpansion
\mathbb{R}
%EndExpansion
^{4}\text{,}%
\]
and all functions $u\in H^{2}\left(
%TCIMACRO{\U{211d} }%
%BeginExpansion
\mathbb{R}
%EndExpansion
^{4}\right)  \cap C^{4}\left(
%TCIMACRO{\U{211d} }%
%BeginExpansion
\mathbb{R}
%EndExpansion
^{4}\right)  $ satisfying $\left(  \Delta^{2}+1\right)  u=f$ can be
represented by%
\begin{equation}
u\left(  x\right)  =\int_{%
%TCIMACRO{\U{211d} }%
%BeginExpansion
\mathbb{R}
%EndExpansion
^{4}}\Gamma\left(  x,y\right)  f\left(  y\right)  dy. \label{representation}%
\end{equation}

We will need the following useful estimates for $\Gamma:$%

\begin{equation}
\left\vert \Gamma\left(  x,y\right)  \right\vert \leq c\ln\left(  1+\left\vert
x-y\right\vert ^{-1}\right)  ,\left\vert \nabla^{i}\Gamma\left(  x,y\right)
\right\vert \leq c\left\vert x-y\right\vert ^{-i}\text{, }i\geq1
\label{Green esti}%
\end{equation}
for all $x,y\in%
%TCIMACRO{\U{211d} }%
%BeginExpansion
\mathbb{R}
%EndExpansion
^{4}$, $x\neq y$ with $\left\vert x-y\right\vert \rightarrow0$, and%

\begin{equation}
\nabla^{i}\left\vert \Gamma\left(  x,y\right)  \right\vert =o\left(
\exp\left(  -\frac{1}{\sqrt{2}}\left\vert x-y\right\vert \right)
\right)  \,,i=0,1,2. \label{Green lage}%
\end{equation}
for all $x,y\in%
%TCIMACRO{\U{211d} }%
%BeginExpansion
\mathbb{R}
%EndExpansion
^{4}$, with $\left\vert x-y\right\vert \rightarrow+\infty$.

The above properties of $\Gamma$ can be found in \cite{Deng}.

\begin{lemma}
\label{bound for bu}\bigskip
For any $1<r<2$, $c_{k}u_{k}$ is bounded in
$W^{2,r}\left(
%TCIMACRO{\U{211d} }%
%BeginExpansion
\mathbb{R}
%EndExpansion
^{4}\right)$.
\end{lemma}

\begin{remark}
It is quite difficult to prove $\left\Vert c_{k}u_{k}\right\Vert
_{W^{2,r}\left(
%TCIMACRO{\U{211d} }%
%BeginExpansion
\mathbb{R}
%EndExpansion
^{4}\right)  }\leq c$ directly. But as we have showed the fact $\underset
{k\rightarrow\infty}{\lim}\frac{c_{k}}{b_{k}}=1$ in Lemma \ref{tau}, we will
prove this lemma by showing that $\left\Vert b_{k}u_{k}\right\Vert
_{W^{2,r}\left(
%TCIMACRO{\U{211d} }%
%BeginExpansion
\mathbb{R}
%EndExpansion
^{4}\right)  }\leq c$. We find it quite easy to obtain the desired result for
$b_{k}u_{k}$.
\end{remark}

\begin{proof}
Let $\eta_{k}$ be the solution of
\[
\Delta^{2}\eta_{k}+\eta_{k}=\frac{b_{k}u_{k}}{\lambda_{k}}\left(  \exp\left\{
\beta_{k}u_{k}^{2}\right\}  -\frac{\alpha}{\beta_{k}}\right)  ,\ \ x\in%
%TCIMACRO{\U{211d} }%
%BeginExpansion
\mathbb{R}
%EndExpansion
^{4}.
\]
By the representation formula, we have
\[
\eta_{k}\left(  x\right)  =\int_{%
%TCIMACRO{\U{211d} }%
%BeginExpansion
\mathbb{R}
%EndExpansion
^{4}}\Gamma\left(  x,y\right)  \frac{b_{k}u_{k}\left(  y\right)  }{\lambda
_{k}}\left(  \exp\left\{  \beta_{k}u_{k}^{2}\left(  y\right)  \right\}
-\frac{\alpha}{\beta_{k}}\right)  dy.
\]
Then, by H\"{o}rder's inequality, for any $1<r<2$, we have

\begin{align*}
\left\vert \nabla^{i}\eta_{k}\left(  x\right)  \right\vert ^{r} &  =\left(
\frac{b_{k}}{\lambda_{k}}\int_{%
%TCIMACRO{\U{211d} }%
%BeginExpansion
\mathbb{R}
%EndExpansion
^{4}}\nabla^{i}\Gamma\left(  x,y\right)  u_{k}\left(  y\right)  \left(
\exp\left\{  \beta_{k}u_{k}^{2}\left(  y\right)  \right\}  -\frac{\alpha
}{\beta_{k}}\right)  dy\right)  ^{r}\\
&  \leq\left(  \int_{%
%TCIMACRO{\U{211d} }%
%BeginExpansion
\mathbb{R}
%EndExpansion
^{4}}\nabla^{i}\Gamma\left(  x,y\right)  \frac{u_{k}\left(  y\right)  \left(
\exp\left\{  \beta_{k}u_{k}^{2}\left(  y\right)  \right\}  -\frac{\alpha
}{\beta_{k}}\right)  }{\int_{%
%TCIMACRO{\U{211d} }%
%BeginExpansion
\mathbb{R}
%EndExpansion
^{4}}u_{k}\left(  z\right)  \left(  \exp\left\{  \beta_{k}u_{k}^{2}\left(
z\right)  \right\}  dz-\frac{\alpha}{\beta_{k}}\right)  dz}dy\right)  ^{r}\\
&  \leq\int_{%
%TCIMACRO{\U{211d} }%
%BeginExpansion
\mathbb{R}
%EndExpansion
^{4}}\left\vert \nabla^{i}\Gamma\left(  x,y\right)  \right\vert ^{r}%
\frac{u_{k}\left(  y\right)  \left(  \exp\left\{  \beta_{k}u_{k}^{2}\left(
y\right)  \right\}  -\frac{\alpha}{\beta_{k}}\right)  }{\int_{%
%TCIMACRO{\U{211d} }%
%BeginExpansion
\mathbb{R}
%EndExpansion
^{4}}u_{k}\left(  z\right)  \left(  \exp\left\{  \beta_{k}u_{k}^{2}\left(
z\right)  \right\}  dz-\frac{\alpha}{\beta_{k}}\right)  dz}dy,
\end{align*}
where $i=0,1,2$. Applying Fubini's theorem, we get by (\ref{Green esti}) and
(\ref{Green lage}) that
\begin{align*}
\int_{%
%TCIMACRO{\U{211d} }%
%BeginExpansion
\mathbb{R}
%EndExpansion
^{4}}\left\vert \nabla^{i}\eta_{k}\left(  x\right)  \right\vert ^{r}dx &
=\int_{%
%TCIMACRO{\U{211d} }%
%BeginExpansion
\mathbb{R}
%EndExpansion
^{4}}\int_{%
%TCIMACRO{\U{211d} }%
%BeginExpansion
\mathbb{R}
%EndExpansion
^{4}}\left\vert \nabla^{i}\Gamma\left(  x,y\right)  \right\vert ^{r}%
dx\frac{u_{k}\left(  y\right)  \left(  \exp\left\{  \beta_{k}u_{k}^{2}\left(
y\right)  \right\}  -\frac{\alpha}{\beta_{k}}\right)  }{\int_{%
%TCIMACRO{\U{211d} }%
%BeginExpansion
\mathbb{R}
%EndExpansion
^{4}}u_{k}\left(  z\right)  \left(  \exp\left\{  \beta_{k}u_{k}^{2}\left(
z\right)  \right\}  dz-\frac{\alpha}{\beta_{k}}\right)  dz}dy\\
&  \leq c\text{, for }i=0,1,2,
\end{align*}
thus, we get
\begin{equation}
\left\Vert \eta_{k}\right\Vert _{W^{2,r}\left(
%TCIMACRO{\U{211d} }%
%BeginExpansion
\mathbb{R}
%EndExpansion
^{4}\right)  }<c. \label{bound}%
\end{equation}
Let $\eta_{k}=b_{k}u_{k}$, then $\eta_{k}$ satisfies%
\[
\Delta^{2}\eta_{k}+\eta_{k}=\frac{b_{k}u_{k}}{\lambda_{k}}\left(
\exp\left\{  \beta_{k}u_{k}^{2}\right\}  -\frac{\alpha}{\beta_{k}}\right)
\text{ in }%
%TCIMACRO{\U{211d} }%
%BeginExpansion
\mathbb{R}
%EndExpansion
^{4}.
\]
By (\ref{bound}),\ we have $\left\Vert \eta_{k}\right\Vert _{W^{2,r}\left(
%TCIMACRO{\U{211d} }%
%BeginExpansion
\mathbb{R}
%EndExpansion
^{4}\right)  }<c$. This accomplishes the proof of Lemma \ref{bound for bu}.
\end{proof}

Now, we will show that $c_{k}u_{k}$ converges to some Green function.

\begin{lemma}\label{Dircm}
For any $\varphi\in C_{0}^{\infty}\left(
%TCIMACRO{\U{211d} }%
%BeginExpansion
\mathbb{R}
%EndExpansion
^{4}\right)  $, one has
\begin{equation}
\underset{k\rightarrow\infty}{\lim}\int_{%
%TCIMACRO{\U{211d} }%
%BeginExpansion
\mathbb{R}
%EndExpansion
^{4}}\varphi\left(  x\right)  \frac{c_{k}u_{k}}{\lambda_{k}}\left(
\exp\left(  \beta_{k}u_{k}^{2}\right)  -\frac{\alpha}{\beta_{k}}\right)
dx=\varphi\left(  0\right)  .\ \label{concentra}%
\end{equation}

\end{lemma}

\begin{proof}
Suppose $supp\ \varphi\subset B_{\rho}$ and we spit the integral as follows%
\begin{align}
&  \int_{%
%TCIMACRO{\U{211d} }%
%BeginExpansion
\mathbb{R}
%EndExpansion
^{4}}\varphi\left(  x\right)  \frac{c_{k}u_{k}}{\lambda_{k}}\left(
\exp\left(  \beta_{k}u_{k}^{2}\right)  -\frac{\alpha}{\beta_{k}}\right)
dx\nonumber\\
&  =\int_{B_{\rho_{k}^{M}}(x_{k})\setminus B_{Lr_{k}}}\varphi\left(  x\right)
\frac{c_{k}u_{k}}{\lambda_{k}}\left(  \exp\left(  \beta_{k}u_{k}^{2}\right)
-\frac{\alpha}{\beta_{k}}\right)  dx\nonumber\\
&  +\int_{B_{Lr_{k}}}\varphi\left(  x\right)  \frac{c_{k}u_{k}}{\lambda_{k}%
}\left(  \exp\left(  \beta_{k}u_{k}^{2}\right)  -\frac{\alpha}{\beta_{k}%
}\right)  dx\nonumber\\
&  +\int_{B_{\rho}\setminus B_{\rho_{k}^{M}}(x_{k})}\varphi\left(  x\right)
\frac{c_{k}u_{k}}{\lambda_{k}}\left(  \exp\left(  \beta_{k}u_{k}^{2}\right)
-\frac{\alpha}{\beta_{k}}\right)  dx\nonumber\\
&  =I_{1}^{k}+I_{2}^{k}+I_{3}^{k}.
\end{align}
For $I_{1}^{k}$, it follows that
\begin{equation}
\label{g1}%
\begin{split}
I_{1}^{k}  &  \leq M \|\varphi\|_{C^{0}}\int_{B_{\rho_{k}^{M}}(x_{k})\setminus
B_{Lr_{k}}}  \frac{u^{2}_{k}}{\lambda_{k}}\left(
\exp\left(  \beta_{k}u_{k}^{2}\right)  -\frac{\alpha}{\beta_{k}}\right)  dx\\
&  =M \|\varphi\|_{C^{0}}\Big(1-\int_{B_{L}}\exp\left(  2\beta_{k}\psi
_{k}\left(  x\right)  +o_{k}(1) \right)  dx\Big).
\end{split}
\end{equation}
Letting $k\rightarrow+\infty$ and $L\rightarrow+\infty$, we derive that
$\underset{k\rightarrow\infty}{\lim}I_{1}^{k}=0$.

For $I_{2}^{k}$, we have
\begin{equation}
\label{g2}%
\begin{split}
I_{2}^{k}  &  =\int_{B_{L}}\varphi(r_{k}x+x_k) \frac{u_{k}(r_{k}x+x_k)}{c_{k}}%
\exp\left(  2\beta_{k}\psi_{k}\left(  x\right)  +o_{k}(1) \right)  dx.\\
\end{split}
\end{equation}
Letting $k\rightarrow+\infty$ and $L\rightarrow+\infty$, we derive that
$\underset{k\rightarrow\infty}{\lim}I_{2}^{k}=\varphi(0)$.

For $I_{3}^{k}$, since $\exp(\beta_{k}|u_{k}^{M}|^{2})$ is bounded in
$L^{p}(B_{\rho})$ for some $p>1$, choosing $p>1$ sufficiently close to $1$ and
by H\"{o}lder's inequality, we derive that
\begin{equation}
\label{g3}%
\begin{split}
I_{3}^{k}\leq\frac{c_{k}}{\lambda_{k}}\|\varphi\|_{C^{0}}\big(\int_{B_{\rho}%
}|u_{k}|^{p^{\prime}}dx\big)^{\frac{1}{p^{\prime}}}\big(\int_{B_{\rho}}\exp(\beta
_{k}p|u_{k}^{M}|^{2})dx\big)^{\frac{1}{p}}.
\end{split}
\end{equation}
Note that $\underset{k\rightarrow\infty}{\lim}\frac{c_{k}}{\lambda_{k}}=0$,
hence $\underset{k\rightarrow\infty}{\lim}I_{3}^{k}=0$. Combining \eqref{g1},
\eqref{g2} and \eqref{g3}, we conclude that
\[
\underset{k\rightarrow\infty}{\lim}\int_{%
%TCIMACRO{\U{211d} }%
%BeginExpansion
\mathbb{R}
%EndExpansion
^{4}}\varphi\left(  x\right)  \frac{c_{k}u_{k}}{\lambda_{k}}\left(
\exp\left(  \beta_{k}u_{k}^{2}\right)  -\frac{\alpha}{\beta_{k}}\right)
dx=\varphi\left(  0\right)  .\ \label{concentra}%
\]

\end{proof}

\begin{lemma}
\label{esti for the out of the ball}For any $1<r<2$, $c_{k}u_{k}%
\rightharpoonup G\in C^{3}\left(
%TCIMACRO{\U{211d} }%
%BeginExpansion
\mathbb{R}
%EndExpansion
^{4}\right)  \backslash\left\{  0\right\}  $ weakly in $W^{2,r}\left(
%TCIMACRO{\U{211d} }%
%BeginExpansion
\mathbb{R}
%EndExpansion
^{4}\right)  $, where $G$ is a Green function satisfying
\[
\Delta^{2}G+G=\delta\left(  x\right)  \text{ in }%
%TCIMACRO{\U{211d} }%
%BeginExpansion
\mathbb{R}
%EndExpansion
^{4}.
\]
Also, we have%
\begin{equation}
G=-\frac{1}{8\pi^{2}}\ln\left\vert x\right\vert +A+\varphi\left(  x\right)
,\label{Green}%
\end{equation}
where $A$ is a constant depending on $0$, $\varphi\left(  x\right)  \in
C^{3}\left(
%TCIMACRO{\U{211d} }%
%BeginExpansion
\mathbb{R}
%EndExpansion
^{4}\right)  $ and $\varphi\left(  0\right)  =0$. Moreover, we have
\begin{align*}
&  \underset{k\rightarrow\infty}{\lim}\left(  \int_{%
%TCIMACRO{\U{211d} }%
%BeginExpansion
\mathbb{R}
%EndExpansion
^{4}\backslash B_{\delta}\left(  0\right)  }\left\vert \Delta\left(
c_{k}u_{k}\right)  \right\vert ^{2}dx+\int_{%
%TCIMACRO{\U{211d} }%
%BeginExpansion
\mathbb{R}
%EndExpansion
^{4}\backslash B_{\delta}\left(  0\right)  }\left\vert c_{k}u_{k}\right\vert
^{2}dx\right) \\
&  =-\frac{1}{8\pi^{2}}\ln\delta-\frac{1}{16\pi^{2}}+A+O\left(  \delta\right)
.
\end{align*}
\end{lemma}

\begin{proof}
By Lemma \ref{bound for bu}, there exists a function $G\in W^{2,r}\left(
%TCIMACRO{\U{211d} }%
%BeginExpansion
\mathbb{R}
%EndExpansion
^{4}\right)  $\ such that $c_{k}u_{k}\rightharpoonup G$ weakly in
$W^{2,r}\left(
%TCIMACRO{\U{211d} }%
%BeginExpansion
\mathbb{R}
%EndExpansion
^{4}\right)  $ for any $1<r<2$. For any $s>0,$ by (\ref{dirc m}), we know
$\exp\left(  \beta_{k}u_{k}^{2}\right)  $ is bounded in $L^{p}\left(
B_{R}\backslash B_{s}\right)  $,\ for any $0<s<R$. Notice that $c_{k}u_{k}$
satisfies
\begin{equation}
\Delta^{2}\left(  c_{k}u_{k}\right)  +c_{k}u_{k}=\frac{c_{k}u_{k}}{\lambda
_{k}}\left(  \exp\left(  \beta_{k}u_{k}^{2}\right)  -\frac{\alpha}{\beta_{k}%
}\right)  \text{ in }%
%TCIMACRO{\U{211d} }%
%BeginExpansion
\mathbb{R}
%EndExpansion
^{4}.\label{equa}%
\end{equation}

Then the standard regularity theory gives  $c_{k}u_{k}\rightarrow G$ in
$C_{loc}^{3}\left(
%TCIMACRO{\U{211d} }%
%BeginExpansion
\mathbb{R}
%EndExpansion
^{4}\backslash\left\{  0\right\}  \right)  $. \

For any $\varphi\left(  x\right)  \in C_{0}^{\infty}\left(
%TCIMACRO{\U{211d} }%
%BeginExpansion
\mathbb{R}
%EndExpansion
^{4}\right)  $, in view of Lemma \ref{Dircm}, we have
\[
\int_{%
%TCIMACRO{\U{211d} }%
%BeginExpansion
\mathbb{R}
%EndExpansion
^{4}}\varphi\left(  x\right)  \left(  \frac{c_{k}u_{k}}{\lambda_{k}}\left(
\exp\left(  \beta_{k}u_{k}^{2}\right)  -\frac{\alpha}{\beta_{k}}\right)
-c_{k}u_{k}\right)  dx=\varphi\left(  0\right)  -\int_{%
%TCIMACRO{\U{211d} }%
%BeginExpansion
\mathbb{R}
%EndExpansion
^{4}}G\left(  x\right)  \varphi\left(  x\right)  dx.
\]
Hence, it follows that
\[
\Delta^{2}G=\delta\left(  x\right)  -G\text{ in }%
%TCIMACRO{\U{211d} }%
%BeginExpansion
\mathbb{R}
%EndExpansion
^{4}.
\]

Fix $r>0$, we choose some cutoff function $\phi\in C_{0}^{\infty}\left(
B_{2r}\left(  0\right)  \right)  $ such that $\phi=1$ in $B_{r}\left(
0\right)  $, and let%
\[
g\left(  x\right)  =G\left(  x\right)  +\frac{1}{8\pi^{2}}\phi\left(
x\right)  \ln\left\vert x\right\vert .
\]
Then a direct computation shows that
\[
\Delta^{2}g\left(  x\right)  =f\text{ in }%
%TCIMACRO{\U{211d} }%
%BeginExpansion
\mathbb{R}
%EndExpansion
^{4}\text{,}%
\]
where
\begin{align*}
f\left(  x\right)   &  =-\frac{1}{8\pi^{2}}\left(  \Delta^{2}\phi\cdot
\ln\left\vert x\right\vert +2\nabla\Delta\phi\cdot\nabla\ln\left\vert
x\right\vert +\right.  \\
&  \left.  +2\Delta\left(  \nabla\phi\cdot\nabla\ln\left\vert x\right\vert
\right)  +2\nabla\phi\cdot\nabla\Delta\ln\left\vert x\right\vert +\phi
\cdot\Delta^{2}\ln\left\vert x\right\vert \right)  +\delta\left(
x\right)  -G.
\end{align*}
Since $\frac{1}{8\pi^{2}}\phi\cdot\Delta^{2}\ln\left\vert x\right\vert
=\delta\left(  x\right)  $ in $%
%TCIMACRO{\U{211d} }%
%BeginExpansion
\mathbb{R}
%EndExpansion
^{4}$, direct calculations yield that
\begin{align*}
f\left(  x\right)   &  =-\frac{1}{8\pi^{2}}\left(  \Delta^{2}\phi\cdot
\ln\left\vert x\right\vert +2\nabla\Delta\phi\cdot\nabla\ln\left\vert
x\right\vert +\right.  \\
&  \left.  +2\Delta\left(  \nabla\phi\cdot\nabla\ln\left\vert x\right\vert
\right)  +2\nabla\phi\cdot\nabla\Delta\ln\left\vert x\right\vert \right)  -G.
\end{align*}
Since $G\in W^{2,r}\left(
%TCIMACRO{\U{211d} }%
%BeginExpansion
\mathbb{R}
%EndExpansion
^{4}\right)  $ for any $1<r<2$, we have $f\left(  x\right)  \in L_{loc}%
^{p}\left(
%TCIMACRO{\U{211d} }%
%BeginExpansion
\mathbb{R}
%EndExpansion
^{4}\right)  $ for any $p>2$. By the standard regularity theory, we get
$g\left(  x\right)  \in C_{loc}^{3}\left(
%TCIMACRO{\U{211d} }%
%BeginExpansion
\mathbb{R}
%EndExpansion
^{4}\right)  $. Let $A=g\left(  0\right)  $ and%
\[
\varphi\left(  x\right)  =g\left(  x\right)  -g\left(  0\right)  +\frac
{1}{8\pi^{2}}\left(  1-\phi\right)  \ln\left\vert x\right\vert .
\]
Then\ we have
\begin{equation}
G=-\frac{1}{8\pi^{2}}\ln\left\vert x\right\vert +A+\varphi\left(  x\right)
,\label{Green1}%
\end{equation}
where $A$ is constant depending on $0$, $\varphi\left(  x\right)  \in
C^{3}\left(
%TCIMACRO{\U{211d} }%
%BeginExpansion
\mathbb{R}
%EndExpansion
^{4}\right)  $ and $\varphi\left(  0\right)  =0$. Then (\ref{Green}) follows
directly from (\ref{Green1}).

Setting $U_{k}=c_{k}u_{k}$, then $U_{k}$ satisfy:
\[
\Delta^{2}U_{k}=\frac{U_{k}}{\lambda_{k}}\left(  \exp\left(  \beta_{k}%
u_{k}^{2}\right)  -\frac{\alpha}{\beta_{k}}\right)  -U_{k}\text{ in }%
%TCIMACRO{\U{211d} }%
%BeginExpansion
\mathbb{R}
%EndExpansion
^{4}.
\]
Testing it with $U_{k}$, by Proposition \ref{split}, we get%
\[
\int_{%
%TCIMACRO{\U{211d} }%
%BeginExpansion
\mathbb{R}
%EndExpansion
^{4}\backslash B_{\delta}\left(  0\right)  }\left\vert \Delta U_{k}\right\vert
^{2}dx+\int_{%
%TCIMACRO{\U{211d} }%
%BeginExpansion
\mathbb{R}
%EndExpansion
^{4}\backslash B_{\delta}\left(  0\right)  }\left\vert U_{k}\right\vert
^{2}dx=\int_{\partial B_{\delta}\left(  0\right)  }v\left(  U_{k}\Delta
^{3/2}U_{k}-\Delta^{1/2}U_{k}\Delta U_{k}\right)  dx,
\]
where $v$ is the outer normal vector of $\partial B_{\delta}\left(  0\right)
$.\ Then we have
\begin{align}
&  \underset{k\rightarrow\infty}{\lim}\left(  \int_{%
%TCIMACRO{\U{211d} }%
%BeginExpansion
\mathbb{R}
%EndExpansion
^{4}\backslash B_{\delta}\left(  0\right)  }\left\vert \Delta U_{k}\right\vert
^{2}dx+\int_{%
%TCIMACRO{\U{211d} }%
%BeginExpansion
\mathbb{R}
%EndExpansion
^{4}\backslash B_{\delta}\left(  0\right)  }\left\vert U_{k}\right\vert
^{2}dx\right)  \nonumber\\
&  =\int_{\partial B_{\delta}\left(  0\right)  }v\left(  G\Delta^{3/2}%
G-\Delta^{1/2}G\Delta G\right)  dx.\label{esti en}%
\end{align}
\bigskip It is known that the fundamental solution of $\Delta^{2}$ in $%
%TCIMACRO{\U{211d} }%
%BeginExpansion
\mathbb{R}
%EndExpansion
^{4}$ is $-\frac{1}{8\pi^{2}}\ln\left\vert x\right\vert $, and it satisfies
\begin{equation}
\Delta^{\frac{1}{2}}\left(  \log\left\vert x\right\vert \right)  =\frac
{x}{\left\vert x\right\vert ^{2}},\Delta\left(  \log\left\vert x\right\vert
\right)  =\frac{2}{\left\vert x\right\vert ^{2}},\Delta^{1+\frac{1}{2}}\left(
\log\left\vert x\right\vert \right)  =-4\frac{x}{\left\vert x\right\vert ^{4}%
}.\label{gradi es}%
\end{equation}
After some computation, we obtain
\begin{align}
v\cdot G\left(  \delta\right)  \Delta^{3/2}G\left(  \delta\right)   &
=\left(  -\frac{1}{8\pi^{2}}\ln\left\vert \delta\right\vert +A+O\left(
\delta\right)  \right)  \left(  \frac{1}{2\pi^{2}}\frac{1}{\delta^{3}%
}+O\left(  1\right)  \right)  \label{es1}\\
&  =\frac{1}{2\pi^{2}}\frac{1}{\delta^{3}}\left(  -\frac{1}{8\pi^{2}}\ln
\delta+A+O\left(  \delta\right)  \right)  \nonumber
\end{align}
and%
\begin{align}
-v\cdot\Delta^{1/2}G\left(  \delta\right)  \Delta G\left(  \delta\right)   &
=-\left(  -\frac{1}{8\pi^{2}}\frac{1}{\delta}+O\left(  1\right)  \right)
\left(  -\frac{1}{8\pi^{2}}\frac{2}{\delta^{2}}+O\left(  1\right)  \right)
\label{es2}\\
&  =-\frac{1}{32\pi^{4}}\frac{1}{\delta^{3}}\left(  1+O\left(  \delta\right)
\right)  .\nonumber
\end{align}
Plugging (\ref{es1})\ and (\ref{es2}) into (\ref{esti en}), we get
\[
\underset{k\rightarrow\infty}{\lim}\left(  \int_{%
%TCIMACRO{\U{211d} }%
%BeginExpansion
\mathbb{R}
%EndExpansion
^{4}\backslash B_{\delta}\left(  0\right)  }\left\vert \Delta U_{k}\right\vert
^{2}dx+\int_{%
%TCIMACRO{\U{211d} }%
%BeginExpansion
\mathbb{R}
%EndExpansion
^{4}\backslash B_{\delta}\left(  0\right)  }\left\vert U_{k}\right\vert
^{2}dx\right)  =-\frac{1}{8\pi^{2}}\ln\delta-\frac{1}{16\pi^{2}}+A+O\left(
\delta\right)  ,
\]
and we are done.
\end{proof}

\subsection{\bigskip The upper bound of Adams inequality for normalized
concentration sequence}
The strategy we will use to obtain the upper bound for the Adams inequality on the entire
Euclidean space\textbf{ }$\mathbb{R}^{4}$ is similar to that of Li-Ruf \cite{liruf}. Firstly, we need to know the specific value
of the upper bound for any blow up function sequences in $H_{0}^{2}\left(
B_{R}\right)  $.

\begin{lemma}
\label{bound for ball}Let $B_{R}$ be the ball\ with radius $R$ in $%
%TCIMACRO{\U{211d} }%
%BeginExpansion
\mathbb{R}
%EndExpansion
^{4}$. Assume that $u_{k}$ is a bounded sequence in $H_{0}^{2}\left(
B_{R}\right)  $ with $\int_{B_{R}}\left\vert \Delta u_{k}\right\vert ^{2}%
dx=1$. If $u_{k}\rightharpoonup0$, then
\[
\underset{k\rightarrow+\infty}{\lim\sup}\int_{B_{R}}\left(  \exp\left(
\beta_{k}u_{k}^{2}\right)  -1-\alpha u_{k}^{2}\right)  dx\leq\frac{1}%
{3}\left\vert B_{R}\right\vert \exp\left(  -\frac{1}{3}\right)  .
\]

\end{lemma}

\begin{proof}
By the results in \cite[(5.23)]{lu-yang 1}, we have
\[
\underset{k\rightarrow+\infty}{\lim\sup}\int_{B_{R}}\left(  \exp\left(
\beta_{k}u_{k}^{2}\right)  -1-\alpha u_{k}^{2}\right)  dx\leq\frac{\pi^{2}}%
{6}\exp\left(  \frac{5}{3}+32\pi^{2}A_{0}\right)  ,
\]
where $A_{0}$ is the value at $0$ of\ the trace of the regular part of the
Green function $\tilde{G}$ for the operator $\Delta^{2}$.\ Actually, when the
domain is a ball, by solving the corresponding ODE's, we have
\[
\tilde{G}=-\frac{1}{8\pi^{2}}\log\left\vert x\right\vert +\frac{1}{16\pi^{2}%
}\frac{\left\vert x\right\vert ^{2}}{R^{2}}+\frac{1}{8\pi^{2}}\log R-\frac
{1}{16\pi^{2}},
\]
and the value at $0$ of\ the trace of the regular part of $G$ is $\frac
{1}{8\pi^{2}}\log R-\frac{1}{8\pi^{2}}$. Therefore we have%
\begin{align*}
&  \underset{k\rightarrow+\infty}{\lim\sup}\int_{B_{R}}\left(  \exp\left(
\beta_{k}u_{k}^{2}\right)  -1-\alpha u_{k}^{2}\right)  dx\\
&  \leq\frac{\pi^{2}}{6}\exp\left(  \frac{5}{3}+32\pi^{2}\left(  \frac{1}%
{8\pi^{2}}\log R-\frac{1}{16\pi^{2}}\right)  \right) \\
&  =\frac{1}{3}\left\vert B_{R}\right\vert \exp\left(  -\frac{1}{3}\right)  .
\end{align*}
\end{proof}

Based on the specific value of the upper bound above, we can obtain the follow upper bound for the Adams inequality on the entire
Euclidean space.
\begin{lemma}
\bigskip If $S\left(  \alpha\right)  $ can not be attained, then
\[
S\left(  \alpha\right)  =\underset{u\in H}{\sup}\int_{%
%TCIMACRO{\U{211d} }%
%BeginExpansion
\mathbb{R}
%EndExpansion
^{4}}\left(  \exp\left(  \beta_{k}u_{k}^{2}\right)  -1-\alpha u_{k}%
^{2}\right)  dx\leq\frac{\pi^{2}}{6}\exp\left(  \frac{5}{3}+32\pi^{2}A\right)
,
\]
where $A$ is the value at $0$ of\ the trace of the regular part of the Green
function $G$ for the operator $\Delta^{2}+1$.
\end{lemma}

\begin{proof}
Set
\[
\tilde{u}_{k}\left(  x\right)  =\frac{u_{k}\left(  x\right)  -u_{k}^{\delta}%
}{\left\Vert \Delta\left(  u_{k}\left(  x\right)  -u_{k}^{\delta}\right)
\right\Vert _{L^{2}\left(  B_{\delta}\left(  x_{k}\right)  \right)  }},
\]
where\
\begin{align}
u_{k}^{\delta}=u_{k}\left(  \delta\right)  +\frac{u_{k}^{\prime}\left(
\delta\right)  r^{2}}{2\delta}-\frac{u_{k}^{\prime}\left(  \delta\right)
\delta}{2}. \label{Def}
\end{align}
\ Now, we compute $\int_{B_{\delta}}\left(  \Delta u_{k}^{\delta}-\Delta
u_{k}\right)  ^{2}dx$. For this, we rewrite it as follows

\begin{align}
\int_{B_{\delta}}\left(  \Delta u_{k}^{\delta}-\Delta u_{k}\right)  ^{2}dx  &
=\int_{B_{\delta}}\left(  \Delta u_{k}\right)  ^{2}dx+\int_{B_{\delta}}\left(
\Delta u_{k}^{\delta}\right)  ^{2}dx-2\int_{B_{\delta}}\Delta u_{k}^{\delta
}\Delta u_{k}dx\nonumber\\
&  =I+II-III \label{M1}%
\end{align}
By Lemma \ref{esti for the out of the ball}, we have%
\begin{align*}
&  \underset{k\rightarrow\infty}{\lim}\left(  \int_{%
%TCIMACRO{\U{211d} }%
%BeginExpansion
\mathbb{R}
%EndExpansion
^{4}\backslash B_{\delta}\left(  0\right)  }\left\vert c_{k}\Delta
u_{k}\right\vert ^{2}dx+\int_{%
%TCIMACRO{\U{211d} }%
%BeginExpansion
\mathbb{R}
%EndExpansion
^{4}\backslash B_{\delta}\left(  0\right)  }\left\vert c_{k}u_{k}\right\vert
^{2}dx\right) \\
&  \leq-\frac{1}{8\pi^{2}}\ln\delta-\frac{1}{16\pi^{2}}+A+O_{k}\left(
\delta\right)  .
\end{align*}
Thus we get
\begin{align}
I  &  =\int_{B_{\delta}\left(  0\right)  }\left\vert \Delta u_{k}\right\vert
^{2}dx=1-\int_{%
%TCIMACRO{\U{211d} }%
%BeginExpansion
\mathbb{R}
%EndExpansion
^{4}\backslash B_{\delta}\left(  0\right)  }\left\vert \Delta u_{k}\right\vert
^{2}dx-\int_{%
%TCIMACRO{\U{211d} }%
%BeginExpansion
\mathbb{R}
%EndExpansion
^{4}\backslash B_{\delta}\left(  0\right)  }\left\vert u_{k}\right\vert
^{2}dx-\int_{B_{\delta}\left(  0\right)  }u_{k}^{2}dx\nonumber\\
&  =1-\frac{-\frac{1}{8\pi^{2}}\ln\delta-\frac{1}{16\pi^{2}}+A+O_{k}\left(
\delta\right)  }{c_{k}^{2}}. \label{addd1}%
\end{align}

By the definition of $u_{k}^{\delta}$, we have%
\begin{align}
II  &  =\int_{B_{\delta}}\left(  \Delta u_{k}^{\delta}\right)  ^{2}%
dx=\int_{B_{\delta}}\left(  4\frac{u_{k}^{\prime}\left(  \delta\right)
}{\delta}\right)  ^{2}dx\nonumber\\
&  =16\frac{\left(  u_{k}^{\prime}\left(  \delta\right)  \right)  ^{2}}%
{\delta^{2}}\left\vert B_{\delta}\right\vert =8\pi^{2}\left(  u_{k}^{\prime
}\left(  \delta\right)  \right)  ^{2}\delta^{2}. \label{M3}%
\end{align}

In order to estimate $III$, by Proposition \ref{split}, we rewrite it as follows:%

\begin{align*}
III &  =2\int_{B_{\delta}}u_{k}^{\delta}\Delta^{2}u_{k}dx-2\int_{\partial
B_{\delta}}v\cdot u_{k}^{\delta}\Delta^{\frac{3}{2}}u_{k}d\sigma
+2\int_{\partial B_{\delta}}v\cdot\Delta^{\frac{1}{2}}u_{k}^{\delta}\Delta
u_{k}d\sigma\\
&  =2\left(  \left(  1\right)  -\left(  2\right)  +\left(  3\right)  \right)
,
\end{align*}

By (\ref{equa}) and Lemma \ref{Dircm}, we have%

\begin{align*}
\left(  1\right)   &  =\int_{B_{\delta}}u_{k}^{\delta}\Delta^{2}u_{k}%
dx=\frac{1}{c_{k}}\int_{B_{\delta}}u_{k}^{\delta}\left(  \frac{c_{k}u_{k}%
}{\lambda_{k}}\left(  \exp\left(  \beta_{k}u_{k}^{2}\right)  -\frac{\alpha
}{\beta_{k}}\right)  -c_{k}u_{k}\right)  dx\\
&  =\frac{1}{c_{k}}\int_{B_{\delta}}u_{k}^{\delta}\left(  \delta\left(
x\right)  -G\left(  x\right)  +o_{k}\left(  1\right)  \right)  dx\\
&  =\frac{1}{c_{k}}\left(  u_{k}\left(  \delta\right)  -\frac{u_{k}^{\prime
}\left(  \delta\right)  \delta}{2}+\frac{o_{k}\left(  \delta\right)  }{c_{k}%
}\right) \\
&  =\frac{1}{c_{k}^{2}}\left(  -\frac{1}{8\pi^{2}}\ln\delta+A+\frac{1}%
{16\pi^{2}}+o_{k}\left(  \delta\right)  \right)  ,
\end{align*}

By (\ref{gradi es}), we have%

\begin{align*}
\left(  2\right)   &  =\int_{\partial B_{\delta}}v\cdot u_{k}^{\delta}%
\Delta^{\frac{3}{2}}u_{k}d\sigma=\frac{1}{c_{k}}\left(  \int_{\partial
B_{\delta}}v\cdot u_{k}^{\delta}\Delta^{\frac{3}{2}}Gd\sigma+\frac
{o_{k}\left(  \delta\right)  }{c_{k}}\right)  \\
&  =\frac{-1}{8\pi^{2}c_{k}}\left(  \int_{\partial B_{\delta}}v\cdot
u_{k}^{\delta}\Delta^{\frac{3}{2}}\ln xd\sigma+\frac{o_{k}\left(
\delta\right)  }{c_{k}}\right)  \\
&  =\frac{-u_{k}\left(  \delta\right)  }{8\pi^{2}c_{k}}\left(  \int_{\partial
B_{\delta}}v\cdot\left(  -4\frac{x}{\left\vert x\right\vert ^{4}}\right)
d\sigma+\frac{o_{k}\left(  \delta\right)  }{c_{k}}\right)  \\
&  =\frac{1}{c_{k}^{2}}\frac{G\left(  \delta\right)  }{2\pi^{2}\delta^{3}%
}\cdot2\pi^{2}\delta^{3}=\frac{G\left(  \delta\right)  +o_{k}\left(
\delta\right)  }{c_{k}^{2}}%
\end{align*}
and%
\begin{align*}
\left(  3\right)   &  =\int_{\partial B_{\delta}}v\cdot\Delta^{\frac{1}{2}%
}u_{k}^{\delta}\Delta u_{k}d\sigma=\frac{1}{c_{k}}\int_{\partial B_{\delta}%
}\left(  u_{k}^{\prime}\left(  \delta\right)  \right)  \Delta Gd\sigma
+\frac{o_{k}\left(  \delta\right)  }{c_{k}^{2}}\\
&  =\left(  -\frac{1}{8\pi^{2}}\right)  ^{2}\frac{2}{\delta^{2}}\frac
{1}{\delta c_{k}^{2}}2\pi^{2}\delta^{3}+\frac{o_{k}\left(  \delta\right)
}{c_{k}^{2}}\\
&  =\frac{1}{c_{k}^{2}}\left(  \frac{1}{16\pi^{2}}+o_{k}\left(  \delta\right)
\right)  .
\end{align*}
Thus, we have%
\begin{align}
&  \int_{B_{\delta}}\Delta u_{k}^{\delta}\Delta u_{k}dx\nonumber\\
&  =\frac{1}{c_{k}^{2}}\left(  -\frac{1}{8\pi^{2}}\ln\delta+A+\frac{1}%
{16\pi^{2}}\right)  +\frac{1}{c_{k}^{2}}\left(  \frac{1}{16\pi^{2}}-G\left(
\delta\right)  +o_{k}\left(  \delta\right)  \right)  \nonumber\\
&  =\frac{1}{c_{k}^{2}}\left(  -\frac{1}{8\pi^{2}}\ln\delta+A+\frac{1}%
{16\pi^{2}}\right)  +\nonumber\\
&  +\frac{1}{c_{k}^{2}}\left(  \frac{1}{16\pi^{2}}+\frac{1}{8\pi^{2}}\ln
\delta-A+O_{k}\left(  \delta\right)  \right)  \nonumber\\
&  =\frac{1}{c_{k}^{2}}\left(  \frac{1}{8\pi^{2}}+O_{k}\left(  \delta\right)
\right)  \label{M4}%
\end{align}
Combining (\ref{addd1}) to (\ref{M4}), \ we get%
\begin{align}
&  \int_{B_{\delta}}\left(  \Delta u_{k}^{\delta}-\Delta u_{k}\right)
^{2}dx\nonumber\\
&  =1-\frac{1}{c_{k}^{2}}\left(  -\frac{1}{8\pi^{2}}\ln\delta-\frac{1}%
{16\pi^{2}}+A+o_{k}\left(  \delta\right)  \right)  +\left(  u_{k}^{\prime
}\left(  \delta\right)  \right)  ^{2}8\pi^{2}\delta^{2}\nonumber\\
&  -\frac{1}{c_{k}^{2}}\left(  \frac{1}{4\pi^{2}}+O_{k}\left(  \delta\right)
\right) \nonumber\\
&  =1-\frac{1}{c_{k}^{2}}\left(  -\frac{1}{8\pi^{2}}\ln\delta+A+\frac{1}%
{16\pi^{2}}+O_{k}\left(  \delta\right)  \right)  . \label{esforlap}%
\end{align}
Therefore, we have%
\begin{align*}
\tilde{u}_{k}^{2}\left(  x\right)   &  =\frac{\left(  u_{k}\left(  x\right)
-u_{k}^{\delta}\left(  x\right)  \right)  ^{2}}{1-\frac{-\frac{1}{8\pi^{2}}%
\ln\delta+\frac{1}{16\pi^{2}}+A+O_{k}\left(  \delta\right)  }{c_{k}^{2}}}\\
&  =u_{k}^{2}\left(  x\right)  \left(  1+\frac{-\frac{1}{8\pi^{2}}\ln
\delta+\frac{1}{16\pi^{2}}+A+O_{k}\left(  \delta\right)  }{c_{k}^{2}}\right)
\\
&  -\left(  2u_{k}^{\delta}\left(  x\right)  u_{k}\left(  x\right)  +\left(
u_{k}^{\delta}\left(  x\right)  \right)  ^{2}\right)  \left(  1+o_{k}\left(
1\right)  \right) \\
&  =u_{k}^{2}\left(  x\right)  -c\ln\delta^{4}.
\end{align*}
\bigskip

On the other hand, by Lemma \ref{concentre}, we have
\[
\underset{L\rightarrow\infty}{\lim}\underset{k\rightarrow\infty}{\lim}%
\int_{B_{\rho}\backslash B_{Lr_{k}}\left(  x_{k}\right)  }\exp\left(
\beta_{k}u_{k}^{2}\right)  =\left\vert B_{\rho}\right\vert ,
\]
for any $\rho<\delta$.

Thus we have
\begin{align*}
\underset{L\rightarrow\infty}{\lim}\underset{k\rightarrow\infty}{\lim}%
\int_{B_{\rho}\backslash B_{Lr_{k}}(x_{k})}\exp\left(  \beta_{k}\tilde{u}%
_{k}^{2}\right)  dx  &  \leq O\left(  \delta^{-4c}\right)  \underset
{L\rightarrow\infty}{\lim}\underset{k\rightarrow\infty}{\lim}\int_{B_{\rho
}\backslash B_{Lr_{k}}(x_{k})}\exp\left(  \beta_{k}u_{k}^{2}\right)  dx\\
&  =O\left(  \delta^{-4c}\right)  \left\vert B_{\rho}\right\vert
\rightarrow0\text{ as }\rho\rightarrow0\text{. }%
\end{align*}
Also, by (\ref{dirc m}) we get
\[
\underset{k\rightarrow\infty}{\lim}\int_{B_{\delta}\backslash B_{\rho}}\left(
\exp\left(  \beta_{k}\tilde{u}_{k}^{2}\right)  -1-\alpha \tilde{u}_{k}^{2}\right)
dx=0.
\]
Hence, by Lemma \ref{bound for ball},\ we derive that
\[
\underset{L\rightarrow\infty}{\lim}\underset{k\rightarrow\infty}{\lim}%
\int_{B_{Lr_{k}}}\left(  \exp\left(  \beta_{k}\tilde{u}_{k}^{2}\right)
-1-\alpha \tilde{u}_{k}^{2}\right)  dx=\underset{k\rightarrow\infty}{\lim}%
\int_{B_{\delta}}\exp(\beta_{k}\tilde{u}_{k}^{2})dx\leq\left\vert B_{\delta
}\right\vert \frac{1}{3}\exp\left(  -\frac{1}{3}\right)  .
\]
Now, we fix some $L>0$, then for any\ $x\in B_{Lr_{k}(x_{k})},$ we have%
\begin{align*}
\beta_{k}u_{k}^{2}  &  =\beta_{k}\left(  \frac{u_{k}}{\left\Vert \Delta\left(
u_{k}\left(  x\right)  -u_{k}^{\delta}\left(  x\right)  \right)  \right\Vert
_{L^{2}\left(  B_{\delta}\right)  }}\right)  ^{2}\int_{B_{\delta}}\left\vert
\Delta\left(  u_{k}\left(  x\right)  -u_{k}^{\delta}\left(  x\right)  \right)
\right\vert ^{2}dx\\
\  &  =\beta_{k}\left(  \tilde{u}_{k}+\frac{u_{k}^{\delta}\left(  x\right)
}{\left\Vert \Delta\left(  u_{k}\left(  x\right)  -u_{k}^{\delta}\left(
x\right)  \right)  \right\Vert _{L^{2}\left(  B_{\delta}\right)  }}\right)
^{2}\int_{B_{\delta}}\left\vert \Delta\left(  u_{k}\left(  x\right)
-u_{k}^{\delta}\left(  x\right)  \right)  \right\vert ^{2}dx.
\end{align*}
By (\ref{esforlap}), we have
\begin{align*}
\beta_{k}u_{k}^{2}  &  =\beta_{k}\left(  \tilde{u}_{k}+u_{k}^{\delta}+O\left(
\frac{1}{c_{k}^{2}}\right)  \right)  ^{2}\cdot\\
&  \cdot\left(  1-\frac{1}{c_{k}^{2}}\left(  -\frac{1}{8\pi^{2}}\ln
\delta+A+\frac{1}{16\pi^{2}}+O_{k}\left(  \delta\right)  \right)  \right) \\
&  =\beta_{k}\tilde{u}_{k}^{2}\left(  1+\frac{u_{k}^{\delta}}{c_{k}}+O\left(
\frac{1}{c_{k}^{3}}\right)  \right)  ^{2}\cdot\\
&  \cdot\left(  1-\frac{1}{c_{k}^{2}}\left(  -\frac{1}{8\pi^{2}}\ln
\delta+A+\frac{1}{16\pi^{2}}+O_{k}\left(  \delta\right)  \right)  \right)  ,
\end{align*}
Observe that
\[
\lim\frac{\tilde{u}_{k}\left(  x_{k}+r_{k}x\right)  }{c_{k}}=1\text{, and
}\tilde{u}_{k}\left(  x_{k}+r_{k}x\right)  u_{k}\left(  \delta\right)
\rightarrow G\left(  \delta\right)  ,
\]
we derive that
\begin{align*}
\beta_{k}u_{k}^{2}  &  =\beta_{k}\tilde{u}_{k}^{2}\left(  1+\frac{1}{c_{k}%
^{2}}\left(  G\left(  \delta\right)  +\frac{1}{16\pi^{2}}\right)  +O\left(
\frac{1}{c_{k}^{3}}\right)  \right)  ^{2}\cdot\\
&  \cdot\left(  1-\frac{1}{c_{k}^{2}}\left(  -\frac{1}{8\pi^{2}}\ln
\delta+A+\frac{1}{16\pi^{2}}+O_{k}\left(  \delta\right)  \right)  \right) \\
&  =\beta_{k}\tilde{u}_{k}^{2}\left(  1+\frac{2G\left(  \delta\right)  }%
{c_{k}^{2}}+\frac{1}{8\pi^{2}c_{k}^{2}}-\frac{G\left(  \delta\right)
+\frac{1}{16\pi^{2}}+O_{k}\left(  \delta\right)  }{c_{k}^{2}}\right) \\
&  =\beta_{k}\tilde{u}_{k}^{2}+\beta_{k}G\left(  \delta\right)  +\frac
{\beta_{k}}{16\pi^{2}}+O_{k}\left(  \delta\right)  .
\end{align*}
Thus we have%

\begin{align*}
&  \underset{L\rightarrow\infty}{\lim}\underset{k\rightarrow\infty}{\lim}%
\int_{B_{Lr_{k}}(x_{k})}\left(  \exp\left(  \beta_{k}u_{k}^{2}\right)
-1-\alpha u_{k}^{2}\right)  dx\\
&  \leq\underset{L\rightarrow\infty}{\lim}\underset{k\rightarrow\infty}{\lim
}\int_{B_{Lr_{k}}(x_{k})}\exp\left(  \beta_{k}u_{k}^{2}\right)  dx\\
&  \leq\underset{L\rightarrow\infty}{\lim}\underset{k\rightarrow\infty}{\lim
}\exp\left(  32\pi^{2}G\left(  \delta\right)  +2+o_{\delta}\left(  1\right)
\right)  \int_{B_{Lr_{k}}}\exp\left(  \beta_{k}\tilde{u}_{k}^{2}\right)  dx\\
&  =\exp\left(  32\pi^{2}G\left(  \delta\right)  +2+o_{\delta}\left(
1\right)  \right)  \left\vert B_{\delta}\right\vert \frac{1}{3}\exp\left(
-\frac{1}{3}\right) \\
&  =\exp\left(  \ln\delta^{-4}+32\pi^{2}A+\varphi\left(  \delta\right)
+2+o_{\delta}\left(  1\right)  \right)  \left\vert B_{\delta}\right\vert
\frac{1}{3}\exp\left(  -\frac{1}{3}\right) \\
&  =\frac{\pi^{2}}{6}\exp\left(  \frac{5}{3}+32\pi^{2}A\right)  +o\left(
\delta\right)  .
\end{align*}

Letting $\delta\rightarrow0$, we get
\[
S\left(  \alpha\right)  \leq\frac{\pi^{2}}{6}\exp\left(  \frac{5}{3}+32\pi
^{2}A\right)  ,
\]
and the proof is finished.
\end{proof}
\section{\bigskip The test function}

Let
\[
\phi_{\varepsilon}=%
%TCIMACRO{\QDATOPD{\{}{.}{C+\frac{a-\frac{1}{16\pi^{2}}\ln\left(  1+\frac{\pi
%}{\sqrt{6}}\frac{x^{2}}{\varepsilon^{2}}\right)  +A+\varphi\left(  x\right)
%+b\left\vert x\right\vert ^{2}}{C}\text{ if }\left\vert x\right\vert \leq
%L\varepsilon,}{\frac{G\left(  x\right)  }{C}\text{
%\ \ \ \ \ \ \ \ \ \ \ \ \ \ \ \ \ \ \ \ \ \ \ \ if }\left\vert x\right\vert
%\geq L\varepsilon,}}%
%BeginExpansion
\genfrac{\{}{.}{0pt}{0}{C+\frac{a-\frac{1}{16\pi^{2}}\ln\left(  1+\frac{\pi
}{\sqrt{6}}\frac{x^{2}}{\varepsilon^{2}}\right)  +A+\varphi\left(  x\right)
+b\left\vert x\right\vert ^{2}}{C}\text{ if }\left\vert x\right\vert \leq
L\varepsilon,}{\frac{G\left(  x\right)  }{C}\text{
\ \ \ \ \ \ \ \ \ \ \ \ \ \ \ \ \ \ \ \ \ \ \ \ if }\left\vert x\right\vert
\geq L\varepsilon,}%
%EndExpansion
\]
where $L,C,a,b$ are functions of $\varepsilon$ (which will be defined later)
such that

i) $\varepsilon=\exp\left(  -L\right)  $, $\frac{1}{C^{2}}=O\left(  \frac
{1}{L}\right)  $ as $\varepsilon\rightarrow0,$

ii) $a=-\frac{1}{8\pi^{2}}\ln\left(  L\varepsilon\right)  -C^{2}+\frac
{1}{16\pi^{2}}\ln\left(  1+\frac{\pi}{\sqrt{6}}L^{2}\right)  -bL^{2}%
\varepsilon^{2},$

iii) $b=-\frac{1}{16\pi^{2}L^{2}\varepsilon^{2}\left(  1+\frac{\pi}{\sqrt{6}%
}L^{2}\right)  }.$

By Lemma \ref{esti for the out of the ball}, we have
\begin{align*}
\int_{%
%TCIMACRO{\U{211d} }%
%BeginExpansion
\mathbb{R}
%EndExpansion
^{4}\backslash B_{L\varepsilon}\left(  0\right)  }\left(  \left\vert
\Delta\phi_{\varepsilon}\right\vert ^{2}+\left\vert \phi_{\varepsilon
}\right\vert ^{2}\right)  dx  &  =\frac{1}{C^{2}}\int_{%
%TCIMACRO{\U{211d} }%
%BeginExpansion
\mathbb{R}
%EndExpansion
^{4}\backslash B_{L\varepsilon}\left(  0\right)  }\left(  \left\vert \Delta
G\right\vert ^{2}+\left\vert G\right\vert ^{2}\right)  dx\\
&  =\frac{1}{C^{2}}\left(  -\frac{1}{8\pi^{2}}\ln\left(  L\varepsilon\right)
-\frac{1}{16\pi^{2}}+A+O\left(  L\varepsilon\right)  \right)  ,
\end{align*}
and\ as \cite{lu-yang 1}, one has%
\[
\int_{B_{L\varepsilon}\left(  0\right)  }\left\vert \Delta\phi_{\varepsilon
}\right\vert ^{2}dx=\frac{1}{96\pi^{2}C^{2}}\left(  6\ln\left(  1+\frac{\pi
}{\sqrt{6}}L^{2}\right)  +1+O\left(  \frac{1}{\ln^{2}\varepsilon}\right)
\right)  .
\]

It is easy to check that \ \
\begin{equation}
\int_{B_{L\varepsilon}\left(  0\right)  }\left\vert \phi_{\varepsilon
}\right\vert ^{2}dx=O\left(  \left(  L\varepsilon\right)  ^{4}C^{4}\right)  ,
\label{38}%
\end{equation}
thus it follows that
\begin{align*}
&  \int_{%
%TCIMACRO{\U{211d} }%
%BeginExpansion
\mathbb{R}
%EndExpansion
^{4}}\left(  \left\vert \Delta\phi_{\varepsilon}\right\vert ^{2}+\left\vert
\phi_{\varepsilon}\right\vert ^{2}\right)  dx\\
&  =\frac{1}{32\pi^{2}C^{2}}\left(  2\ln\left(  1+\frac{\pi}{\sqrt{6}}%
L^{2}\right)  -\frac{5}{3}-4\ln\left(  L\varepsilon\right)  +32\pi
^{2}A+O\left(  \frac{1}{\ln^{2}\varepsilon}\right)  \right) \\
&  =\frac{1}{32\pi^{2}C^{2}}\left(  2\ln\left(  \frac{\frac{\pi}{\sqrt{6}%
}\left(  1+\frac{\sqrt{6}}{\pi}\frac{1}{L^{2}}\right)  }{\varepsilon^{2}%
}\right)  -\frac{5}{3}+32\pi^{2}A+O\left(  \frac{1}{\ln^{2}\varepsilon
}\right)  \right) \\
&  =\frac{1}{32\pi^{2}C^{2}}\left(  2\ln\left(  \frac{\frac{\pi}{\sqrt{6}}%
}{\varepsilon^{2}}\right)  +2\ln\left(  1+\frac{\sqrt{6}}{\pi}\frac{1}{L^{2}%
}\right)  -\frac{5}{3}+32\pi^{2}A+O\left(  \frac{1}{\ln^{2}\varepsilon
}\right)  \right) \\
&  =\frac{1}{32\pi^{2}C^{2}}\left(  2\ln\left(  \frac{\frac{\pi}{\sqrt{6}}%
}{\varepsilon^{2}}\right)  +\frac{2\sqrt{6}}{\pi}\frac{1}{L^{2}}-\frac{5}%
{3}+32\pi^{2}A+O\left(  \frac{1}{\ln^{2}\varepsilon}\right)  \right) \\
&  =\frac{1}{32\pi^{2}C^{2}}\left(  2\ln\left(  \frac{\frac{\pi}{\sqrt{6}}%
}{\varepsilon^{2}}\right)  -\frac{5}{3}+32\pi^{2}A+O\left(  \frac{1}{\ln
^{2}\varepsilon}\right)  \right)  .
\end{align*}

Set $\int_{%
%TCIMACRO{\U{211d} }%
%BeginExpansion
\mathbb{R}
%EndExpansion
^{4}}\left(  \left\vert \Delta\phi_{\varepsilon}\right\vert ^{2}+\left\vert
\phi_{\varepsilon}\right\vert ^{2}\right)  dx=1$ and direct computations yield
that \ $\ $%

\begin{equation}
32\pi^{2}C^{2}=2\ln\frac{\pi}{\sqrt{6}\varepsilon^{2}}-\frac{5}{3}+32\pi
^{2}A+O\left(  \frac{1}{\ln^{2}\varepsilon}\right)  , \label{es for c}%
\end{equation}
then
\begin{equation}
C^{2}\sim\frac{1}{8\pi^{2}}\ln\frac{1}{\varepsilon}. \label{40}%
\end{equation}
For any $x\in B_{L\varepsilon}$, by careful calculation, we also derive that
\begin{align}
32\pi^{2}\phi_{\varepsilon}^{2}  &  \geq32\pi^{2}\left(  C^{2}+2\left(
a-\frac{1}{16\pi^{2}}\ln\left(  1+\frac{\pi}{\sqrt{6}}\frac{x^{2}}%
{\varepsilon^{2}}\right)  +A+\varphi\left(  x\right)  +b\left\vert
x\right\vert ^{2}\right)  \right) \nonumber\\
&  =32\pi^{2}\left(  -C^{2}+2\left(  -\frac{1}{8\pi^{2}}\ln\left(
L\varepsilon\right)  +\frac{1}{16\pi^{2}}\ln\left(  1+\frac{\pi}{\sqrt{6}%
}L^{2}\right)  -bL^{2}\varepsilon^{2}\right)  \right. \nonumber\\
&  \left.  +2A+2\varphi\left(  x\right)  +2b\left\vert x\right\vert ^{2}%
-\frac{1}{8\pi^{2}}\ln\left(  1+\frac{\pi}{\sqrt{6}}\frac{x^{2}}%
{\varepsilon^{2}}\right)  \right) \nonumber\\
&  =4\ln\left(  1+\frac{\pi}{\sqrt{6}}L^{2}\right)  -8\ln\left(
L\varepsilon\right)  -2\ln\frac{\pi}{\sqrt{6}\varepsilon^{2}}+\frac{5}%
{3}+32\pi^{2}A\nonumber\\
&  -4\ln\left(  1+\frac{\pi}{\sqrt{6}}\frac{x^{2}}{\varepsilon^{2}}\right)
+O\left(  \frac{1}{\ln^{2}\varepsilon}\right)  +64\pi^{2}\left(
\varphi\left(  x\right)  +b\left\vert x\right\vert ^{2}\right)  -64bL^{2}%
\varepsilon^{2}\nonumber\\
&  =\ln\left(  \left(  1+\frac{\pi}{\sqrt{6}}L^{2}\right)  ^{4}\frac
{6\varepsilon^{4}}{\pi^{2}}\left(  L\varepsilon\right)  ^{-8}\right)
+\frac{5}{3}+32\pi^{2}A\nonumber\\
&  -4\ln\left(  1+\frac{\pi}{\sqrt{6}}\frac{x^{2}}{\varepsilon^{2}}\right)
+O\left(  \frac{1}{\ln^{2}\varepsilon}\right)  +64\pi^{2}\left(
\varphi\left(  x\right)  +b\left\vert x\right\vert ^{2}\right)  -64bL^{2}%
\varepsilon^{2}\nonumber\\
&  =\ln\frac{6\varepsilon^{-4}}{\pi^{2}}\frac{\pi^{4}}{36}+O\left(
L^{-2}\right)  +\frac{5}{3}+32\pi^{2}A-4\ln\left(  1+\frac{\pi}{\sqrt{6}}%
\frac{x^{2}}{\varepsilon^{2}}\right)  +O\left(  \frac{1}{\ln^{2}\varepsilon
}\right) \nonumber\\
&  \geq\ln\frac{\pi^{2}}{6\varepsilon^{4}}+32\pi^{2}A+\frac{5}{3}-\ln\left(
1+\frac{\pi r^{2}}{\sqrt{6}\varepsilon^{2}}\right)  ^{4}+O\left(  \frac{1}%
{\ln^{2}\varepsilon}\right)  , \label{41}%
\end{align}
where we have used that fact that $\varphi$ is a continuous function and
$\varphi\left(  0\right)  =0$.

Therefore, combining (\ref{38}) and (\ref{41}), we derive that%

\begin{align*}
&  \int_{B_{L\varepsilon}}\left(  \exp\left(  32\pi^{2}\phi_{\varepsilon}%
^{2}\right)  -1-\alpha\phi_{\varepsilon}^{2}\right)  dx\\
&  \geq\int_{B_{L\varepsilon}}\exp\left(  32\pi^{2}\phi_{\varepsilon}%
^{2}\right)  dx+O\left(  \left(  L\varepsilon\right)  ^{4}C^{4}\right) \\
&  =\frac{\pi^{2}}{6\varepsilon^{4}}\exp\left(  32\pi^{2}A+\frac{5}{3}\right)
\int_{B_{L\varepsilon}}\left(  1+\frac{\pi r^{2}}{\sqrt{6}\varepsilon^{2}%
}\right)  ^{-4}dx+O\left(  \frac{1}{\ln^{2}\varepsilon}\right)  .
\end{align*}
Since%

\begin{align*}
&  \int_{B_{L\varepsilon}}\left(  1+\frac{\pi r^{2}}{\sqrt{6}\varepsilon^{2}%
}\right)  ^{-4}dx\\
&  =2\pi^{2}\bigskip\varepsilon^{4}\int_{0}^{L}\left(  1+\frac{\pi r^{2}%
}{\sqrt{6}}\right)  ^{-4}r^{3}dr\\
&  =2\pi^{2}\bigskip\varepsilon^{4}\frac{1}{2}\frac{6}{\pi^{2}}\int_{0}%
^{\frac{\pi}{\sqrt{6}}L^{2}}\frac{u}{\left(  1+u\right)  ^{4}}du\\
&  =2\pi^{2}\bigskip\varepsilon^{4}\frac{1}{2}\frac{6}{\pi^{2}}\left(
\frac{1}{6}-\frac{1}{3}\left(  1+\frac{\pi}{\sqrt{6}}L^{2}\right)
^{-2}+O\left(  L^{-6}\right)  \right) \\
&  =\varepsilon^{4}\left(  1-2\left(  1+\frac{\pi}{\sqrt{6}}L^{2}\right)
^{-2}+O\left(  L^{-6}\right)  \right)  ,
\end{align*}
we have \
\begin{align*}
&  \int_{B_{L\varepsilon}}\left(  \exp\left(  32\pi^{2}\phi_{\varepsilon}%
^{2}\right)  -1-\alpha\phi_{\varepsilon}^{2}\right)  dx\\
&  \geq\frac{\pi^{2}}{6\varepsilon^{4}}\exp\left(  32\pi^{2}A+\frac{5}%
{3}\right)  \varepsilon^{4}\left(  1-2\left(  1+\frac{\pi}{\sqrt{6}}%
L^{2}\right)  ^{-2}+O\left(  L^{-6}\right)  \right)  +\\
&  +O\left(  C^{4}\left(  L\varepsilon\right)  ^{4}\right)  +O\left(  \frac
{1}{\ln^{2}\varepsilon}\right) \\
&  =\frac{\pi^{2}}{6}\exp\left(  32\pi^{2}A+\frac{5}{3}\right)  +O\left(
\frac{1}{\ln^{2}\varepsilon}\right)  .
\end{align*}

On the other hand, we also have%
\begin{align*}
&  \int_{%
%TCIMACRO{\U{211d} }%
%BeginExpansion
\mathbb{R}
%EndExpansion
^{4}\backslash B_{L\varepsilon}}\left(  \exp\left(  32\pi^{2}\phi
_{\varepsilon}^{2}\right)  -1-\alpha\phi_{\varepsilon}^{2}\right)  dx\\
&  \geq\frac{32\pi^{2}-\alpha}{C^{2}}\int_{%
%TCIMACRO{\U{211d} }%
%BeginExpansion
\mathbb{R}
%EndExpansion
^{4}\backslash B_{L\varepsilon}}G\left(  x\right)  ^{2}dx\\
&  =\frac{32\pi^{2}-\alpha}{C^{2}}\left\Vert G\left(  x\right)  \right\Vert
_{L^{2}\left(
%TCIMACRO{\U{211d} }%
%BeginExpansion
\mathbb{R}
%EndExpansion
^{4}\right)  }^{2}\\
&  =\frac{32\pi^{2}-\alpha}{C^{2}}\left(  \left\Vert G\left(  x\right)
\right\Vert _{L^{2}\left(
%TCIMACRO{\U{211d} }%
%BeginExpansion
\mathbb{R}
%EndExpansion
^{4}\right)  }^{2}+O\left(  L^{4}\varepsilon^{4}\right)  \right)  .
\end{align*}
Then we get%

\begin{align*}
&  \bigskip\int_{%
%TCIMACRO{\U{211d} }%
%BeginExpansion
\mathbb{R}
%EndExpansion
^{4}}\left(  \exp\left(  32\pi^{2}\phi_{\varepsilon}^{2}\right)  -1-\alpha
\phi_{\varepsilon}^{2}\right)  dx\\
&  =\frac{\pi^{2}}{6}\exp\left(  32\pi^{2}A+\frac{5}{3}\right)  +\frac
{32\pi^{2}-\alpha}{C^{2}}\left\Vert G\left(  x\right)  \right\Vert
_{L^{2}\left(
%TCIMACRO{\U{211d} }%
%BeginExpansion
\mathbb{R}
%EndExpansion
^{4}\right)  }^{2}+O\left(  \frac{1}{\ln^{2}\varepsilon}\right)  .
\end{align*}
By (\ref{40}), we know $C^{2}\sim\left\vert \ln\varepsilon\right\vert $, which
implies that
\begin{equation}
\int_{%
%TCIMACRO{\U{211d} }%
%BeginExpansion
\mathbb{R}
%EndExpansion
^{4}}\left(  \exp\left(  32\pi^{2}\phi_{\varepsilon}^{2}\right)  -1\right)
dx>\frac{\pi^{2}}{6}\exp\left(  32\pi^{2}A+\frac{5}{3}\right)  \label{supass}%
\end{equation}
for $\varepsilon$ small enough. This accomplishes the proof.

\bigskip

\section{Nonexistence of extremals}

In this section, we will show that when $32\pi^{2}-\alpha$ large enough, the
supremum $S\left(  \alpha\right)  $ is not attained. For this aim, we will
need the precise estimates for the best constants of Gagliardo-Nirenberg
inequalities, the detailed proof is given in the Appendix.

\begin{lemma}
\label{sharp1}Let \bigskip$B_{k}=\underset{u\in H^{2}\left(
%TCIMACRO{\U{211d} }%
%BeginExpansion
\mathbb{R}
%EndExpansion
^{4}\right)  }{\sup}\frac{\int_{%
%TCIMACRO{\U{211d} }%
%BeginExpansion
\mathbb{R}
%EndExpansion
^{4}}u^{2k}dx}{\left(  \int_{%
%TCIMACRO{\U{211d} }%
%BeginExpansion
\mathbb{R}
%EndExpansion
^{4}}\left\vert \Delta u\right\vert ^{2}dx\right)  ^{k-1}\int_{%
%TCIMACRO{\U{211d} }%
%BeginExpansion
\mathbb{R}
%EndExpansion
^{4}}u^{2}dx}$, then
\begin{equation}
B_{k}\leq\frac{1}{4}\left(  1+\frac{2k}{2k-1}\right)  ^{2k}\left(
\frac{2\left(  k-1\right)  }{2k-1}\right)  ^{2k-2}\cdot\sqrt{\left(  \frac
{k}{k-1}\right)  ^{k-1}}k^{-1/2}\left(  \frac{k}{32\pi^{2}}\right)  ^{k}%
\cdot32\pi^{2}. \label{best constant}%
\end{equation}

\end{lemma}

\begin{proof}
[Proof of Theorem \ref{nonatain}]Let $u\in H^{2}\left(
%TCIMACRO{\U{211d} }%
%BeginExpansion
\mathbb{R}
%EndExpansion
^{4}\right)  $ satisfying $\left\Vert u\right\Vert _{H^{2}\left(
%TCIMACRO{\U{211d} }%
%BeginExpansion
\mathbb{R}
%EndExpansion
^{4}\right)  }=1$, we have%
\begin{align*}
&  \int_{%
%TCIMACRO{\U{211d} }%
%BeginExpansion
\mathbb{R}
%EndExpansion
^{4}}\left(  \exp\left(  32\pi^{2}\left\vert u\right\vert ^{2}\right)
-1-\alpha\left\vert u\right\vert ^{2}\right)  dx\\
&  =\left(  32\pi^{2}-\alpha\right)  \int_{%
%TCIMACRO{\U{211d} }%
%BeginExpansion
\mathbb{R}
%EndExpansion
^{4}}u^{2}dx+\frac{\left(  32\pi^{2}\right)  ^{2}}{2}\int_{%
%TCIMACRO{\U{211d} }%
%BeginExpansion
\mathbb{R}
%EndExpansion
^{4}}u^{4}dx+\ldots+\frac{\left(  32\pi^{2}\right)  ^{k}}{k!}\int_{%
%TCIMACRO{\U{211d} }%
%BeginExpansion
\mathbb{R}
%EndExpansion
^{4}}u^{2k}dx+\ldots.
\end{align*}
By the Gagliardo-Nirenberg inequalities, we have
\[
\frac{\left(  32\pi^{2}\right)  ^{k}}{k!}\int_{%
%TCIMACRO{\U{211d} }%
%BeginExpansion
\mathbb{R}
%EndExpansion
^{4}}u^{2k}dx\leq\frac{\left(  32\pi^{2}\right)  ^{k}}{k!}B_{k}\left(  \int_{%
%TCIMACRO{\U{211d} }%
%BeginExpansion
\mathbb{R}
%EndExpansion
^{4}}\left\vert \Delta u\right\vert ^{2}dx\right)  ^{k-1}\int_{%
%TCIMACRO{\U{211d} }%
%BeginExpansion
\mathbb{R}
%EndExpansion
^{4}}u^{2}dx,
\]
where $B_{k}$ is the best constant. We employ \eqref{best constant} to obtain
that
\begin{align*}
B_{k}  &  \leq\frac{1}{4}\left(  1+\frac{1}{2k-1}\right)  ^{2k-1}\left(
\frac{2k}{2k-1}\right)  \left(  1-\frac{1}{2k-1}\right)  ^{2k-1}\left(
\frac{2k-1}{2\left(  k-1\right)  }\right)  \cdot\\
&  \sqrt{\left(  1+\frac{1}{k-1}\right)  ^{k-1}}k^{-1/2}\left(  \frac{k}%
{32\pi^{2}}\right)  ^{k}\cdot32\pi^{2}\\
&  \sim\frac{8\pi^{2}\sqrt{e}}{\sqrt{k}}\left(  \frac{k}{32\pi^{2}}\right)
^{k},\text{ }%
\end{align*}
when $k$ is large enough. Thus, it follows that
\begin{align*}
\frac{\left(  32\pi^{2}\right)  ^{k}}{k!}B_{k} &  \leq c\frac{\left(
32\pi^{2}\right)  ^{k}}{k!}\frac{8\pi^{2}\sqrt{e}}{\sqrt{k}}\left(  \frac
{k}{32\pi^{2}}\right)  ^{k}\\
&  =c8\pi^{2}\sqrt{e}\frac{k^{k}}{\sqrt{k}k!}\left(  \text{by the stirling
formula}\right)  \\
&  \sim c\frac{8\pi^{2}\sqrt{e}}{\sqrt{2\pi}}\cdot\frac{e^{k}}{k}.
\end{align*}
Setting $\int_{%
%TCIMACRO{\U{211d} }%
%BeginExpansion
\mathbb{R}
%EndExpansion
^{4}}\left\vert \Delta u\right\vert ^{2}dx=t$, we have
\[
\frac{\left(  32\pi^{2}\right)  ^{k}}{k!}\int_{%
%TCIMACRO{\U{211d} }%
%BeginExpansion
\mathbb{R}
%EndExpansion
^{4}}u^{2k}dx\leq\frac{c8\pi^{2}\sqrt{e}}{\sqrt{2\pi}}\cdot\frac{e^{k}}{k}%
t^{k}\left(  1-t\right).
\]
 Since the series
\[
\underset{k=1}{\overset{\infty}{\sum}}\frac{8\pi^{2}\sqrt{e}}{\sqrt{2\pi}%
}\cdot\frac{e^{k}}{k}t^{k}\left(  1-t\right)
\]
converges if $t<\frac{1}{e}$, hence, there exists some constant $c>0$ such
that
\begin{align*}
&  \int_{%
%TCIMACRO{\U{211d} }%
%BeginExpansion
\mathbb{R}
%EndExpansion
^{4}}\left(  \exp\left(  32\pi^{2}\left\vert u\right\vert ^{2}\right)
-1-\alpha\left\vert u\right\vert ^{2}\right)  dx\\
&  \leq F\left(  t\right)  =\left(  32\pi^{2}-\alpha\right)  \left(
1-t\right)  +c\underset{k=2}{\overset{\infty}{\sum}}\frac{8\pi^{2}\sqrt{e}%
}{\sqrt{2\pi}}\cdot\frac{e^{k}}{k}t^{k}.
\end{align*}
Since $F^{\prime}\left(  t\right)  |_{t=0}=-32\pi^{2}+\alpha<0$ if
$\alpha<32\pi^{2}$ and $c\underset{k=1}{\overset{\infty}{\sum}}\frac{8\pi
^{2}\sqrt{e}}{\sqrt{2\pi}}\cdot\frac{e^{k}}{k}t^{k}\rightarrow0$ as
$t\rightarrow0$, there exists some $t_{0}$ such that $F(t)$ is decreasing on $[0,t_0]$. Hence, it follows that
\[
F\left(  t\right)  \leq F\left(  0\right) \leq 32\pi^{2}-\alpha.
\]

Now we consider the case $t\geq t_{0}$. By the Adams inequality, there exists some
$M>0$ such that
\[
\int_{%
%TCIMACRO{\U{211d} }%
%BeginExpansion
\mathbb{R}
%EndExpansion
^{4}}\left(  \exp\left(  32\pi^{2}\left\vert u\right\vert ^{2}\right)
-1-32\pi^{2}\left\vert u\right\vert ^{2}\right)  dx<M.
\]
Hence
\begin{align*}
F\left(  t\right)   &  \leq\left(  32\pi^{2}-\alpha\right)  \left(
1-t\right)  +M\\
&  =32\pi^{2}-\alpha+M-\left(  32\pi^{2}-\alpha\right)  t_{0},
\end{align*}
and we have $F\left(  t\right)  <32\pi^{2}-\alpha$, if $32\pi^{2}-\alpha
>\frac{M}{t_{0}}$.
\end{proof}

Finally, we give the

\begin{proof}
[Proof of Theorem \ref{fina}]We only need to show that if $\ S\left(
\alpha_{1}\right)  $ is attained, then for any $\alpha_{1}<\alpha_{2}$,
$S\left(  \alpha_{2}\right)  $ is also attained.

For any $\alpha\in%
%TCIMACRO{\U{211d} }%
%BeginExpansion
\mathbb{R}
%EndExpansion
$, we denote by $d_{nv}\left(  \alpha\right)  $ and $d_{nc}\left(
\alpha\right)  $ for the upper bounds of Adams' inequality of the normalized
vanishing sequences and the normalized concentration sequences, respectively.
It is easy to check the following facts:
\[
d_{nc}\left(  \alpha_{1}\right)  =d_{nc}\left(  \alpha_{2}\right)
,d_{nv}\left(  \alpha_{1}\right)  -d_{nv}\left(  \alpha_{2}\right)
=\alpha_{2}-\alpha_{1},
\]
for any $\alpha_{1}<\alpha_{2}$. \ Since $S\left(  \alpha_{1}\right)  $ is
attained, we have
\[
S\left(  \alpha_{1}\right)  \geq\max\left\{  d_{nv}\left(  \alpha_{1}\right)
,d_{nc}\left(  \alpha_{1}\right)  \right\}  .
\]
On the other hand, by (\ref{supass}),\ we know that
\begin{equation}
S\left(  \alpha\right)  >d_{nc}\left(  \alpha\right)  \text{, for any }%
\alpha\in%
%TCIMACRO{\U{211d} }%
%BeginExpansion
\mathbb{R}
%EndExpansion
, \label{concen}%
\end{equation}
thus, we have
\begin{equation}
S\left(  \alpha_{1}\right)  \geq d_{nv}\left(  \alpha_{1}\right)  .
\label{sup}%
\end{equation}

Now, we show that $S\left(  \alpha_{2}\right)  \geq d_{nv}\left(  \alpha
_{2}\right)  $. Indeed, since $S\left(  \alpha_{1}\right)  $ is attained by
some $\bar{u}\in H^{2}\left(
%TCIMACRO{\U{211d} }%
%BeginExpansion
\mathbb{R}
%EndExpansion
^{4}\right)  \backslash\left\{  0\right\}  $ satisfying $\left\Vert \bar
{u}\right\Vert _{H^{2}\left(
%TCIMACRO{\U{211d} }%
%BeginExpansion
\mathbb{R}
%EndExpansion
^{4}\right)  }=1$, that is,
\begin{align*}
S\left(  \alpha_{1}\right)   &  =\int_{%
%TCIMACRO{\U{211d} }%
%BeginExpansion
\mathbb{R}
%EndExpansion
^{4}}\left(  \exp\left(  32\pi^{2}\left\vert \bar{u}\right\vert ^{2}\right)
-1-\alpha_{1}\left\vert \bar{u}\right\vert ^{2}\right)  dx\\
&  =\left(  32\pi^{2}-\alpha_{1}\right)  \int_{%
%TCIMACRO{\U{211d} }%
%BeginExpansion
\mathbb{R}
%EndExpansion
^{4}}\bar{u}^{2}dx+G\left(  \bar{u}\right)
\end{align*}
where $G\left(  \bar{u}\right)  =\int_{%
%TCIMACRO{\U{211d} }%
%BeginExpansion
\mathbb{R}
%EndExpansion
^{4}}\left(  \exp\left(  32\pi^{2}\left\vert \bar{u}\right\vert ^{2}\right)
-1-32\pi^{2}\left\vert \bar{u}\right\vert ^{2}\right)  dx$. \ By (\ref{sup})
and the fact that $\int_{%
%TCIMACRO{\U{211d} }%
%BeginExpansion
\mathbb{R}
%EndExpansion
^{4}}\bar{u}^{2}dx<1$, we have
\begin{align}
S\left(  \alpha_{2}\right)   &  \geq\left(  32\pi^{2}-\alpha_{2}\right)
\int_{%
%TCIMACRO{\U{211d} }%
%BeginExpansion
\mathbb{R}
%EndExpansion
^{4}}\bar{u}^{2}dx+G\left(  \bar{u}\right) \nonumber\\
&  =\left(  \alpha_{1}-\alpha_{2}\right)  \int_{%
%TCIMACRO{\U{211d} }%
%BeginExpansion
\mathbb{R}
%EndExpansion
^{4}}\bar{u}^{2}dx+\left(  32\pi^{2}-\alpha_{1}\right)  \int_{%
%TCIMACRO{\U{211d} }%
%BeginExpansion
\mathbb{R}
%EndExpansion
^{4}}\bar{u}^{2}dx+G\left(  \bar{u}\right) \nonumber\\
&  =\left(  \alpha_{1}-\alpha_{2}\right)  \int_{%
%TCIMACRO{\U{211d} }%
%BeginExpansion
\mathbb{R}
%EndExpansion
^{4}}\bar{u}^{2}dx+S\left(  \alpha_{1}\right) \nonumber\\
&  \geq\left(  \alpha_{1}-\alpha_{2}\right)  \int_{%
%TCIMACRO{\U{211d} }%
%BeginExpansion
\mathbb{R}
%EndExpansion
^{4}}\bar{u}^{2}dx+d_{nv}\left(  \alpha_{1}\right) \nonumber\\
&  >\alpha_{1}-\alpha_{2}+32\pi^{2}-\alpha_{1}\nonumber\\
&  =32\pi^{2}-\alpha_{2}=d_{nv}\left(  \alpha_{2}\right)  . \label{vanish}%
\end{align}
Combining (\ref{concen}) and (\ref{vanish}), we have $S\left(  \alpha
_{2}\right)  >\max\left\{  d_{nv}\left(  \alpha_{2}\right)  ,d_{nc}\left(
\alpha_{2}\right)  \right\}  $, and then $S\left(  \alpha_{2}\right)  $ is attained.
\end{proof}

\section{Existence and nonexistence of extremal functions for the Trudinger-Mpser inequalities in $\mathbb{R}^2$}\label{section6}

In this section, we will provide  sketchy proofs of Theorems \ref{addthm2} and \ref{addthm3} all together.

\vskip0.1cm

Step 1: Setting
\[
I_{\beta}^{\alpha}\left(  u\right)  =\int_{B_{R}}\left(  \exp(\beta
|u|^{2})-1-\alpha|u|^{2}\right)  dx,
\]
one can easily check that $I_{\beta}^{\alpha}$ could be achieved through combining the subcritical Trudinger-Moser inequality in $H^{1}(\mathbb{R}^2)$ and Vitali convergence theorem.
\medskip

Step 2. We can find a positive radially symmetric maximizing sequence
$\left\{  u_{k}\right\}  $\ for critical functional
\[
\tilde{S}\left(  \alpha\right)  =\underset{u\in H^{1},\left\Vert u\right\Vert
_{H^{1}}=1}{\sup}\int_{\mathbb{R}^{2}}\left(  \exp(4\pi|u|^{2})-1-\alpha
|u|^{2}\right)  dx.
\]
where, $u_k$ is positive, radial extremals for the Trudinger-Moser inequality
$$\int_{B_{R_k}}\left(  \exp(\beta_k
|u|^{2})-1-\alpha|u|^{2}\right)  dx
$$
and $R_k\rightarrow \mathbb{R}^2$, $\beta_k\rightarrow 32\pi^2$.
\medskip

Step 3. If $c_{k}$ is bounded from above, then one of the following holds.

(i) $u\neq0$ and $\tilde{S}\left(  \alpha\right)$ could be achieved by a
radial function $u\in H^{1}(\mathbb{R}^{2})$,

(ii) $u=0$ and $\{u_{k}\}$ is a normalized vanishing sequence, furthermore,
$\tilde{S}\left(  \alpha\right)\leq d_{nv}=4\pi-\alpha$.

where $d_{nv}$ is the upper bound of Trudinger-Moser inequality for normalized
vanishing sequence.
\medskip

Step 4. Similar to the proof of Lemma \ref{lem2}, we can show $\tilde{S}\left(\alpha\right)
>d_{nv}$, when $4\pi-8\pi^{2}B_{2}<\alpha$. Indeed, for any $v\in
H^{1}(\mathbb{R}^{2})$ and $t>0$, we introduce a family of functions $v_{t}$
by
\[
v_{t}(x)=t^{\frac{1}{2}}v(t^{\frac{1}{2}}x),
\]
and we easily verify that
\[
\Vert\nabla v_{t}\Vert_{2}^{2}=t\Vert\nabla v\Vert_{2}^{2},\Vert v_{t}%
\Vert_{p}^{p}=t^{\frac{p-2}{2}}\Vert v\Vert_{p}^{p}.
\]
Hence, it follows that

\bigskip
\begin{equation}%
\begin{split}
&  \int_{\mathbb{R}^{2}}\left(  \exp\Big(4\pi\big(\frac{v_{t}}{\Vert
v_{t}\Vert_{H^{1}(\mathbb{R}^{2})}}\big)^{2}\Big)-1-\alpha\left(  \frac{v_{t}%
}{\Vert v_{t}\Vert_{H^{1}(\mathbb{R}^{2})}}\right)  ^{2}\right)  dx\\
&  \ \ \geq\left(  4\pi-\alpha\right)  \frac{\Vert v_{t}\Vert_{2}^{2}}%
{\Vert\nabla v_{t}\Vert_{2}^{2}+\Vert v_{t}\Vert_{2}^{2}}+\frac{\left(
4\pi\right)  ^{2}}{2}\frac{\Vert v_{t}\Vert_{4}^{4}}{(\Vert\nabla v_{t}%
\Vert_{2}^{2}+\Vert v_{t}\Vert_{2}^{2})^{2}}\\
&  \ \ =\left(  4\pi-\alpha\right)  \Big(\frac{\Vert v\Vert_{2}^{2}}%
{t\Vert\nabla v\Vert_{2}^{2}+\Vert v\Vert_{2}^{2}}+\frac{\left(  4\pi\right)
^{2}}{2\left(  4\pi-\alpha\right)  }\frac{t\Vert v\Vert_{4}^{4}}{(t\Vert\nabla
v\Vert_{2}^{2}+\Vert v\Vert_{2}^{2})^{2}}\Big)\\
&  \ \ =\left(  4\pi-\alpha\right)  g_{v}(t).
\end{split}
\label{s2}%
\end{equation}

Noting that $g_{v}(0)=1$, and $g_{v}^{\prime}(0)=\frac{\left\Vert \nabla
v\right\Vert _{2}^{2}}{\left\Vert v\right\Vert _{2}^{2}}\left(  \frac{\left(
4\pi\right)  ^{2}\left\Vert v\right\Vert _{4}^{4}}{2\left(  4\pi
-\alpha\right)  \left\Vert \nabla v\right\Vert _{2}^{2}\left\Vert v\right\Vert
_{2}^{2}}-1\right)  $, we recall that $\bar{B}_{1}=\sup\frac{\left\Vert
v\right\Vert _{4}^{4}}{\left\Vert \nabla v\right\Vert _{2}^{2}\left\Vert
v\right\Vert _{2}^{2}}$ is attained by some $U\in H^{1}\left(  \mathbb{R}%
^{2}\right)  \backslash\left\{  0\right\}  $. Then $g_{U}^{\prime}(0)>0$ if
$4\pi-\alpha<8\pi^{2}B_{1}$. Therefore, the vanishing phenomenon does not occur.
\medskip

Step 5. If $\sup_{k}c_{k}=+\infty$, we define $d_{cv}$ by the upper bound of
Trudinger-Moser inequality for normalized concentration sequence. From the
proof of Ruf \cite{ruf}, we know that $\tilde{S}\left(  \alpha\right) >d_{cv}$ for
any $\alpha\in \mathbb{R}$, thus $\tilde{S}\left(  \alpha\right)$ is attained for any
$\alpha$, thus $\tilde{S}\left(  \alpha\right)  >\max\left\{  d_{cv}%
,d_{nv}\right\}  $, when $4\pi-\alpha<8\pi^{2}B_{1}$.
\medskip

Step 6. Similar to the proof of Theorem \ref{nonatain}, we can show that there exists some
$\beta^{\ast\ast}$, such that when $4\pi-\alpha>\beta^{\ast\ast}$, $\tilde
{S}\left(  \alpha\right)  $ is not attained, by using estimates of the
constants
\[
\bar{B}_{k}=\underset{u\in H^{1}}{\sup}\frac{\int_{\mathbb{R}^2} u^{2k}dx}{\int_{\mathbb{R}^2}\left\vert
\nabla u\right\vert ^{2k}dx\int_{\mathbb{R}^2} u^{2}dx}.%
\]
Similarly, we also can show that if $\tilde{S}\left(  \alpha_{1}\right)  $  is
attained, then for any $\alpha_{1}<\alpha_{2}$, $\tilde{S}\left(  \alpha
_{2}\right)  $ is also attained as the argument of Theorem \ref{fina}. Hence we can conclude that
When $4\pi-\alpha<\beta^{\ast}$ then $\tilde{S}\left(  \alpha\right)< \beta^{\ast}$ and $\tilde{S}(\alpha)$ could be attained, while when $4\pi-\alpha>\beta^{\ast}$, $\tilde{S}\left(  \alpha\right)=4\pi-\alpha$, and $\tilde{S}(\alpha)$ is not attained, where $\beta^{*}$
is defined as \[
\beta^{\ast}=\sup\left\{  \left.  \left(  4\pi-\alpha\right)
\right\vert \tilde{S}\left(  \alpha\right)  \text{ is attained}\right\}
\] and $\beta^{*}\geq \frac{(4\pi)^2B_{1}}{2}>4\pi$.

\section{Appendix}

\subsection{Estimates for sharp constants of Gagliardo-Nirenberg inequalities}

\begin{lemma}
\label{lem3} For any fixed $a,b,M,N>0$, let $h(s):=s^{a}M+s^{-b}N$, then we
have
\[
\inf_{s>0}h(s)=h((\frac{bN}{aM})^{\frac{1}{a+b}})=\frac{a+b}{a}(\frac{b}%
{a})^{\frac{-b}{a+b}}M^{\frac{b}{a+b}}N^{\frac{a}{a+b}}.
\]

\end{lemma}

\begin{proof}
[Proof of Lemma \ref{sharp1}]Define
\[
C_{j}=\sup_{u\in H^{2}(\mathbb{R}^{4})}\frac{\left(  \int_{\mathbb{R}^{4}%
}|u|^{2j}dx\right)  ^{\frac{1}{j}}}{\int_{\mathbb{R}^{4}}\left(  |\Delta
u|^{2}+|u|^{2}\right)  dx}.
\]
First, we claim that $B_{j}=\frac{C_{j}^{j}j^{j}}{(j-1)^{j-1}}$. \ Set
\[
I(u):=\int_{\mathbb{R}^{4}}\left(  |\Delta u|^{2}+|u|^{2}\right)  dx
\]
and
\[
\Lambda:=\{u\in H^{2}(\mathbb{R}^{4}),\int_{\mathbb{R}^{4}}|u|^{2j}dx=1\}.
\]
For any $\tau>0$, we set $$u_{\tau}:=\tau^{\frac
{2}{j}}u(\tau x). $$Direct calculations lead to
\[
\int_{\mathbb{R}^{4}}|u_{\tau}|^{2j}dx=\int_{\mathbb{R}^{4}}|u|^{2j}dx
\]
and
\[
\int_{\mathbb{R}^{4}}|\Delta u_{\tau}|^{2}dx=\tau^{\frac{4}{j}}\int
_{\mathbb{R}^{4}}|\Delta u|^{2}dx,\ \int_{\mathbb{R}^{4}}|u_{\tau}|^{2}%
dx=\tau^{\frac{4}{j}-4}\int_{\mathbb{R}^{4}}|u|^{2}dx.
\]
It follows from Lemma \ref{lem3} that
\[%
\begin{split}
\inf_{u\in\Lambda}I(u)  &  =\inf_{u\in\Lambda}\inf_{\tau>0}I(u_{\tau})\\
&  =\inf_{u\in\Lambda}\inf_{\tau>0}\tau^{\frac{4}{j}}\int_{\mathbb{R}^{4}%
}|\Delta u|^{2}dx+\tau^{\frac{4}{j}-4}\int_{\mathbb{R}^{4}}|u|^{2}dx\\
&  =j(j-1)^{\frac{1-j}{j}}\left(  \int_{\mathbb{R}^{4}}|\Delta u|^{2}%
dx\right)  ^{1-\frac{1}{j}}\left(  \int_{\mathbb{R}^{4}}|u|^{2}dx\right)
^{\frac{1}{j}}.
\end{split}
\]
According to the definitions of $B_{j}$ and $C_{j}$, we derive that
$B_{j}=\frac{C_{j}^{j}j^{j}}{(j-1)^{j-1}}$.

Next, we turn to the estimate of sharp constants $C_{j}$. Through Fourier
transform and duality, it is easy to check that
\[
\big(\int_{\mathbb{R}^{4}}|u|^{2j}dx\big)^{\frac{1}{j}}\leq C_{j}%
\int_{\mathbb{R}^{4}}\left(  |\Delta u|^{2}+|u|^{2}\right)  dx
\]
is equivalent to
\[
\big(\int_{\mathbb{R}^{4}}|G\ast u|^{2j}|dx\big)^{\frac{1}{2j}}\leq
C_{j}^{\frac{1}{2}}\big(\int_{\mathbb{R}^{4}}|u|^{2}dx\big)^{\frac{1}{2}}%
\]
and
\[
\big(\int_{\mathbb{R}^{4}}|G\ast u|^{2}|dx\big)^{\frac{1}{2}}\leq C_{j}%
^{\frac{1}{2}}\big(\int_{\mathbb{R}^{4}}|u|^{\frac{2j}{2j-1}}dx\big)^{\frac
{2j-1}{2j}},
\]
where $\hat{G}(\xi)=(16\pi^4|\xi|^4+1)^{-\frac{1}{2}}$.
Then it follows that
\begin{equation}%
\begin{split}
\int_{\mathbb{R}^{4}}|(G\ast G\ast u)u|dx  &  \leq\big(\int_{\mathbb{R}^{4}%
}|G\ast G\ast u|^{2j}dx\big)^{\frac{1}{2j}}\big(\int_{\mathbb{R}^{4}%
}|u|^{\frac{2j}{2j-1}}\big)^{\frac{2j-1}{2j}}\\
&  \leq C_{j}^{\frac{1}{2}}\big(\int_{\mathbb{R}^{4}}|G\ast u|^{2}%
dx\big)^{\frac{1}{2}}\big(\int_{\mathbb{R}^{4}}|u|^{\frac{2j}{2j-1}%
}\big)^{\frac{2j-1}{2j}}\\
&  \leq C_{j}\big(\int_{\mathbb{R}^{4}}|u|^{\frac{2j}{2j-1}}\big)^{\frac
{2j-1}{j}}.
\end{split}
\label{ap1}%
\end{equation}

On the other hand, for any $t>1$, define $A_t=\Big(\frac{t^{\frac{1}{t}}}{t'^{\frac{1}{t'}}}\Big)^{\frac{1}{2}}$, where $t'$  satisfies $\frac{1}{t'}+\frac{1}{t}=1 $. Then it follows from sharp convolution Young inequality and Haousdorff
Young inequality that
\begin{equation}%
\begin{split}
&  \int_{\mathbb{R}^{4}}|(G\ast G\ast u)udx\\
&  \leq\big(\int_{\mathbb{R}^{4}}|G\ast G\ast u|^{2j}dx\big)^{\frac{1}{2j}%
}\big(\int_{\mathbb{R}^{4}}|u|^{\frac{2j}{2j-1}}\big)^{\frac{2j-1}{2j}}\\
&  \leq\big(\frac{A_{\frac{2j}{2j-1}}A_{j}}{A_{2j}}\big)^{4}\big(\int
_{\mathbb{R}^{4}}|G\ast G|^{j}dx\big)^{\frac{1}{j}}\big(\int_{\mathbb{R}^{4}%
}|u|^{\frac{2j}{2j-1}}\big)^{\frac{2j-1}{j}}\\
&  \leq\big(\frac{A_{\frac{2j}{2j-1}}A_{j}}{A_{2j}}\big)^{4}\big(j^{\prime}%
{}^{\frac{1}{j^{\prime}}}j^{-\frac{1}{j}}\big)^{\frac{1}{2}}\big(\int
_{\mathbb{R}^{4}}|\hat{G}\hat{G}|^{j^{\prime}}dx\big)^{\frac{1}{j^{\prime}}%
}\big(\int_{\mathbb{R}^{4}}|u|^{\frac{2j}{2j-1}}\big)^{\frac{2j-1}{j}},
\end{split}
\label{ap2}%
\end{equation}
which together with (\ref{ap1}) yields that
\[
C_{j}\leq\big(\frac{A_{\frac{2j}{2j-1}}A_{j}}{A_{2j}}\big)^{4}\big(j^{\prime
}{}^{\frac{1}{j^{\prime}}}j^{-\frac{1}{j}}\big)^{\frac{1}{2}}\big(\int
_{\mathbb{R}^{4}}|\hat{G}\hat{G}|^{j^{\prime}}dx\big)^{\frac{1}{j^{\prime}}},
\]
and the proof is finished.
\end{proof}

\subsection{Some useful results}

\begin{proposition}
\label{split}Let $\Omega\subseteq%
%TCIMACRO{\U{211d} }%
%BeginExpansion
\mathbb{R}
%EndExpansion
^{4}$ be a bounded open domain with Lipschitz boundary. Then for any $u\in
H^{2}\left(  \Omega\right)  ,\omega\in H^{4}\left(  \Omega\right)  $, we have%
\[
\int_{\Omega}\Delta u\cdot\Delta\omega dx=\int_{\Omega}u\cdot\Delta^{2}\omega
dx-\int_{\partial\Omega}v\cdot u\Delta^{\frac{3}{2}}\omega dx+\int
_{\partial\Omega}v\cdot\Delta^{\frac{1}{2}}u\Delta\omega dx
\]
where $v$ denotes the outer normal to $\partial\Omega$.
\end{proposition}

\begin{lemma}
[Pizzetti \cite{Pizz}]\label{Pizz} Let $u\in C^{2m}(B_{R}(x_{0}))$,
$B_{R}(x_{0})\subset\mathbb{R}^{n}$, for some $m$, $n$ positive integers. Then
there are positive constants $c_{i}=c_{i}(n)$ such that
\[
\int_{B_{R}(x_{0})}u(x)dx=\sum_{i=0}^{m-1}c_{i}R^{n+2i}\Delta^{i}%
u(x_{0})+c_{m}R^{n+2m}\Delta^{m}u(\xi),
\]
for some $\xi\in B_{R}(x_{0})$.
\end{lemma}

\begin{lemma}
[Martinazzi \cite{Mar}]\label{Pizz1} Let $u$ solve $\Delta^{m}u=f\in L(\log
L)^{\alpha}$ in smooth bounded $\Omega$ with the Dirichlet boundary condition
for some $0\leq\alpha\leq1$ and $n\geq2m$. Then
\[
\nabla^{2m-l}u\in L^{(\frac{n}{n-l},\frac{1}{\alpha})}(\Omega),1\leq
l\leq2m-1
\]
and
\[
\Vert\nabla^{2m-l}u\Vert_{(\frac{n}{n-l},\frac{1}{\alpha})}\leq C\Vert
f\Vert_{L(logL)^{\alpha}}.
\]

\end{lemma}

\begin{lemma}
[\cite{Gazzola}]\label{whole estimate} Let $\Omega\subset%
%TCIMACRO{\U{211d} }%
%BeginExpansion
\mathbb{R}
%EndExpansion
^{N}$ be a bounded open set with smooth boundary, and take \thinspace$k,m\in%
%TCIMACRO{\U{2115} }%
%BeginExpansion
\mathbb{N}
%EndExpansion
$,$k\geq2m$, and $\gamma\in\left(  0,1\right)  $. If $u\in H^{m}\left(
\Omega\right)  $ is a weak solutions of the problem%
\[%
%TCIMACRO{\QDATOPD{\{}{.}{\left(  -\Delta\right)  ^{m}u=f\text{ in }\Omega
%}{\partial_{v}^{i}u=h_{i}\text{ on }\partial\Omega,\text{ }0\leq i\leq m-1}}%
%BeginExpansion
\genfrac{\{}{.}{0pt}{0}{\left(  -\Delta\right)  ^{m}u=f\text{ in }%
\Omega}{\partial_{v}^{i}u=h_{i}\text{ on }\partial\Omega,\text{ }0\leq i\leq
m-1}%
%EndExpansion
\]
with $f\in C^{k-2m,\gamma}\left(  \Omega\right)  $ and $h_{i}\in
C^{k-i,\gamma}\left(  \partial\Omega\right)  $, then $u\in C^{k,\gamma}\left(
\Omega\right)  $ and there exists a constant $c=c\left(  \Omega,k,\gamma
\right)  $ such that
\[
\left\Vert u\right\Vert _{C^{k,\gamma}\left(  \Omega\right)  }\leq c\left(
\left\Vert f\right\Vert _{C^{k-2m,\gamma}\left(  \Omega\right)  }%
+\underset{i=0}{\overset{m-1}{\sum}}\left\Vert h_{i}\right\Vert
_{C^{k-i,\gamma}\left(  \partial\Omega\right)  }\right)  .
\]
Similarly, If $f\in C^{k-2m,\gamma}(\Omega)$ and $u$ is a weak solution of
$(-\Delta)^{m}u=f$ in $\Omega$, then $u\in C_{loc}^{k,\gamma}(\Omega)$, and
for any open set $V\Subset\Omega$, then there exists a constant
$C=C(k,p,V,\Omega)$ such that
\[
\Vert u\Vert_{C^{k,\gamma}(V)}\leq C(\Vert f\Vert_{C^{k-2m,\gamma}(\Omega)}+\Vert
u\Vert_{L^{1}(\Omega)}).
\]

\end{lemma}

\begin{lemma}
[\cite{Gazzola}]\label{local estimate} Let $\Omega\in\mathbb{R}^{n}$ be a
bounded open set with smooth boundary and take $m$, $k\in\mathbb{N}$,
$k\geq2m$, $p>1$. If $f\in W^{k-2m,p}(\Omega)$ and $u\in H^{m}(\Omega)$ is a
weak solution of $(-\Delta)^{m}u=f$ in $\Omega$, then $u\in W_{loc}%
^{k,p}(\Omega)$, and for any open set $V\subset\subset\Omega$, then there
exists a constant $C=C(k,p,V,\Omega)$ such that
\[
\Vert u\Vert_{W^{k,p}(V)}\leq C(\Vert f\Vert_{W^{k-2m,p}(\Omega)}+\Vert
u\Vert_{L^{1}(\Omega)}).
\]
Similarly, If $f\in C^{k-2m,\gamma}(\Omega)$ and $u$ is a weak solution of
$(-\Delta)^{m}u=f$ in $\Omega$, then $u\in C_{loc}^{k,\gamma}(\Omega)$, and
for any open set $V\subset\subset\Omega$, then there exists a constant
$C=C(k,p,V,\Omega)$ such that
\[
\Vert u\Vert_{C^{k,\gamma}(V)}\leq C(\Vert f\Vert_{C^{k-2m,\gamma}(\Omega)}+\Vert
u\Vert_{L^{1}(\Omega)}).
\]

\end{lemma}
\begin{lemma}[\cite{Mar}]\label{Liouville}
Suppose that $u$ satisfies the bi-harmonic equation $(-\Delta)^2u=0$ with $u(x)\lesssim(1+|x|^{l})$ for some $l\geq 0$. Then
$u$ is a polynomial of degree at most $\max\{l,2\}$.
\end{lemma}

\end{document}